\documentclass[leqno]{amsart}
\usepackage{amssymb,amsmath,times}
\usepackage{xcolor}
\usepackage[mathscr]{eucal}
\usepackage{xspace}
\usepackage{stix}
\usepackage[colorlinks = false]{hyperref}

\newcommand{\R}{\mathbb{R}}

\newcommand{\N}{\mathbb{N}}

\newcommand{\CS}{\mathcal{S}}

\newcommand{\F}{{\mathcal F}}

\newcommand{\po}{\partial}

\newcommand{\wto}{\rightharpoonup} 
\newcommand{\ve}{\varepsilon}

\newcommand{\la}{\left\langle}
\newcommand{\ra}{\right\rangle}

\newcommand{\loc}{{\text{\rm loc}}}
\newcommand{\X}{\times}

\renewcommand{\d}{\delta}
\renewcommand{\l}{\lambda}

\renewcommand{\a}{\alpha}
\renewcommand{\b}{\beta}

\newcommand{\s}{\sigma}
\newcommand{\g}{\gamma} 
\newcommand{\z}{\zeta}
\renewcommand{\k}{\kappa}
\newcommand{\sgn}{\operatorname{sgn}\,}

\newcommand{\Om}{\Omega}
\newcommand{\om}{\omega}

\newcommand{\supp}{\text{\rm supp}\,}

\renewcommand{\div}{\text{\rm div}\,}

\renewcommand{\supp}{\text{\rm supp}\,}

\newcommand{\A}{{\mathcal A}}

\newcommand{\AP}{\operatorname{AP}}

\newcommand{\WAP}{\operatorname{WAP}}

\newcommand{\APs}{\operatorname{{\mathcal W}^*\! AP}}
\newcommand{\FS}{\operatorname{FS}}

\newcommand{\cA}{{\mathcal A}}
\newcommand{\cK}{{\mathcal K}}
\newcommand{\cN}{{\mathcal N}}

\newcommand{\esslim}{\operatorname{ess}\, \lim}

\newcommand{\Lip}{\text{\rm Lip}}

\renewcommand{\S}{{\mathcal S}}
\newcommand{\B}{{\mathcal B}}
\newcommand{\NN}{\mathcal{N}}
\newcommand{\BUC}{\operatorname{BUC}}

\newcommand{\K}{\mathcal{K}}

\renewcommand{\Lip}{\text{Lip\,}}

\newcommand{\cP}{{\mathcal P}}

\newcommand{\bbP}{{\mathbb P}}

\newcommand{\bbB}{{\mathbb B}}

\newcommand{\mm}{{\mathfrak m}}

\newcommand{\medint}{{\mbox{\vrule height3.5pt depth-2.8pt
          width4pt}\mkern-13mu\int\nolimits}}
\newcommand{\Medint}{\mkern12mu\mbox{\vrule height4pt
         depth-3.2pt
          width5pt}\mkern-16.5mu\int\nolimits}

\newcommand{\pt}{\partial_t}
\newcommand{\Grad}{\nabla}
\DeclareMathOperator{\Div}{div}
\newcommand{\dotW}{\dot{W}}
\newcommand{\dW}{\,dW}

\newcommand{\dx}{\,dx}

\newcommand{\cS}{\mathcal{S}}
\newcommand{\cF}{\mathcal{F}}
\newcommand{\seq}[1]{\left\{#1\right\}}
\newcommand{\set}[1]{\left\{#1\right\}}

\newcommand{\abs}[1]{\left | #1 \right |}
\newcommand{\norm}[1]{\left \| #1 \right \|}
\newcommand{\fU}{\mathfrak{U}}
\newcommand{\D}{\mathcal{D}}
\newcommand{\Dp}{\mathcal{D}^\prime}
\DeclareMathOperator*{\EE}{\mathbb{E}}
\newcommand{\En}{\mathbb{I}}
\newcommand{\chip}{\chi_+}
\newcommand{\chim}{\chi_-}
\newcommand{\chipm}{\chi_\pm}
\newcommand{\dxi}{\,d\xi}
\newcommand{\pxi}{\partial_\xi}
\newcommand{\radon}{\mathcal{M}}
\DeclareMathOperator*{\esup}{ess\,sup}

\newcommand{\action}[2]{\left\langle #1, #2 \right\rangle}
\newcommand{\eps}{\varepsilon}
\newcommand{\liploc}{\text{Lip}_{\text{loc}}}

\newcommand{\CT}{\mathcal{T}}

\theoremstyle{plain}
\newtheorem{theorem}{Theorem}[section]

\newtheorem{lemma}{Lemma}[section]
\newtheorem{proposition}{Proposition}[section]
\theoremstyle{definition}
\newtheorem{definition}{Definition}[section]
\theoremstyle{remark}

\newtheorem{remark}{Remark}[section]
\numberwithin{equation}{section}

\begin{document}

\title[Homogenization of Stochastic Conservation Laws]
{Homogenization of 
Stochastic Conservation Laws with Multiplicative Noise}

\author{Hermano Frid} 
\address{Instituto de Matem\'atica Pura e Aplicada - IMPA \\
Estrada Dona Castorina, 110 \\
Rio de Janeiro, RJ 22460-320, Brazil}
\thanks{H.~Frid gratefully acknowledges the support from CNPq, 
through grant proc.\ 305097/2019-9, and FAPERJ, 
through grant proc.\ E-26/202.900-2017}
\email{hermano@impa.br}

\author{Kenneth H.\ Karlsen} 
\thanks{K. H.~Karlsen gratefully acknowledges 
the support from the Research Council of Norway 
through the project Stochastic Conservation Laws (250674)}
\address{Department of Mathematics, University of Oslo, 
P.O.\ Box 1053, N-0316 Oslo, Norway}
\email{kennethk@math.uio.no}

\author{Daniel Marroquin}
\thanks{D.~Marroquin thankfully 
acknowledges the support from CNPq, through grant proc. 150118/2018-0.}
\address{Instituto de Matem\'{a}tica - Universidade Federal do Rio de Janeiro\\ Cidade Universit\'{a}ria, 21945-970, Rio de Janeiro, Brazil}
\email{marroquin@im.ufrj.br}


\date{\today}

\begin{abstract}
We consider the generalized almost periodic 
homogenization problem for two 
different types of stochastic conservation 
laws with oscillatory coefficients and multiplicative noise. 
In both cases the stochastic perturbations are such that the equation admits special 
stochastic solutions which play the role of
the steady-state solutions in the deterministic case. 
Specially in the second type, these stochastic solutions are crucial elements 
in the homogenization analysis. Our homogenization method is based on the notion of stochastic two-scale
Young measure, whose existence is established here.

\end{abstract}

\maketitle 

\tableofcontents

\section{Introduction}\label{S:1}

We consider two very representative homogenization problems 
for conservation laws subjected to a 
stochastic perturbation by a multiplicative noise. 

The first problem we consider is the one of the 
nonlinear transport equation whose deterministic 
case was first addressed in \cite{E}, in the periodic case, and 
later on in \cite{AF, FS2} in the almost periodic, 
Fourier-Stieltjes algebras cases, respectively. 
See also \cite{Da,JS}. The equation is the following  
\begin{equation}\label{e2.1}
du^\ve+a\left(\frac{x}{\ve}\right) 
\cdot \nabla_x f(u^\ve)\,dt 
= \k_0\,\s(u^\ve)\,dW +\frac12 \k_0^2\,h(u^\ve)\,dt,
\end{equation}
where $W$ is a scalar Brownian motion, $dW$ denotes It\^o differential, $a(y)\in \left(\Lip\cap \A\right)(\R^d)^d$ 
satisfies $\nabla_y\cdot a(y)=0$, $\A(\R^d)$ is a 
general ergodic algebra, a concept whose 
definition we recall subsequently, $f,\s,h:\R\to\R$ are
smooth functions, with $\s$ and $h$ satisfying $h=\s'\s$.  
We also assume  that $f', \s',h' \in L^\infty(\R)$,  and $\s\ge \d_0 >0$. 
We further assume that the set of zeros of $f'$ has measure zero, namely, $|\{u\in\R\,:\, f'(u)=0\}|=0$.

Note that by the well-known conversion formula 
between Stratonovich and It\^o differentials (see, e.g., \cite{Ar}) equation \eqref{e2.1} may be written as
\begin{equation*}
du^\ve+a\left(\frac{x}{\ve}\right) 
\cdot \nabla_x f(u^\ve)\,dt 
= \k_0\,\s(u^\ve)\,\circ dW,
\end{equation*}
where $\circ dW$ denotes integration in the Stratonovich sense.

The initial condition  is given by
\begin{equation}\label{e2.2}
u^\ve(0,x)=U_0\left(x,\frac{x}{\ve}\right),
\end{equation}
where $U_0(x,y)\in L^\infty(\R^d;\A(\R^d))$. Although we study the homogenization problems here in the general
context of ergodic algebras, the results established 
in this paper are new even in the context of periodic homogenization.
So, the reader not familiarized with the concept of 
ergodic algebras may, in a first reading, 
just assume the periodic case.  

The concept of ergodic algebra was introduced 
in \cite{ZK} (see also \cite{JKO}), motivated by 
algebras generated by typical realizations of stationary ergodic 
processes and their self-averaging property provided by 
Birkhoff theorem. Namely, an ergodic algebra is an 
algebra $\A(\R^d)$ of bounded uniformly continuous (BUC)  
functions in $\R^d$ satisfying the following: (i)  $\A(\R^d)$ is invariant by 
translations, that is, if $f\in\A$, then $f(\cdot+\l)\in\A$, 
for all $\l\in\R^d$;  (ii) every function $f\in\A(\R^d)$ 
possesses mean-value, that is, there exists a number $M(f)$ 
such that $f(\ve^{-1}x) \wto M(f)$ as $\ve\to0$ in the 
weak--$\star$ topology of $L^\infty(\R^d)$.  In particular, we have
$$
M(f):=\lim_{R\to\infty} \frac{1}{|B(0;R)|}\int_{B(0;R)}f(x)\,dx,
$$
where $B(0;R)$ is the open ball with radius $R$ centered at 
the origin 0, and $|B(0,R)|$ is its $n$-dimensional Lebesgue measure. 
Also, one easily sees that $M(f(\cdot+\l))=M(f)$, 
for all $\l\in\R^n$. We also use the notation $M(f)=\medint f\,dx$ ; 
(iii) $\A$ is ergodic in the sense that if we define in $\A$ 
the semi-norm $[f]_2:= M(|f|^2)^{1/2}$, taking equivalence 
classes by the relation $f\sim g \iff [f-g]_2=0$, and 
denoting the completion of the quotient space by $
\B^2(\R^n)$, the  Besicovitch space of exponent 2 
associated with $\A(\R^d)$, we have that any $ g\in\B^2(\R^d)$, 
satisfying $g(\cdot+\l)=g(\cdot)$, in the sense of $\B^2(\R^d)$, 
for all $\l\in\R^d$, is equal to a constant in $\B^2(\R^d)$.  
As examples of ergodic algebras, besides the periodic functions, we 
have $\AP(\R^d)$, the space of almost 
periodic functions (see, e.g., \cite{Be}), the 
Fourier-Stieltjes algebra $\FS(\R^d)$ 
(see, e.g.,  \cite{Eb1,FS2}), or the larger one $\WAP(\R^d)$, 
the space of the weak almost periodic functions, see \cite{Eb1,Eb2}.
In particular, in \cite{Eb2}, Eberlein 
proved that every function $\phi\in\WAP(\R^d)$ 
admits a decomposition $\phi=\phi_*+\phi_{\NN}$, where 
$\phi_*\in\AP(\R^d)$ and $\phi_\NN\in\NN(\R^d)$ where
\begin{equation*}
\NN(\R^d):= \{f\in \BUC(\R^d)\, :\, 
\lim_{R\to\infty}\frac{1}{|B(0;R)|}
\int_{B(0;R)}|f(y)|\,dy=0\}.
\end{equation*}
This motivates the introduction in \cite{F} of 
the algebra of the weak--$*$
almost periodic functions, $\APs(\R^d)$, defined by
$$
\APs(\R^d):=\AP(\R^d)+\NN(\R^d),
$$ 
which is clearly an ergodic algebra and 
contains all the ergodic algebras containing 
the periodic functions so far known.   

In all that follows, we assume that the ergodic algebra $\A(\R^d)$ is a subalgebra of  $\APs(\R^d)$, that is, $\A(\R^d)\subset\APs(\R^d)$. 
 
Let $\B^2(\R^d)$ denote the $L^2$-Besicovitch 
space associated with $\A(\R^d)$. Set 
\begin{equation*}
	\CT:=\{ v\in\A(\R^n)\cap W^{1,\infty}(\R^n)\,:\, 
	\nabla_a v:= a\cdot\nabla v\in \A(\R^n)\}.
\end{equation*}
We define
\begin{equation*}
	\S :=\biggl\{ v\in
	\B^2(\R^d)\,:\, \Medint_{\R^d} v(y) 
	a(y)\cdot\nabla\varphi(y)\,dy=0,\ 
	\text{ for all  $\varphi\in \CT$}  \biggr\}
\end{equation*}
and its subspaces
\begin{equation*}
\S^*:=\bigl\{ v\in\A(\R^d)\cap W^{1,\infty}(\R^d)\,:\, 
 \nabla_a v=0,\ \text{a.e.}\,\bigr\},
\end{equation*}
and
\begin{equation*}
	\CS^\dag:=\left\{v\in\CS\,:\,\exists\, 
	(v_k)_{k\in\N}\subset\CT, \; 	
	v_k\overset{\B^2\cap L_\loc^2}{\longrightarrow} v\ 
	\text{and}\ \nabla_a 
	v_k\overset{\B^2\cap L_\loc^2}{\longrightarrow}0\right\}.
\end{equation*}
In the periodic case we have $\S^\dag=\S$, as proven in \cite{Da} by applying the commutation lemma in \cite{DiPerna:1989aa}. In general, it holds $\S^*\subset\S^\dag$. 
In \cite{AF} it was shown that for a large collection of 
fields $a(y)\in \left(\AP\cap \Lip\right)(\R^d;\R^d)$, with $\div a=0$, 
the space $\S^*$ is dense in $\S$ in the $\B^2(\R^d)$ topology, when 
$\A(\R^d)=\AP(\R^d)$.  Similarly, in \cite{FS2} also 
a large collection of fields $a(y)\in 
\left(\FS\cap\Lip\right)(\R^d;\R^d)$, with $\div a=0$, 
was described for which the space $\S^*$ 
is dense in $\S$ in the topology of $\B^2(\R^d)$, 
when $\A(\R^d)=\FS(\R^d)$.  Finally, in \cite{JS}, 
it was shown that for any 
$a(y)\in \left(\AP\cap \Lip\right)(\R^d;\R^d)$, 
$\S^\dag$ is dense in $\S$, in 
the topology of $\B^2(\R^d)$, for $\A(\R^d)=\AP(\R^d)$. 
  
We assume that
\begin{equation}\label{e2.7'}
\text{$U_0\in L^\infty(\R^d;\A(\R^d))$,  $U_0(x,\cdot )\in\S$ 
for a.e.\ $x\in\R^d$}.
\end{equation}  

Let $\cK$ be the compactification of $\R^d$ associated with the ergodic 
algebra $\A(\R^d)$, through a classical theorem 
by Stone (see, e.g., \cite{DS,DS2}). For each $y\in\cK$, consider 
the following auxiliary initial value problem 
\begin{align}\label{e2.22}
&dU+ \nabla_x\cdot(\tilde a(y) f(U))\,dt=\k_0\,\s(U)\,dW
+\frac12 \k_0^2\, h(U)\, dt, \quad t>0,\ x\in\R^d,\\
& U(0,x,y)=U_0(x,y), \quad  x\in\R^d, \label{e2.23}
\end{align}
where $\tilde a(y)$ is the orthogonal projection 
of $a(y)$ onto $\S$ in $\B^2(\R^d)$. 
In particular, $\tilde a$ is a Borel function over $\cK$. 
Actually, it has been proven in \cite{JS} (see  
Theorem~3.2 in \cite{JS}) that $\tilde a(y)\in C(\cK)\sim\A(\R^d)$; 
we will not make use of this fact here. The stability properties 
of solutions of the Cauchy problem for 
stochastic scalar conservation laws imply 
that $U\in L^2(\Om;L^\infty((0,T)\X\R^d\X\cK))$, for any $T>0$;
we will comment further on this point in Section~\ref{S:2}. 

Let  $(\Om,\F,\bbP)$ be a probability space, 
$\{\F_t\,:\, 0\le t\le T\}$ be a complete filtration, 
that is, an increasing family of $\s$-algebras contained in $\F$, 
all of them containing all the null sets of $\F$, such 
that $\F_s=\bigcap_{t\ge s}\F_t$. In this paper, for simplicity, we assume
that the $\s$-algebra $\F$ is countably  generated and  $\F_t$  is the filtration generated by the Brownian motion 
$\{W(s)\,:\,0\le s\le t\}$ and $\F_0$, the $\s$-algebra generated by the null sets of $\F$.  
 
If  $X$ is a Banach space,   
let $\cN_W^2(0,T,X)$ denote the space of the predictable 
$X$-valued processes (see, e.g., \cite{DPZ}, p.94, \cite{PR}, p.28).  
This is the same as the space $L^2([0,T]\X\Om,X)$ with 
the product measure $dt\otimes \,d\bbP$ on $\cP_T$, the 
predictable $\s$-algebra, i.e., the $\s$-algebra 
generated by the sets $\{0\}\X\F_0$ and the 
rectangles $(s,t]\X A$ for $A\in\F_s$. 
We denote $\cN_W^2(0,T, L_\loc^2(\R^d)) 
:=\bigcap_{R>0} \cN_W^2(0,T, L^2(B(0,R)))$, 
where $B(0,R)$ is the open ball centered at 0 with 
radius $R$ in $\R^d$. We will say that $u$ is 
predictable if $u\in \cN_W^2(0,T, L_\loc^2(\R^d))$.  
Let us also denote $Q=(0,T)\X\R^d$.

\begin{definition}\label{D:2.1} We say that a predictable 
function $u^\ve \in L^2(\Om; L^\infty(Q))$ 
is an entropy solution of \eqref{e2.1}--\eqref{e2.2} 
if for all convex  $\eta\in C^2(\R)$, for 
$q\in C^2(\R)$, such that $q'(u)=\eta'(u)  f'(u)$, 
and for all  $0\le \varphi \in C_c^\infty((-\infty,T)\X\R^d)$, 
a.s.\ in $\Om$, we have
\begin{multline*}
\int_Q \eta(u^\ve)\po_t\varphi
+q(u^\ve) a\left(\frac{x}{\ve}\right)\cdot\nabla\varphi 
+  \frac{\k_0^2}2 \left( \eta'(u^\ve) h(u^\ve) 
+ \eta''(u^\ve)  \s^2( u^\ve) \right) \varphi\,dx\,dt\\ 
+\k_0\int_0^T\int_{\R^d}\eta'(u^\ve)\s(u^\ve)\varphi\,dx\,dW(t)
+\int_{\R^d}\eta\left(U_0\left(x,\frac{x}{\ve}\right)
\right)\varphi(0,x)\,dx\,dt\ge 0.
\end{multline*}
\end{definition}

\begin{definition}\label{D:2.2} For each $y\in\cK$, we say 
that a predictable function $U(y) \in L^2(\Om; L^\infty(Q))$ 
is an entropy solution of \eqref{e2.22}--\eqref{e2.23} 
if for all convex  $\eta\in C^2(\R)$, for 
$q\in C^2(\R)$, such that $q'(u)=\eta'(u)  f'(u)$, 
and for all  $0\le \varphi \in C_c^\infty((-\infty,T)\X\R^d)$, 
a.s.\ in $\Om$, we have
\begin{multline}\label{e2.25_0}
\int_Q \eta(U(y))\po_t\varphi 
+  q(U(y)) \tilde a(y)\cdot\nabla\varphi 
+  \frac{\k_0^2}2 \left( \eta'(U(y)) h(U(y)) 
+ \eta''(U(y))  \s^2( U(y)) \right) \varphi\,dx\,dt
\\ + \k_0\int_0^T\int_{\R^d}\eta'(U(y))
\s(U(y))\varphi\,dx\,dW(t) 
+\int_{\R^d}\eta(U_0(y))\varphi(0,x)\,dx\,dt\ge 0.
\end{multline}
\end{definition}

\begin{theorem}\label{T:1.2} 
Let $u^\ve$ be the entropy solution of \eqref{e2.1}--\eqref{e2.2}, 
with $U_0$ satisfying \eqref{e2.7'}, and, for 
each $y\in\K$, let $U(y)$ be the 
entropy solution \eqref{e2.22}--\eqref{e2.23}.
Assume that $\S^\dag$ is dense in $\S$ in the 
topology of $\B^2(\R^d)$. Then, we have that $u^\ve\wto u$, in the weak
topology of $L^2(\Om;L_\loc^2(Q))$, that is, $L^2(\Om;L^2((0,T)\X\{|x|<R\})$, for any $R>0$,  where 
\begin{equation*}
	u(t,x)=\int_\K U(t,x,y)\,d\mm(y),
\end{equation*}
and $d\mm(y)$ is the measure on $\K$ 
induced by the mean value on $\A(\R^d)$.   Moreover, if $U\in L^2(\Om;\B^2(\R^d;C_b([0,T]\X\R^d)))$, then $u^\ve(t,x)-U\left(t,x,\frac{x}{\ve}\right)$ strongly converges to zero in $L^2(\Om; L_\loc^2(Q))$.
\end{theorem}

\bigskip
The second problem is the one of a 
stiff oscillatory external force whose 
deterministic case was first addressed in \cite{ES}, in 
the periodic one-dimensional case and later on 
in \cite{AF, AFS} in the almost periodic and ergodic algebras 
multidimensional case. The corresponding equation is as follows
\begin{equation}\label{e1.1}
du^\ve+ \nabla_{x}\cdot f(u^\ve)\,dt
=\frac1{\ve}V'\left(\frac{x_1}{\ve}\right)\,dt
+ \k_0\,\s_{f_1}(u^\ve)\,dW+ \frac12\k_0^2\,h_{f_1}(u^\ve)\,dt,
\end{equation}
where $f=(f_1,\ldots,f_d)$, $f_i:\R\to\R$ are smooth functions, 
$i=1,\ldots,d$, $f_1'\ge \delta_0 >0$, $f_k'\ge 0$, $k=2,\ldots,d$,  
We also assume that
$f'\in L^\infty(\R;\R^d)$ and $f_1',f_1'',f_1'''\in L^\infty(\R)$.
$\k_0\in\R$ is a constant. $V:\R\to\R$ is a smooth 
function belonging to an arbitrary ergodic 
algebra $\A(\R)$, $W:\Om\X[0,T]\to\R$ 
is a standard Brownian motion, 
and $\s_{f_1}$, $h_{f_1}$ are obtained 
from $f_1$ from the expressions
$$
\s_{f_1}(u):= \frac1{f_1'(u)},
\qquad h_{f_1}:=-\frac{f_1''(u)}{f_1'(u)^3}.
$$  
We observe that, from the assumptions on $f_1$,  
it follows that $h_{f_1}'\in L^\infty(\R)$.

Again, in view of the Stratonovich-It\^o 
conversion formula, we note that 
equation \eqref{e1.1} may be written as
\begin{equation*}
du^\ve+ \nabla_{x}\cdot f(u^\ve)\,dt
=\frac1{\ve}V'\left(\frac{x_1}{\ve}\right)\,dt
+ \k_0\,\s_{f_1}(u^\ve)\,\circ dW.
\end{equation*}

We prescribe an initial data for \eqref{e1.1} of the form
\begin{equation}\label{e1.2}
u^\ve(0,x)=u_0\left(x,\frac{x}{\ve}\right),
\end{equation}
which, for simplicity, we may assume to be deterministic, 
whose hypotheses we will specify later on. 

Let  $g=f_1^{-1}$ be the inverse of $f_1$. We assume that, for 
some $v_0\in L^\infty(\R^d)$,  $u_0(x,y)$ satisfies
\begin{equation}\label{e1.2'}
u_0(x,y)=g(V(y)+v_0(x)).
\end{equation}

Let us consider the auxiliary equation 
\begin{equation}\label{e1.4}
d\bar u+\nabla\cdot \bar f(\bar u)\,dt
=\k_0\,\s_{\bar f_1}(\bar u)\,dW 
+\frac12 \k_0^2\,h_{\bar f_1}( \bar u)\,dt,
\end{equation}
where $\bar f=(\bar f_1,\bar f_2,\ldots,\bar f_d)$ , with 
$\bar f_1,\bar f_2,\ldots, \bar f_d$, satisfying
\begin{align}
p &=\Medint_{\R}g\left(\bar f_1(p)
+V(z_1)\right)\,dz_1, \label{e1.5}\\
\bar f_k(p)&=\Medint_{\R}f_k\circ 
g\left(\bar f_1(p)+V(z_1)\right)\,dz_1,  
\quad k=2,\cdots, d, \label{e1.6}
\end{align}
and $\s_{\bar f_1}(\cdot), \ h_{\bar f_1}(\cdot)$ 
are defined as $\s_{f_1},\ h_{f_1}$ 
with $\bar f_1(\cdot)$ instead of $f_1$
We remark that, from the assumptions on $f$ 
and $f_1$, it follows from \eqref{e1.5} and \eqref{e1.6} 
that $\bar f$ and $\bar f_1$ also satisfy 
$\bar f'\in L^\infty(\R;\R^d)$ and 
$\bar f_1',\bar f_1'',\bar f_1'''\in L^\infty(\R)$.

For \eqref{e1.4} the following initial condition is prescribed
\begin{equation} \label{e1.4'}
\bar u(0,x)=\bar u_0(x)
:=\Medint_{\R}u_0(x,z_1)\,dz_1=\bar f_1^{-1}(v_0(x)).
\end{equation} 

\begin{definition}\label{D:1.1} 
We say that $u^\ve \in\cN_W^2(0,T, L_\loc^2(\R^d))
\cap L^2(\Om; L^\infty(Q))$ is an entropy solution 
of \eqref{e1.1}--\eqref{e1.2}, with 
$u_0\left(\cdot, \frac{\cdot}{\ve}\right)
\in L^2(\Om; L^\infty(\R^d))$, 
satisfying \eqref{e1.2'}, if for all convex 
$\eta\in C^2(\R)$, for $q\in C^2(\R,\R^d)$, 
such that $q'(u)=\eta'(u) f'(u)$, and for all 
$0\le \varphi \in C_c^\infty((-\infty,T)\X\R^d)$, 
a.s.\ in $\Om$, we have
\begin{multline*}
\int_Q \eta(u^\ve)\po_t\varphi + q(u^\ve)\cdot\nabla\varphi 
+\eta'(u^\ve)\left( \frac1{\ve}V'\left(\frac{x_1}{\ve}\right)
+\frac{\k_0^2}2 h_{f_1}(u^\ve) \right)\varphi \,dx\, dt \\
+\frac{\k_0^2}2 \int_Q  \s_{f_1}^2(u)\eta''(u)\varphi\,dx\,dt
+\k_0\int_0^T\int_{\R^d}\eta'(u)\s_{f_1}(u)\varphi\,dx\,dW(t) 
\\ + \int_{\R^d}\eta(u)\varphi(0,x)\,dx\,dt\ge 0.
\end{multline*}
\end{definition}  
  
\begin{definition}\label{D:1.2}  
We say that a predictable function $\bar u\in L^2(\Om; L^\infty(Q))$ 
is an entropy solution of \eqref{e1.4}--\eqref{e1.4'} 
if for all convex  $\eta\in C^2(\R)$, for 
$\bar q\in C^2(\R,\R^d)$, such that $\bar q'(u)=\eta'(u) \bar f'(u)$, 
and for all  $0\le \varphi \in C_c^\infty((-\infty,T)\X\R^d)$, 
a.s.\ in $\Om$, we have
\begin{multline*}
\int_Q \eta(\bar u)\po_t\varphi 
+ \bar q(\bar u)\cdot\nabla\varphi +
\frac{\k_0^2}2 \left( \eta'(\bar u) h_{\bar f_1}(\bar u) 
+ \eta''(\bar u)  \s_{\bar f_1}^2(\bar u) \right) \varphi\,dx\,dt\\ 
+\k_0\int_0^T\int_{\R^d}\eta'(\bar u)\s_{\bar f_1}(\bar u)\varphi\,dx\,dW(t) 
+\int_{\R^d}\eta(\bar u_0)\varphi(0,x)\,dx\,dt\ge 0.
\end{multline*}
\end{definition}

We can state our second main result.

\begin{theorem}\label{T:1.1}  
Let $u^\ve$ be the entropy solution of \eqref{e1.1}--\eqref{e1.2}, 
with $u_0$ satisfying \eqref{e1.2'}, and $\bar u$ 
be the entropy solution of \eqref{e1.4}--\eqref{e1.4'}.   
Then, $u^\ve\wto \bar u$ in the 
weak topology of $L^2(\Om;L_\loc^2(Q))$. 
Moreover, 
 $u^\ve(t,x)-U\left(t,x,\frac{x}{\ve}\right)$
strongly converges to zero in $L^2(\Om; L_\loc^2(Q))$, as $\ve\to0$, where 
$U(t,x,y)=g\left(\bar f_1(\bar u(t,x)) +V(y)\right)$. 

\end{theorem}

Before we make an account of earlier works connected 
to the present one, both in homogenization theory and 
in the theory of SPDEs, and a brief 
description of the contents in this paper, we remark for 
practical purposes that the stochastic perturbation of 
the deterministic versions, of the 
equations we deal with herein, are determined 
by the stochastic equations satisfied by 
certain special solutions, which in turn are natural 
stochastic extensions of the stationary solutions of the 
corresponding deterministic versions, which play a central role in the 
homogenization process in the deterministic case. 
Homogenization theory has been useful in many well known 
cases to derive equations from mechanics 
and other applied areas, as the Darcy law in two-phase flows 
in porous media (see, e.g., the 
famous appendix by Tartar in \cite{SP}), and we 
believe that the way the stochastic 
perturbations were derived here 
may be useful in applications. 

This paper is concerned with both the theory of homogenization of 
partial differential equations and the theory 
of stochastic differential equations. The homogenization 
theory of partial differential equations has been 
a field of intense research since the 1970's and we refer to the 
classical book \cite{BLP} for an account of this theory up to 1978. 
We also refer to the other classical book \cite{JKO} where a section 
is devoted to the homogenization theory in 
the context of ergodic algebras, which is the setting 
adopted in this paper. The homogenization methods 
used in this paper are based on those developed 
in \cite{AF} and \cite{AFS},  which in turn are mostly 
based on the concept of two-scale Young measures for almost periodic 
oscillations and its natural extension to ergodic algebras. 
Two-scale Young measures were introduced in the 
periodic case in \cite{E} (see also \cite{ES}) 
as an extension to the notion of two-scale convergence 
introduced in  \cite{N} and further developed in \cite{A} (see also \cite{Da}). 
Two-scale convergence for general oscillations in 
ergodic algebras were established in \cite{CG}, 
and corresponds to the linear case of the 
two-scale Young measures established in \cite{AF}, 
as proved in \cite{FSV}. 

The theory of stochastic partial differential 
equations has experienced intense progress in the 
last three decades and we cite the 
treatise \cite{DPZ} for a basic 
general account of this theory and references.
More specifically, concerning the theory 
of stochastic conservation laws, we 
mention the first contributions by Kim \cite{Kim}, 
and Feng and Nualart \cite{FN}. The latter was further 
developed in Chen, Ding, and Karlsen in \cite{CDK} 
and Karlsen and Storr{\o}sten in \cite{KSt}.  
An inflection in the course of this theory 
was achieved by Debussche and Vovelle \cite{DV} 
with the introduction of the notion of 
kinetic stochastic solution, extending 
the corresponding deterministic concept introduced by 
Lions, Perthame, and Tadmor \cite{LPT}.  
We also mention the independent development 
in this theory made by Bauzet, Vallet, and Wittbold \cite{BVW1}.  Concerning homogenization of stochastic partial 
differential equations, this has not been a frequently researched 
topic, although the earliest contribution seems to have appeared 
already in the early 1990's by Bensoussan in \cite{B}. 
As to more recent publications on this subject, we mention the 
contributions of Ichihara~\cite{Ic}, Sango~\cite{Sa}, 
Mohammed~\cite{MM}, and Mohammed and Sango~\cite{MMS}, 
among others. Consult also references in these papers.    

Concerning our method for proving Theorem~\ref{T:1.2} and 
Theorem~\ref{T:1.1}, the core of our technique
is to begin by using two-scale Young measures, 
as in \cite{E, ES, AF, AFS}, for instance,  then to 
derive a stochastic  kinetic equation satisfied 
by the generalized kinetic function associated with 
the two-scale Young measure,  and then to apply 
a uniqueness result for weak solutions of the 
corresponding stochastic kinetic equation, as 
is done in \cite{Da} in the deterministic 
periodic case for general conservation laws.  
  
This paper is organized as follows.  In Section~\ref{S:1'} we state and prove a 
result on the existence of stochastic two-scale Young measures which will be used in 
the two subsequent sections. 
In Section~\ref{S:2}, we address the 
homogenization of the stochastic nonlinear transport equation. 
In Section~\ref{S:3}, we deal with the same problem for the stochastic 
stiff oscillatory external force equation.  In Section~\ref{S:4} we 
establish a general well-posedness result for
stochastic conservation laws, which fits the needs of the present article. 
Finally, in Section~\ref{S:5}, we gather a general comparison 
principle and the so-called stochastic Kru\v zkov inequality. 
both needed for the analysis in Sections~\ref{S:2} and~\ref{S:3}.   
 
\section{Stochastic two-scale Young measures} \label{S:1'}

In the following sections our analysis will be based on the notion of two-scale Young measures as was
 done in the deterministic case in, e.g.,  \cite{E,ES, AF, AFS}. For future reference, we next state as a proposition 
 the existence of stochastic two-scale Young measures associated with  (generalized) subsequences satisfying bounds such as  \eqref{e2.4''} or \eqref{e1.8''} below. The proof follows  ideas  in \cite{AF}. Nevertheless, here there is the probability space $\Om$ and the stochastic integral  as new ingredients. Also we need to establish an estimate ({\em cf.}\ \eqref{eq:kinetic-Lp-bound-Young'}) that will be needed in the following sections.  
  Therefore, we include a detailed proof here for the convenience of the reader.  For simplicity, to avoid the use of generalized subsequences, we assume that our ergodic algebra is separable. 
 In practice, this means that if $\Psi_1(t,x,\frac{x}{\ve}, u),\ldots, \Psi_N(t,x,\frac{x}{\ve}, u)$ is the finite family of continuous oscillatory functions involved in our homogenization problem, we consider the closure of the subalgebra of $\A(\R^n)$ (invariant by translations) generated by the functions 
 $g_{\a_1,\b_1,\g_1}(y):=\Psi_1(t_{\a_1},x_{\b_1},y, u_{\g_1}),\cdots, g_{\a_N,\b_N,\g_N}(y):=\Psi_N(t_{\a_N},x_{\b_N}, y, u_{\g_N})$, $\a_i,\b_i,\g_i\in
 \N$, $i=1,\cdots,N$, where  $\{(t_{\a_i},x_{\b_i},u_{\g_i})\,:\, \a_i,\b_i,\g_i\in\N\}$ is a countable dense subset of $[0,\infty)\X\R^d\X\R$, for $i=1,\cdots,N$.

 \begin{proposition}\label{P:1.1}  Let $(\Om,\F,\bbP)$ be a probability space, with $\F$ countably generated, let $\F_t$ be 
 the filtration generated by the Brownian motion $W(t)$ and $\F_0$,
the  $\s$-algebra generated by the null sets of $\F$. Let $\A(\R^d)$ be a separable ergodic algebra and $\cK$ the associated separable compact space such that $\A(\R^d)\sim C(\cK)$, with associated invariant measure 
 $d\mm(y)$.  Let $u^\ve$, $\ve>0$, be a sequence of predictable functions in  $L^p(\Om; L_\loc^1([0,\infty)\X\R^d))$, for all $p\ge1$,  satisfying 
 \begin{equation}\label{e2.P1.1}
\abs{u^\ve(\om,t,x)}\le C_*(1+\abs{W(\om,t)}^{N_0}), 
\quad\text{for a.e.\ $(\om,t,x)\in\Om\X[0,\infty)\X\R^d$}, 
 \end{equation}
 for some $C_*>0$ and $N_0\in\N$. Let $w_N$ be 
 defined in \eqref{eq:scl-weight-def}. Then, there exists a subsequence,  $u^{\ve_k}$, $\ve_k \to0$, 
 and  a parameterized family of probability measures over 
 $\R$, $\nu_{\om,t,x,y}$,
 satisfying the properties:
 \begin{enumerate}
\item[(1)]  $\nu_{\om,t,x,y}$ is measurable, in the sense that for any $\z\in C_c(\R)$, $\la \nu_{\om,t,x,y},\z\ra$ is measurable with respect  to the sigma-algebra $\F\otimes\bbB([0,\infty)\X\R^d)\otimes\bbB(\cK)$;
 
 \item[(2)] For any $A\in\F$, denoting by $\EE_A$ the conditional expectation with respect to $A$, for all $\Psi\in C_c([0,\infty)\X\R^d\X\R;\A(\R^d))\sim C_c([0,\infty)\X\R^d\X\cK\X\R)$,
 \begin{multline}\label{e2.YM}
 \lim_{k\to\infty}{\EE}_A\int_{[0,\infty)\X\R^d}\Psi\left(t,x,\frac{x}{\ve_k},u^{\ve_k}(t,x)\right)w_N(x)\,dt\,dx\\=
{\EE}_A \int_{[0,\infty)\X\R^d\X\cK}\la \nu_{\om,t,x,y}, \Psi(t,x,y,\cdot)\ra w_N(x)\,d\mm(y) \,dt\,dx.
 \end{multline}
 
 \item[(3)] For a.e.\ $y\in\cK$, for all $T>0$, we have
 \begin{equation}\label{eq:kinetic-Lp-bound-Young'}
	\EE\left(\esup_{t\in[0,T]}
	\iint_{\R^d\times \R}\abs{\xi}^p
	w_N(x)\, \nu_{\omega,t,x,y}(d\xi) \dx\right)\le C_{T,N,p} , 
	\qquad 
	\forall p\in [1,\infty),
\end{equation}
where $C_{T,N,p}$ is a positive constant depending only on $T,N,p$. 

\item[(4)] If $\Psi\in C([0,\infty)\X\R^d\X\R;\A(\R^d))\sim C([0,\infty)\X\R^d\X\cK\X\R)$ is such that
$|\Psi(t,x,y,\xi)|\le 1_{[0,T_0]}(t) C(1+|\xi|^p)$, for some $p\ge1$ and $T_0>0$, then \eqref{e2.YM} holds for all 
$A\in\F$. More generally, for such $\Psi$,  if $\ell\in L^2(\Om)$ and $\tilde \Psi(\om,t,x,y,\xi)=\ell(\om)\Psi(t,x,y,\xi)$,
then
\begin{multline}\label{e2.YMom}
 \lim_{k\to\infty}{\EE}\int_{[0,\infty)\X\R^d}\tilde\Psi\left(\om,t,x,\frac{x}{\ve_k},u^{\ve_k}(t,x)\right)w_N(x)\,dt\,dx\\=
{\EE}\int_{[0,\infty)\X\R^d\X\cK}\la \nu_{\om,t,x,y}, \tilde\Psi(\om, t,x,y,\cdot)\ra w_N(x)\,d\mm(y) \,dt\,dx.
 \end{multline}

\item[(5)] If $\Psi\in C([0,\infty)\X\R^d\X\R;\A(\R^d))$ satisfying $|\Psi(t,x,y,\xi)|\le 1_{[0,T_0]}(t) C(1+|\xi|^p)$,
for some $p\ge1$ and $T_0>0$,  then
$$
(\om,t)\mapsto \int_{\R^d}\int_\cK \la \nu_{\om,t,x,y}\,, \Psi(t,x,y,\cdot)\ra w_N(x)\,d\mm(y)\,dx 
$$
is a predictable process on $\Om\X[0,\infty)$,  and, for any $A\in\F$,
  \begin{multline}\label{e2.YM'}
 \lim_{k\to\infty}{\EE}_A\int_0^\infty\int_{\R^d}\Psi\left(t,x,\frac{x}{\ve_k},u^{\ve_k}(t,x)\right)w_N(x)\,dx\,dW(t)\\=
{\EE}_A \int_0^\infty \int_{\R^d}\int_\cK \la \nu_{\om,t,x,y}\,, \Psi(t,x,y,\cdot)\ra w_N(x)\,d\mm(y)\,dx \,dW(t).
 \end{multline}
 Moreover, for $\mm$-a.e.\ $y\in\cK$,  
 $$
(\om,t)\mapsto \int_{\R^d}\la \nu_{\om,t,x,y}\,, \Psi(t,x,y,\cdot)\ra w_N(x)\,dx 
$$
is a predictable process on $\Om\X[0,\infty)$. 

 \end{enumerate}
 
 \end{proposition}
 
 \begin{proof} Let $W^*(t):=\max_{0\le s\le t}|W(s)|$ and, given $M>0$,  let $t_M:=\inf\{t\ge 0\,:\, W^*(t)\ge M\}$. Given 
  $T>0$, for $M$ sufficiently large, $t_M>T$. Therefore, taking $M\in\N$, making $M\to\infty$, and defining $\Om_M(T):=
 \{\om\in\Om\,:\, t_M(\om)>T\}$, we see that $\bbP\left(\Om\setminus\Om_M(T)\right)\to0$. Indeed, $\Om_M(T)$ is an increasing family of subsets of $\Om$ and if $\bbP\left(\Om\setminus\bigcup_{M\in\N}\Om_M(T)\right)>0$,
 then we would be able to find $\om\in\Om$ for which $W(t)$ is defined and continuous for $t\in[0,\infty)$ and such that
 $W^*(\om,t)\to+\infty$ as $t\to T$, which is absurd. We fix $T>0$, and, for simplicity we write simply $\Om_M$ instead of 
 $\Om_M(T)$.  So, for each $M\in\N$, we have that $u^\ve$ is a bounded sequence in $L^\infty(\Om_M\X[0,T]\X\R^d)$.
 Let us consider the countable family of real valued functions over $\Om_M$,  $\mathfrak{F}:=\{ W(\cdot, r)\,:\, r\in\mathbb{Q}\cap[0,T]\}$.  We may assume, without loss of generality, that the functions of the family 
 $\mathfrak{F}$  are defined at every point of  $\Om_M$ and that $\mathfrak{F}$ distinguishes between the points of $\Om_M$, that is, given 
 $\om_1,\om_2\in\Om_M$, $\om_1\ne\om_2$, then there is $r\in\mathbb{Q}\cap[0,T]$ such that $W(\om_1,r)\ne W(\om_2,r)$.
 The first assertion is clear since we may find a set of null $\bbP$-measure in $\Om$ out of which the functions in the countable family $\mathfrak{F}$ are defined everywhere, and so we can define them as 0 over this null $\bbP$-measure subset of $\Om$. The second assertion follows from the fact that we can define in $\Om$ the equivalence relation $\om_1\sim\om_2$
 if and only if $W(\om_1,r)=W(\om_2,r)$ for all $r\in\mathbb{Q}\cap[0,T]$. Then we define the quotient space 
 $\tilde \Om:=\Om/\sim$, with  the natural projection $\pi_{\sim}:\Om\to\tilde\Om$, $\pi_{\sim}(\om)=[
 \om]$, where $[\om]$ is the $\sim$-equivalence class of $\om$. We also define the class $\tilde\F$ of subsets of 
 $\tilde \Om$ by $\tilde A\in \tilde\F$ if and only if $\pi_\sim^{-1}(\tilde A)\in\F$, and for $\tilde A\in\tilde \F$ we define
 $\tilde \bbP(\tilde A)=\bbP(\pi_{\sim}^{-1}(\tilde A))$. It is easy to check that $\tilde \F$ is a sigma-algebra and $\tilde \bbP$ is a probability measure on $\tilde\Om$. Moreover, $W(t)$ is a Brownian motion over $\tilde\Om$, since the distributions of $W(t)$, $t\in[0,T]$,  on $(\Om,\F,\bbP)$ and on $(\tilde \Om,\tilde\F,\tilde\bbP)$ coincide; therefore, for all purposes, we can assume that the family $\mathfrak{F}$ distinguishes between the points of $\Om$; otherwise we replace $(\Om,\F,\bbP)$ by  the quotient space 
 $(\tilde \Om,\tilde \F, \tilde\bbP)$ and, once we obtain the result for the latter, it can be automatically lifted up to the original probability space $(\Om,\F,\bbP)$. 
 
 Let $B(\Om_M)$ be the algebra of bounded functions over $\Om_M$. Let  
 $\mathfrak{A}$ be the closed subalgebra of $B(\Om_M)$ generated by $\{1,\mathfrak{F}\}$. According to a well-known extension of the Stone-Weierstrass theorem (see \cite{DS}, p.274--276, Theorem~18 and Corollary~19) there exist a compact Hausdorff space $\overline{\Om_M}$ and an one-to-one embedding of $\Om_M$ as a dense subset of 
 $\overline{\Om_M}$, such that  each $\psi\in \mathfrak{A}$ has a unique continuous extension $\overline{\psi}$  to 
 $\overline{\Om_M}$, and such that the correspondence $\psi  \leftrightarrow \overline{\psi}$ is an  isomeric isomorphism between 
 $\mathfrak{A}$ and $C(\overline{\Om_M})$.  Moreover, the relation
 $$
 \int_{\overline{\Om_M}}\overline{\psi}(\om)\,d\bbP(\om):=\int_{\Om_M}\psi(\om) \,d\bbP(\om)
 $$
 defines $\bbP$ as a Radon measure over $\overline{\Om_M}$. In particular, we can endow $\Om_M$ with the topology induced by the embedding $\Om_M\to \overline{\Om_M}$ with respect to which $\bbP$ is a Radon measure and $\Om_M$ is relatively compact. Therefore, henceforth, for simplicity, we consider $\Om_M$ as compact and $\bbP$ as a Radon measure on $\Om_M$,  with the referred topology, which coincides with the topology generated by the family $\mathfrak{F}$.

 Let $L_M:=C_*(1+M^{N_0})$, where $C_*$ is as in \eqref{e2.P1.1}.  Denote by $C_0(\Om_M\X[0,T]\X \R^d\X [-L_M,L_M];\A(\R^d))$ the space of functions $\Psi(\om, t,x,y,\xi)$ continuous in
 $\Om_M\X[0,T]\X\R^d\X\R\X\R^d$, belonging to $\A(\R^d)$, as functions of $y$, for each fixed $(\om,t,x,\xi)\in \Om_M\X[0,T]\X\R^d\X\R$,  and such that $\Psi(\cdot,\cdot,x,\cdot,\cdot)\to0$  as $|x|\to\infty$, uniformly in  $\Om_M\X[0,T]\X\R\X\R^d$. Clearly,  $C_0(\Om_M\X[0,T]\X \R^d\X [-L_M,L_M];\A(\R^d))$ is isometrically isomorphic to 
 $C_0(\Om_M\X[0,T]\X \R^d\X\cK\X [-L_M,L_M])$, defined similarly. Given $\Psi\in C_0(\Om_M\X[0,T]\X \R^d\X [-L_M,L_M];\A(\R^d))$,
 define 
 $$
 \la \mu_M^{\ve}, \Psi\ra:=\int_{\Om_M\X [0,T]\X \R^d} 
 \Psi\left(\om,t,x,\frac{x}{\ve}, u^\ve\right) w_N\,dt\,dx\,d\bbP(\om).
 $$  
 Because we are assuming $|u^\ve(\om,t,x)|\le C_0(1+|W(t)|^{N_0})$, the above equation defines $\mu_M^\ve$ as a bounded sequence of Radon measures on  $\Om_M\X[0,T]\X\overline{ \R^d} \X\cK\X[-L_M,L_M]$, where $\overline{\R^d}$ is the one point compactification of $\R^d$ generated by $C_0(\R^d)$, the continuous functions on $\R^d$ vanishing at $\infty$.  Since the space   of the Radon measures on  $\Om_M\X[0,T]\X \overline{\R^d} \X\cK\X[-L_M,L_M]$  is compact in the weak-$\star$ topology by the Banach-Alaoglu theorem, we can 
find a subsequence $\mu_M^{\ve_{M,k}}$ converging to some Radon measure $\mu_M$ on $\Om_M\X[0,T]\X \overline{\R^d}\X\cK\X[-L_M,L_M]$. Making $M=1,2,\cdots$, we can extract for each $M>1$ a subsequence $\ve_{M,k}$ from the subsequence obtained for $M-1$, $\ve_{M-1,k}$, inductively, and then take the diagonal subsequence $\ve_{k,k}=:\ve_k$.  Observe that 
$\mu_M^{\ve}$
restricted to $\Om_{M-1}$, coincides with $\mu_{M-1}^{\ve}$. Therefore, the limit measure 
$\mu=\lim \mu_k^{\ve_{k}}$, which is well defined in $\Om_M\X[0,T]\X\R^d\X\cK\X\R$, for each $M\in\N$,  is then defined in $\Om\X[0,T]\X\R^d\X\cK\X\R$ and   coincides with $\mu_M$ when restricted to $\Om_M\X[0,T]\X\R^d\X\cK\X\R$. In particular, for all 
$C_c(\Om \X[0,T]\X \R^d\X \R;\A(\R^d))$ we have
\begin{multline}\label{e2.YMM} 
\lim_{k\to\infty}\int_{\Om}\int_0^T\int_{\R^d} 
\Psi\left(\om,t,x,\frac{x}{\ve_{k}},u^{\ve_{k}}\right)\,d\bbP\,dt\,w_Ndx\\
=\int_{\Om\X[0,T]\X\R^d\X\cK\X\R} 
\Psi(\om,t,x,y,\xi)\,d\mu(\om,t,x,y,\xi).
\end{multline}

Now it is easy to check that the projection of the measure $\mu$, obtained above, over $\Om\X[0,T]\X\R^d\X\cK$ is equal to $d\bbP\,dt\,w_Ndx\, d\mm(y)$, since this is true for any $\mu_M^{\ve}$.  We can then 
apply the theorem on disintegration of measures (see, e.g., theorem~2.28 in \cite{AFP}, whose extension to the present case is straightforward) to conclude the existence of a $d\bbP\,dt\,w_N dx\, d\mm(y)$-measurable family of probability measures $\nu_{\om,t,x,y}$ such that, for any $\Psi\in L^1(\Om\X[0,T]\X\R^d\X\cK\X\R;\mu)$ we have   
\begin{multline*}
\int_{\Om\X[0,T]\X\R^d\X\cK\X\R}\Psi(\om,t,x,y,\xi)\,d\mu(\om,t,x,y,\xi)\\
=
\int_{\Om\X[0,T]\X\R^d\X\cK}\left(\int_{\R}\Psi(\om,t,x,y,\xi)\, d\nu_{\om,t,x,y}(\xi)\right)\,d\bbP\,dt\,w_N dx\,d\mm(y).
\end{multline*}
 In particular, item (1) follows. 
 
 As for (2), it is enough to prove the result for all $A\in\mathcal{F}_t$, $t\ge 0$. So, take $A\in\F_{T_0}$ for some $T_0\ge 0$.  We can repeat the above construction 
 for $T=1,2,\cdots$, starting at $T=k$ with the subsequence obtained in $T=k-1$ and so, using again the diagonal argument,  we may define a subsequence which is good for any time interval $[0,T]$, with $T>0$ arbitrary. 
 In particular, we may assume that, for each $M\in\N$, $A\cap \Om_M$ is a Borel set in our topology for $\Om_M$.
 Therefore, given $M\in\N$, we can find sets $K$ and $V$ with $K$ compact and $V$ open in $\Om_M$ satisfying $K\subset A\cap\Om_M\subset V$ and such that $\bbP(A\cap\Om_M\setminus K)$ and $\bbP(V\setminus A\cap\Om_M)$ are
 arbitrarily small. We can also find $\psi\in C(\Om_M)$, with $1_K\le \psi\le 1_V$. Using a sequence of such $\psi\in C(\Om_M)$
 in \eqref{e2.YMM}, we get, for any $\Psi\in C_c( [0,\infty)\X\R^d\X\cK\X\R)$, 
 \begin{multline}\label{e2.YMM'} 
\lim_{k\to\infty}\int_{A\cap\Om_M}\int_0^T\int_{\R^d} 
\Psi\left(t,x,\frac{x}{\ve_{k}},u^{\ve_{k}}\right)
\,d\bbP\,dt\,w_Ndx\\
=\int_{A\cap\Om_M\X[0,T]\X\R^d\X\cK\X\R} \Psi(t,x,y,\xi)\,d\mu(\om,t,x,y,\xi).
\end{multline} 
Making $M\to\infty$, we get (2).   
 
 Concerning (3), given $\z\in C([0,T])$, with $\|\z\|_{L^1([0,T])}=1$ and $\varphi\in C(\cK)$, with $\|\varphi\|_{L^1(\cK)}=1$, for each $M$ we have
 \begin{multline*}
{\EE}_{\Om_M}\int_0^T\z(t)\left( \int_{\K}\varphi(y)\left( \int_{\R^d} w_N(x)\left(\int_\R|\xi|^p\,d\nu_{\om,t,x,y}(\xi)\right)\,dx\right)\,d\mm(y)\right)\,dt \\
 =\lim_{k\to\infty}{\EE}_{\Om_M}\int_0^T\z(t) \left( \int_{\R^d}\varphi\left(\frac{x}{\ve_k}\right) w_N(x)
 \abs{u^{\ve_{k}}}^p \,dx\right)\,dt\\
\le \lim_{k\to\infty} C\EE\int_0^T\z(t)\int_{\R^d} 
\left(1+\abs{W(t)}^{N_0}\right)^p
\abs{\varphi\left(\frac{x}{\ve_k}\right)} w_N(x)\,dx\,dt \\
\le C\|\z\|_{L^1([0,T])}\|\varphi\|_{L^1(\cK)} \int_{\R^d}w_N(x)\,dx \EE\sup_{0\le t\le T}\left(1+\abs{W(t)}^{N_0}\right)^p.
 \end{multline*} 
Since the right-hand side does not depend on $M$ we obtain
\begin{multline*}
{\EE} \int_0^T\z(t)\left( \int_{\K}\varphi(y)\left( \int_{\R^d} w_N(x)\left(\int_\R|\xi|^p\,d\nu_{\om,t,x,y}(\xi)\right)\,dx\right)\,d\mm(y)\right)\,dt \\ 
\le C\|\z\|_{L^1([0,T])}\|\varphi\|_{L^1(\cK)} \int_{\R^d}w_N(x)\,dx \EE\sup_{0\le t\le T}(1+|W(t)|^{N_0})^p\\
\le C\|\z\|_{L^1([0,T])}\|\varphi\|_{L^1(\cK)} 
\int_{\R^d}w_N(x)\,dx \EE \left(1+\abs{W(T)}^{N_0}\right)^p,
 \end{multline*} 
where the latter inequality follows from Doob's maximal inequality, and the fact that $(1+|W(t)|^{N_0})^p$ is a submatingale (see, e.g., \cite{DPZ}). Taking the $\sup$ for $\z\in L^1([0,T])$, with $\|\z\|_{L^1([0,T])}=1$, and $\varphi\in L^1(\cK)$, with $\|\varphi\|_{L^1(\cK)}=1$,
 we finally get
$$
{\EE} \sup_{t\in[0,T]}\sup_{y\in\cK} 
\int_{\R^d} w_N(x)\left(\int_\R|\xi|^p
\,d\nu_{\om,t,x,y}(\xi)\right)\,dx 
\le C_{T,N,p},
$$
and so \eqref{eq:kinetic-Lp-bound-Young'} follows.  

Concerning (4), for $\Psi\in C([0,\infty)\X\R^d\X\R;\A(\R^d))$ is such that
$|\Psi(t,x,y,\xi)|\le 1_{[0,T_0]}(t) C(1+|\xi|^p)$, for some $p\ge1$ and $T_0>0$, 
 \begin{multline*}
 \lim_{k\to\infty}{\EE}_A\int_{[0,\infty)\X\R^d}\Psi\left(t,x,\frac{x}{\ve_k},u^{\ve_k}(t,x)\right)w_N(x)\,dt\,dx\\
 =\frac{\bbP(A\cap\Om_M)}{\bbP(A)}\lim_{k\to\infty}{\EE}_{A\cap\Om_M}\int_{[0,T]\X\R^d} \Psi\left(t,x,\frac{x}{\ve_k},u^{\ve_k}(t,x)\right)w_N(x)\,dt\,dx +R_M\\
= \frac{\bbP(A\cap\Om_M)}{\bbP(A)} {\EE}_{A\cap\Om_M} \int_{[0,T]\X\R^d\X\cK}\la \nu_{\om,t,x,y},\tilde \Psi(t,x,y,\cdot)\ra w_N(x) \,dt \,dx\,d\mm(y) + R_M,
 \end{multline*}
where,
\begin{align*}
|R_M|&\leq \frac{\bbP(A\cap(\Om\setminus\Om_M))}{\bbP(A)}{\EE}_{A\cap(\Om\setminus\Om_M)}\int_{[0,T]}C(1+|W(t)|^{N_0})^p\,dt
 \int_{\R^d} w_N(x)\,dx\\
 &= \int_{\R^d} w_N(x)\,dx   \frac{1}{\bbP(A)}\int_{A\cap(\Om\setminus\Om_M)}\int_{[0,T]}C(1+|W(t)|^{N_0})^p\,dt\,d\bbP,
\end{align*}
which yields \eqref{e2.YM}, for such $\Psi(t,x,y,\xi)$,  by making $M\to\infty$. In particular, \eqref{e2.YMom} follows for $\ell(\om)=\frac1{\bbP(A)}1_A$, for $A\in\F$. Now, given any $\ell\in L^2(\Om)$, \eqref{e2.YMom} follows by approximating $\ell$ in $L^2(\Om)$ by finite linear combinations of indicator functions, which concludes the proof of (4). 

We now pass to the proof of (5). Let $\ve_k$ be the subsequence obtained above.  First, we note that \eqref{e2.P1.1} and the assumed bound on $\Psi$ implies that the sequence 
\[
(\om,t)\to \int_{\R^d}\Psi\left(t,x,\frac{x}{\ve_k},u^{\ve_k}(t,x)\right) w_N(x)\,dx, \qquad k\in \mathbb{N}
\]
is uniformly bounded in $L^2(\Om\X [0,T])$ and so it has a subsequence that converges weakly in $L^2(\Om\X [0,T])$. Since each element of the sequence is predictable, then the limit, which by \eqref{e2.YM} equals  
\[
\int_{\R^d}\int_{\cK}\la \nu_{\om,t,x,y}, \Psi(t,x,y,\cdot)\ra w_N\,dx \,d\mm(y),
\]
is also predictable.

Fix $T>0$, and consider the sequence of random variables
\[
X_k:=\int_0^T\int_{\R^d}\Psi\left(t,x,\frac{x}{\ve_k},u^{\ve_k}(t,x)\right)w_N(x)\,dx\,dW(t).
\]
Define also
\[X:= \int_0^T\int_{\R^d\X\cK}\la \nu_{\om,t,x,y}, \Psi(t,x,y,\cdot)\ra w_N(x)\,d\mm(y)\,dx \,dW(t).
\]

To prove (5) it is enough to show that any subsequence $\{X_{k_j}\}$ has a further subsequence that converges to $X$ weakly in $L^2(\Om)$.

Take any subsequence $\{X_{k_j}\}_j$. By \eqref{e2.P1.1} and the It\^ o isometry, we have that the $\{X_{k_j}\}_j$ is uniformly bounded in $L^2(\Om)$. Thus, it has a further subsequence which converges weakly to some $\tilde X\in L^2(\Om)$. For simplicity of notation, we denote this sub-subsequence by $\{X_{k_j}\}_j$ as well. In particular, for any predictable square integrable process $C(t)$ we have that
\begin{equation}\label{e.dW1}
\lim_{j\to\infty}\EE \left( X_{k_j} \int_0^T C(t)
\, dW(t) \right) = \EE \left(\tilde X  \int_0^T C(t)
\, dW(t)\right).
\end{equation}

On the other hand, using the It\^ o isometry and 
applying \eqref{e2.YM} we see that
\begin{align*}
&\lim_{j\to\infty}\EE \left( X_{k_j} \int_0^T C(t)\, dW(t) \right) \\
&\qquad=\lim_{j\to\infty}\EE\left( \int_0^T C(t)\int_{\R^d}\Psi\left(t,x,\frac{x}{\ve_{k_j}},u^{\ve_{k_j}}(t,x)\right)w_N(x)\,dx\, dt\right)\\
  &\qquad=\EE\left( \int_0^T C(t)\int_{\R^d\X\cK}\la \nu_{\om,t,x,y}, \Psi(t,x,y,\cdot)\ra w_N(x)\,dx\,d\mm(y) \, dt\right)
\end{align*}
Now, using the It\^ o isometry once again we see that
\begin{equation}\label{e.dW2}
\lim_{j\to\infty}\EE \left( X_{k_j} \int_0^T C(t)
\, dW(t) \right) = \EE \left(X  \int_0^T C(t)\, dW(t)\right).
\end{equation}

Comparing \eqref{e.dW1} and \eqref{e.dW2} we can conclude that $X=\tilde X$ a.s.. Indeed, note that $X$ is $\cF_T$-measurable, since every $X_k$ is. Also, note that $\EE(\tilde X)=\lim_{j\to\infty}\EE(X_{k_j})=0$. Then, we can define the $\cF_t$-martingale $Y(t)$ by $Y(t):= \EE(\tilde X | \cF_t)$, $0\le t\le T$ (which is sometimes called the Doob martingale associated to $\tilde X$). By the martingale representation theorem (see, e.g., \cite{RY}) we have that there is some predictable integrable process $D(t)$ such that
\[
Y(t)=\int_0^tD(s)\, dW(s).
\] 

Then, choosing
\[
C(t)=D(t)-\int_{\R^d\X\cK}\la \nu_{\om,t,x,y}, 
\Psi(t,x,y,\cdot)\ra w_N(x)\,dx\,d\mm(y),
\]
by virtue of \eqref{e.dW1} and \eqref{e.dW2}, we have that
\[
\EE\left(\abs{\tilde X - X}^2\right)
=\EE\left( (\tilde X - X)\int_0^TC(t)\, dW(t)\right)=0,
\]
which proves the claim.

Since this holds for any sub-subsequence of $\{X_k \}_k$ we have that the whole sequence converges to $X$ weakly in $L^2(\Om)$. In particular, given $A\in \cF$ we have that
\[
\lim_{k\to\infty}\EE\left( 1_A X_k \right) 
= \EE\left( 1_A X \right),
\]
which yields \eqref{e2.YM'}.

Moreover, to prove the  assertion about the predictability of 
\[
\int_{\R^d}\la \nu_{\om,t,x,y}, \Psi(t,x,y,\cdot)\ra w_N\,dx,
\]
for $\mm$-a.e.\ $y\in\cK$ we argue as follows. Since we are assuming that the separable ergodic algebra $\A(\R^d)$ is a subalgebra 
of $\APs(\R^d)$, we may as well assume that $\A(\R^d)$ contains the trigonometric functions $\sin\l\cdot y, \cos\l\cdot y$, for all $\l\in\R^d$ such that $\int_{\cK} g(y) e^{i \l\cdot y}\,d\mm(y)\ne 0$, for some $g\in\A(\R^d)$, which is sometimes called the spectrum of the algebra $\A(\R^d)$, which is a countable set; otherwise we can augment $\A(\R^n)$ to a separable ergodic
algebra containing such trigonometric functions.  In particular, it contains an almost periodic approximation of the unit, that is, a sequence of functions in $\AP(\R^d)$, $\{\rho_k(y)\,:\,k\in\N\}$, such that for all
$g\in\A(\R^d)$, $\rho_k* g(y)=\int_{\cK} \rho_k(y-z)g(z)\,d\mm(z)\to g_*(y)$ in $C(\cK)$, where $g_*$ is the almost periodic component of $g$, and so the convergence is a.e.\ in $\cK$ to $g$; $\rho_k$ may be taken as the Bochner-Fej\'er polynomials associated with the spectrum of the algebra $\A(\R^d)$ (see, e.g., \cite{Be}).  Since $C(\cK)$ is dense in $L^1(\cK)$,  $\rho_k* g(y)\to g(y)$ in $L^1(\cK)$ for all $g\in L^1(\cK)$. 
Now, from what was seen before, for each $y\in\cK$, 
 \[
(\om,t)\mapsto \int_{\R^d}\int_{\cK}\rho_k(y-z)\la \nu_{\om,t,x,z}, \Psi(t,x,z,\cdot)\ra w_N\,dx \,d\mm(z),
\]
is predictable, for all $k\in\N$.  

Let us fix $T>0$. Since $\F$ is countably generated, we can find a family $\{\psi_l\,:\,l\in\N\}$ in $L^\infty(\Om\X[0,T])$ dense in $L^2(\Om\X[0,T])$. Then using the bound for $\Psi$ and \eqref{eq:kinetic-Lp-bound-Young'}, we have that for all $l\in\N$
\begin{multline}\label{e2.100}
\lim_{k\to\infty} \int_{\cK} \rho_k(y-z)\left(\EE \int_0^T \psi_l(\om,t)\int_{\R^d} \la \nu_{\om,t,x,z}, \Psi(t,x,z,\cdot)\ra w_N\,dx\,dt\right)\,d\mm(z)\\
=\EE \int_0^T \psi_l(\om,t)\int_{\R^d} \la \nu_{\om,t,x,y}, \Psi(t,x,y,\cdot)\ra w_N\,dx\,dt
\end{multline}
in $L^1(\cK)$ and, after passing to a subsequence if necessary, the convergence is also a.e.\ in $\cK$. 

Let us now fix, $y\in\cK$ in the subset of full measure in $\cK$ for which \eqref{e2.100} holds for all $l\in\N$.  Using Jensen inequality, the bound for $\Psi$ and \eqref{eq:kinetic-Lp-bound-Young'}, we see that the functions 
$$
\g_k(\om,t):= \int_{\cK} \rho_k(y-z) \int_{\R^d} \la \nu_{\om,t,x,z}, \Psi(t,x,z,\cdot)\ra w_N\,dx\,d\mm(z),\quad k=1,2,\cdots,
$$
form a bounded sequence in $L^2(\Om\X[0,T])$. Then, given any subsequence of this sequence, we can find a further subsequence converging weakly in $L^2(\Om\X[0,T])$, and, because of \eqref{e2.100},  its weak limit in $L^2(\Om\X[0,T])$ 
must be
$$
\g(\om,t):= \int_{\R^d} \la \nu_{\om,t,x,y}, \Psi(t,x,y,\cdot)\ra w_N\,dx,
$$
therefore, the whole sequence $\g_k$ converges weakly to $\g(\om,t)$ in $L^2(\Om\X[0,T])$.  Now, since the $\g_k$'s are 
predictable and $L^2(\Om\X[0,T];\cP)$ is a closed subspace of $L^2(\Om\X[0,T])$, we deduce that $\g(\om,t)$ is also predictable,
for a.e.\ $y\in\cK$,
which concludes the proof.

\end{proof}

\begin{remark}\label{R:2.1} We remark that it follows from item (4) of Proposition~\ref{P:1.1} that given $F\in C(\mathbb{R})$ such that $|F(\xi)|\le C(1+|\xi|^p)$, for some $p\ge 1$, and letting $\overline{F}(\om,t,x,y):=\la \nu_{\om,t,x,y}, F\ra $, then for any $T>0$ we have that $F(u^{\ve_k})\rightharpoonup \overline{F(u)}$ in the weak topology of $L^2(\Om;L_\loc^2((0,T)\times\mathbb{R}^d))$, where 
\begin{equation}\label{e2.New}
\overline{F(u)}(\om,t,x)=\int_{\cK}\overline{F}(\om,t,x,y)d\mm(y).
\end{equation}
In particular, if $\nu_{\om,t,x,y}=\delta_{U(\om,t,x,y)}$, then  $\overline{F(u)} = \int_{\cK}F(U)d\mm(y)$.  

Indeed, it suffices to take $\tilde \Psi$ in \eqref{e2.YMom}  of the form $\Psi(\om, t,x,y,\xi)=\ell(\om)\psi(t,x)F(\xi)$ with $\ell\in L^2(\Om)$, 
$\psi\in C_c((0,T)\times\mathbb{R}^d)$ arbitrary to deduce \eqref{e2.New}, observing that, by  \eqref{e2.P1.1} and the assumption on $F$,  $F(u^{\ve_k})$ is bounded in $L^2(\Om;L^\infty((0,T)\times\mathbb{R}^d))$, therefore bounded in 
$L^2(\Om;L^2((0,T)\X\{|x|<R\}))$, for each $R>0$
 and the functions of the form $\ell(\om)\psi(t,x)$ with 
$\ell\in L^2(\Om)$ and $\psi\in C_c((0,T)\X\R^d)$ are dense in $L^2(\Om; L^2((0,T)\X\R^d))$. 

\end{remark}

The next result gives sufficient conditions for the existence of correctors for the weak convergence of the sequence $u^{\ve_k}$ established by Remark~\ref{R:2.1}. 

\begin{proposition}\label{P:1.2} Let $\nu_{\om,t,x,y}$ be the stochastic two-scale Young measure constructed in Proposition~\ref{P:1.1}. Assume $\nu_{\om,t,x,y}=\d_{U(\om,t,x,y)}$ for
\begin{enumerate}
\item[(a)] $U\in L^2(\Om; \A(\R^d; L^\infty((0,T)\X\R^d)))$, or 
\item[(b)] $U\in L^2(\Om; \B^2(\R^d; C_b([0,T]\X\R^d)))$.
\end{enumerate}
Then, $u^{\ve_k}-U\left(\om,t,x,\frac{x}{\ve_k}\right)\to 0$ strongly in  $L^2(\Om; L_\loc^2((0,T)\X\R^d))$.
\end{proposition}   

\begin{proof} First we observe that, because we only seek to show convergence in  
$L^2(\Om; L_\loc^2((0,T)\X\R^d))$, for item (b), we can just consider  $U\in L^2(\Om; \B^2(\R^d; C_c((0,T)\X\R^d)))$. Second, we see that the result would follow immediately from Proposition~\ref{P:1.1} if we were allowed to use 
\begin{equation}\label{eP12_0}
\tilde \Psi(\om,t,x,y,\xi)=|\xi-U(\om,t,x,y)|^2=\xi^2-2 \xi U(\om,t,x,y)+|U(\om,t,x,y)|^2,
\end{equation}
as a test function in \eqref{e2.YMom}. Let us check this possibility for each of the terms in the right-hand of the last equation in \eqref{eP12_0}. The first term, $\xi^2$, is good and, by Remark~\ref{R:2.1}, we have 
$$
\lim_{k\to\infty}\EE\int_{(0,T)\X\R^d}|u^{\ve_k}|^2 w_N\,dx\,dt=\EE\int_{(0,T)\X\R^d\X\cK}|U(\om,t,x,y)|^2w_N(x)\,d\mm(y)\,dx\,dt.
$$
Concerning the second term, $-2\xi U(\om,t,x,y)$, we observe first that if 
$$
U\in L^2(\Om; \A(\R^d;C_c((0,T)\X\R^d)))
$$ 
we could approximate it in 
$$
L^2(\Om; \A(\R^d;C_c((0,T)\X\R^d)))
$$ 
by finite linear combinations of functions of the form $1_A(\om)\psi(t,x,y)$ with $\psi\in
 \A(\R^d;\\ C_c((0,T)\X\R^d))$, $A\in\F$, and for such functions we could apply Proposition~\ref{P:1.1} to obtain
 \begin{multline}\label{eP12_1}
\lim_{k\to\infty}\EE\int_{(0,T)\X\R^d} u^{\ve_k} U(\om,t,x,\frac{x}{\ve_k}) w_N\,dx\,dt\\
=\EE\int_{(0,T)\X\R^d\X\cK}
|U(\om,t,x,y)|^2 w_N(x)\,d\mm(y)\,dx\,dt,
\end{multline}
so this equation holds for $U\in L^2(\Om; \A(\R^d;C_c((0,T)\X\R^d)))$.

Now, in case (a), for $U\in L^2(\Om;\A(\R^d; L^\infty((0,T)\X\R^d)))$, we can approximate $U$ in $L^2(\Om; \A(\R^d; L_\loc^2((0,T)\X\R^d)))$ by a sequence of functions in $L^2(\Om;\A(\R^d;C_c((0,T)\X\R^d)))$ to
obtain that \eqref{eP12_1} holds also for $U\in L^2(\Om;\A(\R^d; L^\infty((0,T)\X\R^d)))$.

In case (b), if $U\in L^2(\Om; \B^2(\R^d; C_c((0,T)\X\R^d)))$, we  can approximate $U$ in 
$$
L^2(\Om; \B^2(\R^d; C_c((0,T)\X\R^d)))
$$ 
by a sequence of functions in  $L^2(\Om; \A(\R^d; C_c((0,T)\X\R^d)))$ and for the latter we have already shown that equation~\eqref{eP12_1} holds, so it also holds $U\in L^2(\Om; \B^2(\R^d; C_c((0,T)\X\R^d)))$. 

Now, concerning the last term in the right-hand side of the last equation in \eqref{eP12_0}, it does not depend 
on $\xi$, so we just need to use the well known fact that, for a function  $\Psi\in L^2(\Om; \A(\R^d; L^\infty((0,T)\X\R^d)))$, in case (a), 
and  $\Psi\in L^2(\Om; \B^2(\R^d; C_c((0,T)\X\R^d)))$ in case (b), 
$$
\Psi(\om,t,x,\frac{x}{\ve_k})\wto \int_{\cK} \Psi(\om,t,x,y)\,d\mm(y),
$$
in the weak topology of $L^2(\Om;L_\loc^2((0,T)\X\R^d))$, and so we get
\begin{multline*}
\lim_{k\to\infty}\EE\int_{(0,T)\X\R^d} \left|U\left(\om,t,x,\frac{x}{\ve_k}\right)\right|^2 w_N\,dx\,dt\\=
\EE\int_{(0,T)\X\R^d\X\cK} |U(\om,t,x,y)|^2 w_N\,\,d\mm(y)\,dx\,dt.
\end{multline*}
Putting together the facts described above, we conclude that
\begin{equation*}
\lim_{k\to\infty}\EE\int_{(0,T)\X\R^d}\left|\, u^{\ve_k}-U\left(\om,t,x,\frac{x}{\ve_k}\right)\right|^2 w_N\,dx\,dt
=0,
\end{equation*}
which finishes the proof. 
\end{proof}

\section{Stochastic nonlinear transport, proof of Theorem \ref{T:1.2}}\label{S:2}
 
In this section we prove Theorem~\ref{T:1.2}. 
By assumption, we have that $h = \s'\s$. Set
\begin{equation}\label{e2.3}
\psi_\a(t)= g(\a+\k_0W(t)),
\end{equation}
where $g$ is a solution of the ODE 
$g'(\xi)=\s(g(\xi))$ and $\a\in\R$. 
We assert that $\psi_\a$ is a solution of 
equation \eqref{e2.1}, for any $\a\in\R$. Indeed, since 
$g''(\xi)=\s'(g(\xi))g'(\xi)=\s'(g(\xi))\s(g(\xi))=h(g(\xi))$, the 
assertion follows from the It\^{o} formula.
 
By the Stochastic Kru{\v {z}}kov 
inequality, cf. Proposition~\ref{p:Kato-Kruzkov}, a.s. we have
\begin{multline}\label{e2.4}
\int_0^\infty \int_{\R^d}\biggl\{ \abs{u^\ve-\psi_\a(t)}\phi_t
+\sgn(u^\ve-\psi_\a(t))\left(f(u^\ve)
-f(\psi_\a(t))\right)a\left(\frac{x}{\ve}\right)\cdot\nabla_x\phi  \\ 
+ \frac12 \k_0^2 \sgn(u^\ve-\psi_\a)
\left(h(u^\ve)-h(\psi_\a)\right)\phi \biggr\}\, dx\,dt 
+\int_{\R^d} \abs{U_0\left(x,\frac{x}{\ve}\right)-g(\a)}\phi\,dx\\
+\int_0^\infty \int_{\R^d} \k_0\sgn(u^\ve-\psi_\a)
\left(\s(u^\ve)-\s\left(\psi_\a\right)\right)\phi\,dx\,dW(t)\ge 0.
\end{multline}
A similar inequality holds with $(\cdot-\cdot)_+$ instead 
of $|\cdot-\cdot|$, which easily follows by adding 
to \eqref{e2.4} the difference 
of integral equations defining weak solutions 
for $u^\ve(t,x)$ and for $\psi_\a(t)$.  
{}From  \eqref{e2.4} we easily get 
the comparison principle 
\begin{multline*}
\EE \int_0^\infty \int_{\R^d}\biggl\{ 
\left(u^\ve-\psi_\a(t)\right)_+\phi_t
+\sgn(u^\ve-\psi_\a(t))_+\left(f(u^\ve)-f(\psi_\a(t))\right)
a\left(\frac{x}{\ve}\right)\cdot\nabla_x\phi  \\ 
+ \frac12\k_0^2\sgn(u^\ve-\psi_\a)_+ 
\left(h(u^\ve)-h(\psi_\a)\right)\phi \biggr\}\, dx\,dt \\
+\int_{\R^d} \left(U_0\left(x,\frac{x}{\ve}\right)
-g(\a)\right)_+\phi(0,x)\,dx\ge 0,
\end{multline*} 
which,  when 
$g(\a_1)\le U_0\left(x,\frac{x}{\ve}\right)\le g(\a_2)$, 
for some $\a_1,\a_2\in \R$,  implies a.s.\  the following uniform boundedness of 
the solutions of \eqref{e2.1}-\eqref{e2.2}
\begin{equation}\label{e2.4''}
\psi_{\a_1}(t)\le u^\ve(t,x)\le \psi_{\a_2}(t), \quad \text{for a.e.\ $(t,x)$}.
\end{equation}
 
We recall that it follows  from the definition 
of entropy solution (see Definition~\ref{D:2.1}), for  
any $C^2$ convex function $\eta:\R\to\R$, and 
$q$ satisfying $q'=\eta' f'$, $u^\ve$ satisfies
\begin{equation}\label{e2.5}
\begin{aligned}
\int_0^\infty\int_{\R^d}\biggl\{ \eta(u^\ve)\phi_t
+q(u^\ve) a\left(\frac{x}{\ve}\right)\cdot\nabla\phi 
+\frac12 \k_0^2\left(\eta'(u^\ve)h(u^\ve)
+ \eta''(u^\ve)\s(u^\ve)^2\right)\phi\biggr\}\,dxdt\\
\int_0^\infty\int_{\R^d} \k_0\eta'(u^\ve)\s(u^\ve)\phi\,dx\,dW(t)
+\int_{\R^d} \eta\left(U_0\left(x,\frac{x}{\ve}\right)\right)
\phi(0,x)\,dx\ge0. 
\end{aligned}
\end{equation}

Now, in equation \eqref{e2.5} we take $\phi(t,x)
=\ve\varphi\left(\frac{x}{\ve}\right)\zeta(t)\vartheta(x)$, 
where $0\le\varphi \in \A(\R^d)$, $\nabla\varphi\in\A(\R^d;\R^d)$, $0\le\zeta\in C_c^\infty([0,\infty))$
and $0\le \vartheta\in C_c^\infty(\R^{d})$, take conditional expectation with respect to an arbitrary $A\in\F$, and 
let $\ve\to0$, along a subsequence for which $u^\ve$ 
generates a two-scale Young measure $\nu_{\om, t,x,y}$, according to Proposition~\ref{P:1.1}, to obtain, since $A\in\F$
is arbitrary and we drop the $\om$ subscript from $\nu_{\om,t,x,y}$, a.s.,
\begin{equation}\label{e2.6}
\int_0^\infty \zeta(t)\int_{\cK}
\la \s_{t,y}^\vartheta, q(\cdot)\ra a(y)\cdot \nabla_y \varphi
\, d\mm(y)\, dt\ge 0,
\end{equation}
where
\begin{equation*}
\s_{t,y}^\vartheta:=\int_{\R^d} 
\vartheta(x)\nu_{t,x,y}\,dx.
\end{equation*}
By applying inequality \eqref{e2.6} to $C\pm\varphi$, 
with $C=\|\varphi\|_\infty$, and using the 
arbitrariness of $\varphi$,
\begin{equation} \label{e2.10}
y\mapsto\la \s_{t,y}^\vartheta,q(\cdot)\ra\in\S,
\quad\text{for a.e. }t\in (0,T).
\end{equation}
Now, for any $\eta\in C^2$, $C|u|^2+\eta(u)$ is convex 
for $C$ sufficiently large (depending on $\eta$), so \eqref{e2.10} holds for any $\eta\in C^2$ 
and, by approximation, for any Lipschitz continuous $\eta$. 
Now, if $f'\ne0$, given any $\eta\in C^1$, defining $\tilde \eta'=\eta'/f'$, the 
entropy-flux associated to $\tilde\eta$ is $\tilde q=\eta$, 
so that \eqref{e2.10} gives
\begin{equation}\label{e2.11}
y\mapsto\la \s_{t,y}^\vartheta,\eta(\cdot)\ra\in\S,
\qquad \text{for a.e. } t\in (0,T) \text{ and all $\eta\in C^1$}.
\end{equation}
In the more general case, where $\abs{\set{u\,:\, f'(u)=0}}=0$, we argue 
as in \cite{FS2} to deduce that \eqref{e2.11} still holds. Namely, for any open interval $I$ with $\bar I\subset \R\setminus E_0$, where 
$E_0=\set{u\,:\, f'(u)=0}$, we define $\eta_I'=\chi_I/f'$, 
where $\chi_I$ is the indicator function of the interval $I$, 
whose  corresponding entropy flux is $q_I$, with $q_I'=\chi_I$. 
Now, by approximation with convergence everywhere, the property may 
be extended to any open interval in $\R\setminus E_0$. Also, since the 
intersection of any open set with $\R\setminus E_0$ is a countable union 
of intervals in $\R\setminus E_0$, by approximation with convergence everywhere we get the property for any such intersection, and since $E_0$ has measure zero, the primitive of such intersection 
is equal to the primitive of the interval itself, 
so the property holds for $q_I$, where $I$ is any open interval, and 
hence for $q_I$ where $I$ is any interval. Since any $C^1$ 
function may be uniformly approximated 
by piecewise linear functions, which are linear 
combinations of $q_I$ functions, we 
deduce that \eqref{e2.11} also holds 
in this more general case.  
  
Now, we take $\phi(t,x)=\varphi\left(\frac{x}{\ve}\right)
\vartheta(t,x)$ in \eqref{e2.5}, 
where $0\le\varphi\in\S^\dag$ and 
$0\le\vartheta\in C_c^\infty(\R^{d+1})$, and take the conditional expectation with respect to an arbitrary $A\in\F$. 
Passing to the limit as $\ve\to0$ in \eqref{e2.5}, 
along a subsequence  which generates 
a two-scale Young measure according to Proposition~\ref{P:1.1},  as above, we get, a.s.,
\begin{multline} \label{e2.18}
\int_0^\infty\int_{\R^d}\int_{\cK} 
\biggl\{ \la \nu_{t,x,y}, \eta\ra \vartheta_t
+\la\nu_{t,x,y}, q\ra \tilde a(y)\cdot\nabla\vartheta 
+\frac12 \k_0^2\la \nu_{t,x,y}, (\eta' h 
+ \eta'' \s^2)\ra \vartheta \biggr\}\\
\times  \varphi(y)\,d\mm(y)\,dx\,dt\\
+\int_0^\infty\int_{\R^d}
\int_{\cK} \k_0\la \nu_{t,x,y}, 
\eta' \s\ra \vartheta\, \varphi(y)\,d\mm(y)\,dx\,dW(t)
\\+\int_{\R^d}\int_\cK \eta\left(U_0(x,y)\right)
\vartheta(0,x)\varphi(y)\,d\mm(y)\,dx\ge0. 
\end{multline}
Observe that instead of $a(y)$ we 
have, in \eqref{e2.18}, $\tilde a(y)$, the vector field whose 
components are the orthogonal projections  of the 
corresponding components of $a(y)$ onto $\S$, 
in $L^2(\cK)$, which is due to \eqref{e2.11}. 
Indeed, we use the fact that, for 
$g\in \left(\S\cap L^\infty\right)(\cK)$ and $r\in L^2(\cK)$, the 
orthogonal projection of $gr$ onto $\S$, $\widetilde{gr}$, 
is equal to $g\tilde r$, where $\tilde r$ is the 
orthogonal projection of $r$ 
onto $\S$ (see Proposition~4.2 in \cite{FS2}). 
Since we assume that $\S^\dag$ is dense in $\S$,  
we can extend \eqref{e2.18}  from $0\le \varphi\in\S$ 
to all  $0\le \varphi \in L^2(\cK)$, 
where we also use condition \eqref{e2.7'} 
on the initial data $U_0(x,y)$. 
Therefore, for $\bbP$-a.e.\ $\om\in\Om$ and $\mm$-a.e.\ $y\in\cK$, we have, for all 
$\vartheta\in C_c^\infty(\R^{d+1})$, 
\begin{multline} \label{e2.18'}
\int_0^\infty\int_{\R^d} \biggl\{ \la \nu_{t,x,y}, \eta\ra \vartheta_t
+\la\nu_{t,x,y}, q\ra \tilde a(y)\cdot\nabla\vartheta 
+\frac12 \k_0^2\la \nu_{t,x,y}, (\eta' h + \eta'' \s^2)\ra \vartheta \biggr\}\,dx\,dt\\
+ \int_0^\infty\int_{\R^d} \k_0\la \nu_{t,x,y}, 
\eta' \s\ra \vartheta \,dx\,dW(t)
\\+\int_{\R^d} \eta\left(U_0(x,y)\right)
\vartheta(0,x)\,dx\ge0. 
\end{multline}
Now for a convex $\eta\in C^3(\R)$, 
such that $\eta''\in C_c^1(\R)$, we have 
the obvious formulas
\begin{equation}\label{e2.kf}
\begin{aligned}
&\eta(\cdot)=\int_\R \eta'(\xi) 
\En_{(-\infty,\,\cdot)}(\xi)\,d\xi,\\
&q(\cdot)=\int_{\R} f'(\xi)\eta'(\xi) 
\En_{(-\infty,\,\cdot)}(\xi)\,d\xi,\\
& \left(\eta' h + \eta'' \s^2\right)(\cdot)
=\int_{\R} \left(\eta'h'+\eta''(h+2\s\s')
+\eta'''\s^2\right) (\xi)
\En_{(-\infty,\,\cdot)}(\xi)\,d\xi,\\
&\left(\eta'\s\right)(\cdot)
=\int_{\R} \left(\eta'\s'
+\eta''\s\right)(\xi) 
\En_{(-\infty,\,\cdot)}(\xi)\,d\xi,\\
&\eta\left(U_0(x,y)\right)
=\int_{\R} \eta'(\xi)\En_{(-\infty,U_0(x,y))}(\xi)\,d\xi. 
\end{aligned}
\end{equation}
Therefore, for a fixed $y\in\cK$, 
setting $\rho_1(t,x,\xi)=\nu_{t,x,y}((\xi,+\infty))$, 
we get from \eqref{e2.18'}
\begin{multline} \label{e2.18''}
\int_0^\infty\int_{\R^d}\int_{\R} 
\biggl\{ \rho_1(t,x,\xi)\eta'(\xi) \vartheta_t
+\rho_1(t,x,\xi) f'(\xi) \eta'(\xi) \tilde a(y)
\cdot\nabla\vartheta \biggr.\\
\biggl. +\frac12 \k_0^2 \rho_1(t,x,\xi) 
\left(\eta' h' + \eta'' (h+2\s\s')+\eta'''\s^2\right) (\xi) 
\vartheta \biggr\}\,d\xi\,dx\,dt\\
+ \int_0^\infty\int_{\R^d}\int_{\R} 
\k_0\rho_1(t,x,\xi) 
\left(\eta' \s'+\eta''\s\right)(\xi) 
\vartheta\,d\xi \,dx\,dW(t)
\\+\int_{\R^d}\int_\R \eta'(\xi) 
\En_{(-\infty,U_0(x,y))}(\xi)
\vartheta(0,x)\,d\xi\,dx\ge0. 
\end{multline}
Thus, seeing the left-hand side of the inequality above 
as a  distribution applied to $\eta''(\xi)\vartheta(t,x)$, 
we conclude that it is indeed a measure, which we denote 
by $m_1(t,x,\xi)$.  We can then extend the identity defining $m_1$ 
to any $\eta$ of the form $\eta(\xi)=\int_{-\infty}^\xi \z(s)\,ds$, 
for some $\z\in C_c^\infty (\R)$. This way we deduce 
that $\rho_1$ is a weak solution of the following kinetic equation
\begin{equation}\label{e2.19}
\frac{\po\rho_1}{\po t}
+ f'(\xi)\tilde a(y)\cdot\nabla \rho_1
+\frac12\k_0^2 h(\xi)\po_{\xi}\rho_1
-\frac12\k_0^2\po_{\xi}(\s^2\po_\xi \rho_1)
=\po_\xi m_1-\k_0\s\po_\xi \rho_1\frac{dW(t)}{dt},
\end{equation} 
and we see from \eqref{e2.18''} that $\rho_1$ verifies 
\begin{equation}\label{e2.20}
\underset{t\to0+}{\esslim}\int_{\R^{d+1}} 
\rho_1(t,x,\xi) \z(\xi)\phi(x)\,d\xi\,dx=\int_{\R^{d+1}} 
\En_{(-\infty,U_0(x,y))}(\xi)\z(\xi)\phi(x)\,d\xi\,dx,
\end{equation}
for all $\z\in C_c^\infty(\R)$ and 
all $\phi\in C_c^\infty(\R^d)$, see also 
Remark \ref{rem:init-measure}.

Considering the definition of the 
measure $m_1(t,x,\xi)$ from \eqref{e2.18''}, 
we may check that $m_1$ satisfies the 
conditions of a kinetic measure in 
Definition~\ref{def:kinetic-measure}. Also, $\rho_1(t,x,\xi)$ 
is a generalized kinetic function whose 
associated Young measure, $\nu_{t,x,y}$, 
satisfies \eqref{eq:kinetic-Lp-bound-Young}, which 
can be verified without difficulty 
using the bounds \eqref{e2.4''}.   

\medskip
   
Now, let  $U(t,x,y)$ be the entropy solution 
of \eqref{e2.22}-\eqref{e2.23}. 
{}According to Definition~\ref{D:2.2}, 
recalling \eqref{e2.25_0}, for each $y\in\cK$, 
for any convex $\eta\in C^2(\R)$, 
and $0\le \varphi\in C_c^\infty(\R^{d+1})$, 
 \begin{multline*}
\int_Q \eta(U(y))\po_t\varphi 
+q(U(y)) \tilde a(y)\cdot\nabla\varphi 
+\frac{\k_0^2}2 \left( \eta'(U(y)) h(U(y)) 
+\eta''(U(y))  \s^2( U(y)) \right) \varphi\,dx\,dt
\\ + \k_0\int_0^T\int_{\R^d}\eta'(U(y))
\s(U(y))\varphi\,dx\,dW(t) 
+\int_{\R^d}\eta(U_0(y))\varphi(0,x)\,dx\,dt\ge 0.
\end{multline*}
Setting $\rho_2(t,x,\xi)=1_{(-\infty,U(t,x,y))}(\xi)$, 
using the formulas \eqref{e2.kf}, we get from \eqref{e2.25_0}
\begin{multline} \label{e2.21}
\int_0^\infty\int_{\R^d}\int_{\R} 
\biggl\{ \rho_2(t,x,\xi)\eta'(\xi) \varphi_t
+\rho_2(t,x,\xi) f'(\xi) \eta'(\xi) 
\tilde a(y)\cdot\nabla\varphi \biggr.\\
\biggl. +\frac12 \k_0^2 \rho_2(t,x,\xi) \left(\eta' h'(\xi) + \eta'' (h+2\s\s')(\xi)+\eta'''\s^2(\xi)\right) \varphi \biggr\}\,d\xi\,dx\,dt\\
+ \int_0^\infty\int_{\R^d}\int_{\R} \k_0\rho_2(t,x,\xi) 
\left(\eta' \s'(\xi)+\eta''\s(\xi)\right) \varphi\,d\xi \,dx\,dW(t)
\\
+\int_{\R^d}\int_\R \eta'(\xi) \En_{(-\infty,U_0(x,y))}(\xi)
\varphi(0,x)\,d\xi\,dx\ge0. 
\end{multline}
Thus, again, the left-hand side of \eqref{e2.21} 
defines a measure $m_2(t,x,\xi)$ 
applied to $\varphi\eta''$.  Therefore, as above, 
we see that $\rho_2$ is a weak solution of the kinetic equation 
\begin{equation}\label{e2.22'}
\frac{\po\rho_2}{\po t}+ f'(\xi)\tilde a(y)\cdot\nabla 
\rho_2+\frac12\k_0^2 h(\xi)\po_{\xi}\rho_2
-\frac12\k_0^2\po_{\xi}(\s^2\po_\xi \rho_2)=\po_\xi m_2
-\k_0\s\po_\xi \rho_2\frac{dW(t)}{dt},
\end{equation} 
and we see from \eqref{e2.21} that $\rho_2$ verifies 
\begin{equation}\label{e2.23'}
\underset{t\to0+}{\esslim}
\int_{\R^{d+1}} \rho_2(t,x,\xi) \z(\xi)\phi(x)\,d\xi\,dx
=\int_{\R^{d+1}} 
\En_{(-\infty,U_0(x,y))}(\xi)
\z(\xi)\phi(x)\,d\xi\,dx,
\end{equation}
for all $\z\in C_c^\infty(\R)$ 
and all $\phi\in C_c^\infty(\R^d)$. Since $\rho_2$ is a standard kinetic function, a well known argument shows that the convergence in \eqref{e2.23'} may be strengthen to a strong convergence in $L^1(\R^d;w_N)$ (see also 
Remark~\ref{rem:init-measure}).

Again, by the definition of the measure $m_2(t,x,\xi)$ 
from \eqref{e2.21}, we may check that $m_2$ 
satisfies the conditions of a kinetic measure in Definition~\ref{def:kinetic-measure}. 
Also, $\rho_2(t,x,\xi)$ is trivially a 
generalized kinetic function whose associated 
Young measure  $\d_{U(t,x,y)}$, 
satisfies \eqref{eq:kinetic-Lp-bound-Young}, 
which can be verified without difficulty 
using the bounds \eqref{e2.4''}.   

Due to \eqref{e2.19}-\eqref{e2.20} 
and \eqref{e2.22'}-\eqref{e2.23'} 
and the properties satisfied by $m_1$, $m_2$ 
and $\mu_{t,x}^1=\nu_{t,x,y}$, 
$\mu_{t,x}^2=\d_{U(t,x,y)}$ (see, in particular, Proposition~\ref{P:1.1}(3)), we can apply 
Proposition~\ref{prop:scl-rigidity} 
together with the well-posedness result 
in Theorem~\ref{thm:scl-L1stab} (see 
also \eqref{eq:scl-varrho_1(1-varrho_2)}
and discussion at the beginning 
of Section \ref{S:5}) to deduce that  
\begin{multline*}
\EE\int_{\R^d}\int_{\R} 
\rho_1(t,x,\xi)\left(1-\rho_2(t,x,\xi)\right)
w_N\,dx\,d\xi \\ \le C(T) \int_{\R^d}
\int_{\R} \rho_1(0,x,\xi)
\left(1-\rho_2(0,x,\xi)\right)w_N\,dx\,d\xi=0,
\end{multline*}
for $0<t\le T$, $w_N$ given by \eqref{eq:scl-weight-def}, 
and so $\rho_1(t,x,\xi)=\rho_2(t,x,\xi)$, 
a.e.\ in $\Om\X(0,\infty)\X\R^d\X\R$. 
Clearly, this implies that 
\begin{equation*}
 \nu_{t,x,y}= \d_{U(t,x,y)},
\end{equation*} 
a.e.\ in $\Om\X(0,\infty)\X\R^d\X\cK$. 
In particular, due to the uniqueness of the limit, 
we deduce that the whole sequence $u^\ve(t,x)$ satisfies
\begin{equation*}
u^\ve\wto u(t,x):=\int_{\cK}U(t,x,y)\,d\mm(y), 
\end{equation*}
in the weak topology of $L^2(\Om;L_\loc^2((0,T)\X\R^d))$, for each $T>0$. Indeed, if this is not the case, then there would be a sequence $\ve_j\to 0$, a test function $\psi\in L^2(\Om;L^2((0,T)\times \mathbb{R}^d))$ and a constant $\alpha>0$  such that 
\begin{equation}\label{e3.1000}
\left| \EE \int_Q u^{\ve_j}\, \psi dx\, dt - \EE \int_Q u\, \psi dx\, dt\right|>\alpha,\quad \text{for all $j\in\N$.}
\end{equation}
However, by the procedure above, there is a further subsequence $\ve_{j_k}$ for which $u^{\ve_{j_k}}$ generates a Young measure $\nu_{\om,t,x,y}$ that turns out to be equal to $\delta_{U(\om,t,x,y)}$, which by Remark~\ref{R:2.1}, contradicts \eqref{e3.1000}. 

Finally, concerning the last assertion in Theorem~\ref{T:1.2}, it follow  directly from Proposition~\ref{P:1.2}~(b). This concludes the proof of Theorem \ref{T:1.2}.

\section{Stiff oscillatory external force, 
proof of Theorem \ref{T:1.1}}\label{S:3}

In this section we prove Theorem~\ref{T:1.1}.  
For the convenience of the reader, we rewrite 
here the formulas related to the homogenization 
problem  for \eqref{e1.1}, beginning with \eqref{e1.1} itself
\begin{equation*}
du^\ve+ \nabla_{x}\cdot f(u^\ve)\,dt
=\frac1{\ve}V'\left(\frac{x_1}{\ve}\right)\,dt
+ \k_0\,\s_{f_1}(u^\ve)\,dW+ \frac12\k_0^2\,h_{f_1}(u^\ve)\,dt,
\end{equation*}
where $f=(f_1,\ldots,f_d)$, $f_i:\R\to\R$ are smooth functions, 
$i=1,\ldots,d$, $f_1'\ge \delta_0>0$, $f_k'\ge 0$, 
$k=2,\ldots,d$,  We also assume that
$f'\in L^\infty(\R;\R^d)$ and $f_1',f_1'',f_1'''\in L^\infty(\R)$.
$\k_0\in\R$ is a constant. 
$V:\R\to\R$ is a smooth 
function belonging to an arbitrary ergodic 
algebra $\A(\R^d)$, 
$\s_{f_1}$, $h_{f_1}$ are 
obtained from $f_1$ from the expressions
$$
\s_{f_1}(u):= \frac1{f_1'(u)},
\qquad h_{f_1}:=-\frac{f_1''(u)}{f_1'(u)^3}.
$$  
We observe that, from the assumptions 
on $f_1$,  it follows that $h_{f_1}'\in L^\infty(\R)$.

We recall that  $g=f_1^{-1}$ is the inverse of $f_1$. 
We assume that, for some $v_0\in L^\infty(\R^d)$, 
the initial data $u_0\left(x,\frac{x}{\ve}\right)$ 
in \eqref{e1.2} satisfy
\begin{equation*}
	u_0(x,y)=g(V(y)+v_0(x)).
\end{equation*}

We recall the auxiliary equation \eqref{e1.4}
\begin{equation*}
	d\bar u+\nabla\cdot \bar f(\bar u)\,dt
	=\k_0\,\s_{\bar f_1}(\bar u)\,dW 
	+ \frac12 \k_0^2\,h_{\bar f_1}( \bar u)\,dt,
\end{equation*}
where $\bar f=(\bar f_1,\bar f_2,\ldots,\bar f_d)$, with 
$\bar f_1,\bar f_2,\ldots, \bar f_d$, 
satisfying \eqref{e1.5}, \eqref{e1.6} which we recall here
\begin{align*}
p &=\Medint_{\R}g\left(\bar f_1(p)+V(z_1)\right)\,dz_1, \\
\bar f_k(p)&=\Medint_{\R}f_k\circ 
g\left(\bar f_1(p)+V(z_1)\right)\,dz_1,  
\quad k=2,\ldots, d, 
\end{align*}
and $\s_{\bar f_1}(\cdot), \ h_{\bar f_1}(\cdot)$ 
are defined as $\s_{f_1},\ h_{f_1}$ 
with $\bar f_1(\cdot)$ instead of $f_1$
We recall that, from the assumptions on $f$ 
and $f_1$, it follows from \eqref{e1.5} 
and \eqref{e1.6} that $\bar f$ and 
$\bar f_1$ also satisfy $\bar f'\in L^\infty(\R;\R^d)$ 
and $\bar f_1',\bar f_1'',\bar f_1'''\in L^\infty(\R)$.

We recall that for \eqref{e1.4} we have 
prescribed the initial condition 
\begin{equation*}
	\bar u(0,x)=\bar u_0(x):=
	\Medint_{\R}u_0(x,z_1)\,dz_1=\bar f_1^{-1}(v_0(x)).
\end{equation*} 

We begin the proof of Theorem~\ref{T:1.1} 
by observing that \eqref{e1.1} 
admits special solutions of the form
\begin{equation}\label{e1.3}
\psi_\a\left(t,\frac{x_1}{\ve}\right):=
g\left(V\left(\frac{x_1}{\ve}\right)+\k_0 W(t)+\a\right),
\end{equation}
where  $\a\in\R$, as a consequence of It\^o's formula.  

The equation \eqref{e1.4} has the following special solutions
\begin{equation}\label{e1.7}
\psi_{*\g}(t):= \bar g\left(\g+\k_0 W(t)\right),
\end{equation}
where $\bar g(\cdot):=\bar f_1^{-1}(\cdot)$, the inverse 
function of $\bar f_1(\cdot)$, that is,
$$
\psi_{*\g}(t)=
\Medint_{\R}g\left(\g+\k_0 W(t)+V(z_1)\right)\,dz_1.
$$
 
By the stochastic Kru{\v {z}}kov 
inequality, cf. Proposition~\ref{p:Kato-Kruzkov}, we get a.s.
\begin{multline}\label{e1.8} 
 \int_0^\infty\int_{\R^d}
\biggl\{ \abs{u^\ve-\psi_\a\left(t,\frac{x_1}{\ve}\right)}
\phi_t+\abs{f_1(u^\ve)
-f_1\left(\psi_\a\left(t,\frac{x_1}{\ve}\right)\right)} 
\phi_{x_1} \biggr. \\
\biggl. 
+\sum_{k=2}^d \abs{f_k(u^\ve)
-f_k\left(\psi_\a\left(t,\frac{x_1}{\ve}\right)\right)}
\phi_{x_k} +\frac12\k_0^2 S_{u^\ve,\psi_\a^\ve} \biggl(h_{f_1}(u^\ve)
-h_{f_1}\left(\psi_\a\left(t,\frac{x_1}{\ve}\right)\right)\biggr)
\phi\biggr\}\,dx\,dt \\
+\int_0^\infty \int_{\R^d} \k_0 \abs{\s_{f_1}(u^\ve)-\s_{f_1}\left(\psi_\a\left(t,\frac{x_1}{\ve}\right)\right)}
\phi\,dx\,dW(t) \\
+ \int_{\R^d} \abs{u_0^\ve-\psi_\a^\ve}\phi(0,x)\,dx\ge 0,
\end{multline}
where, $\psi_\a^\ve:=\psi_\a\left(t,\frac{x_1}{\ve}\right)$, 
$S_{a,b}=\sgn(a-b)$, as in the last section, and we use the fact that 
$f_1, f_2,\cdots, f_d, \s_{f_1}$ are monotone increasing. 
A similar inequality holds with $(\cdot-\cdot)_\pm$ instead 
of $|\cdot-\cdot|$, which easily follows by adding (subtracting)
to \eqref{e1.8}  the difference 
of integral equations defining weak solutions 
for $u^\ve(t,x)$ and for $\psi_\a\left(t,\frac{x_1}{\ve}\right)$.  
Let $w_N$ be defined as in \eqref{eq:scl-weight-def}.  
In particular, from \eqref{e1.8} it follows 
the comparison principle
\begin{equation*}
\EE \int_{\R^d} 
\left(u^\ve(t,x)-\psi_\a(t,\frac{x_1}\ve)\right)_\pm w_N
\,dx\le e^{Ct}\int_{\R^d} 
\left(u_0(x, \frac{x_1}{\ve})-\psi_{\a}
\left(0,\frac{x_1}{\ve}\right)\right)_\pm w_N\,dx,
\end{equation*}
for some $C>0$. 

Thus, if $\a_1,\a_2\in\R$ are such that 
$$
\psi_{\a_1}\left(0,\frac{x_1}{\ve}\right)
\le u_0(x,\frac{x_1}{\ve}) 
\le \psi_{\a_2}\left(0,\frac{x_1}{\ve}\right),
$$
we obtain the following, which provide bounds 
for $u^\ve$ independent of $\ve$: 
\begin{equation}\label{e1.8''}
\psi_{\a_1}\left(t,\frac{x_1}{\ve}\right)\le u^\ve(t, x) 
\le\psi_{\a_2}\left(t,\frac{x_1}{\ve}\right).
\end{equation}
 
Taking, in \eqref{e1.8}, $\phi(t,x)= 
\ve \varphi\left(\frac{x_1}{\ve}\right)\psi(t,x)$, 
where $\varphi, \varphi' \in \cA(\R)$, $\varphi\ge0$ 
and $0\le \psi\in C_c^\infty((0,\infty)\X \R^d )$, taking conditional expectation with respect to an arbitrary $A\in\F$, 
and letting $\ve\to0$, along a subsequence 
for which $u^\ve$ generates a two-scale 
Young measure $\nu_{\om,t,x,y}$ (see Proposition~\ref{P:1.1}), we get, a.s.,  where we again drop the subscript $\om$ from $\nu_{\om,t,x,y}$,
$$
\int_0^\infty \int_{\R^d} \int_{\cK}\psi(t,x)\la \nu_{t,x,y}, 
\abs{f_1(\l)-f_1(\psi_\a(t,y))}
\ra\,\varphi'(y)\,d\mm(y)\,dx\,dt\ge 0,
$$
where $\cK$ denotes the compactification of $\R^d$ 
generated by $\A(\R^d)$, whose invariant measure 
associated with the mean-value is denoted  by $d\mm(y)$.
 
Applying this inequality to $C\pm \varphi$, 
with $C=\|\varphi\|_\infty$, we obtain, a.s., 
\begin{equation}\label{e1.9}
 \int_0^\infty\int_{\R^d}\int_{\cK} 
\psi(t,x)\la \nu_{t,x,y},
\abs{f_1(\l)-f_1(\psi_\a(t,y))}
\ra\varphi'(y)\, d\mm(y)\,dx\,dt=0.
\end{equation}
 
We define, similarly to \cite{ES, AF}, the family 
of parameterized measures $\mu_{t,x,y}$ over $\R$ by
$$
\la \mu_{t,x,y},\theta\ra:= 
\la \nu_{t,x,y},\theta(f_1(\l)-\k_0 W(t)-V(y)) \ra,
\quad \text{for $\theta\in C_c(\R)$}.
$$
We see from \eqref{e1.9}  that $\mu_{t,x,y}$ 
actually does not depend on $y\in\cK$, since
\begin{equation}\label{e1.10}
\int_0^\infty\int_{\R^d}\int_{\cK} \psi(t,x)\la 
 \mu_{t,x,y},\theta\ra\varphi'(y)\, d\mm(y)\,dx\,dt=0,
\end{equation}
for all $\theta$ of the form $|\cdot-\a|$, $\a\in\R$, 
and, from the remark made just after \eqref{e1.8}, 
also for $\theta(\cdot)=(\cdot -\a)_\pm$, $\a\in\R$, 
which implies that \eqref{e1.10} holds for all $\theta\in C(\R)$.  

Now, taking any nonnegative $\phi\in C_c^1(\R^{d+1})$ 
in \eqref{e1.8}, taking conditional expectation with respect to an arbitrary $A\in\F$, and making $\ve\to0$ along a subsequence as above, given by Proposition~\ref{P:1.1}, we get a.s. 
\begin{multline}\label{e1.11} 
 \int_0^\infty\int_{\R^d}\int_{\cK}
\biggl\{\la\nu_{t,x,y},  \abs{\l-\psi_\a(t,y)}\ra\phi_t  
+\la\nu_{t,x,y}, \abs{f_1(\l)-f_1(\psi_\a(t,y) )}\ra\phi_{x_1} \biggr. 
\\ \biggl.  
+\sum_{k=2}^d \la\nu_{t,x,y}, \abs{f^k(\l)-f^k(\psi_\a(t,y))}\ra \phi_{x_k} 
+\frac12 \k_0^2 \la\nu_{t,x,y}, S_{\l,\psi_\a}\left(h_{f_1}(\l)-h_{f_1}(\psi_\a(t,y))\right)\ra \phi\biggr\}\\
\X\,d\mm(y)\, dx\,dt \\ 
+\int_0^\infty \int_{\R^d}\int_{\cK} \k_0 \la\nu_{t,x,y},\left|\s_{f_1}(\l)-\s_{f_1}(\psi_\a(t,y))\right|\ra\phi\,d\mm(y)\,dx\,dW(t) \\
+ \int_{\R^d}\int_{\cK} 
\abs{u_0(x,y)-\psi_\a(0,y)}\phi(0,x)\,d\mm(y)\, dx\ge 0.
\end{multline}
 
Due to \eqref{e1.10} we can write $\mu_{t,x,y} =\mu_{t,x}$. 
Then, using the substitution formulas 
$\l=g(\rho+\k_0 W(t)+V(y))$, 
$\psi_\a(t,y)=g\left(\a+\k_0 W(t)+V(y)\right)$, 
we can rewrite \eqref{e1.11} a.s.\ as 
\begin{multline*}
 \int_0^\infty\int_{\R^d}\biggl\{\la\mu_{t,x},  
\int_{\cK}\abs{g\left(\cdot+\k_0 W(t)+V(y)\right)
-g\left(\a+\k_0 W(t)+V(y)\right)}\,\mm(y)\ra\phi_t \biggr.
\\  +\sum_{k=1}^d \la\mu_{t,x}, 
\int_{\cK}\abs{p_k(\cdot +\k_0 W(t)+V(y))
-p_k(\a+\k_0W(t)+V(y))}\,d\mm(y)\ra \phi_{x_k} 
\\ \biggl. 
+ \frac12\k_0^2\la\mu_{t,x},
\int_{\cK} S_{\cdot,\a} \left(h_{f_1}\circ 
g\left(\cdot+\k_0W(t)+V(y)\right)
-h_{f_1}\circ g\left(\a+\k_0 W(t)+V(y)\right)\right)\,d\mm(y)
\ra \phi\biggr\} \, dx\,dt \\
+\int_0^\infty \int_{\R^d} \k_0 \Big\langle \mu_{t,x,y},\int_{\cK}  
\left|\s_{f_1}\circ g(\cdot+\k_0W(t)+V(y))-\s_{f_1}
\circ g(\a+\k_0 W(t)+V(y))\right|\,d\mm(y) \Big\rangle\\ 
\X \phi \,dx\,dW(t) \\
+\int_{\R^d}\int_{\cK} 
\abs{u_0(x,y)-g(\a+V(y))}\phi(0,x)\,d\mm(y)\, dx\ge 0,
\end{multline*}
where $p_k=f_k\circ g$ (so $p_1(t)=t$), 
$S_{a,b}=\sgn(a-b)=\sgn(g(\cdot+\k_0 W(t)+V(y))
-g(\a+\k_0W(t)+V(y)))$,  from which it follows
\begin{multline}\label{e1.12'}
 \int_0^\infty\int_{\R^d}\biggl\{\la\mu_{t,x},  
\int_{\cK} \left(g\left(\cdot+\k_0 W(t)+V(y)\right)
-g\left(\a+\k_0 W(t)+V(y)\right)\right)_+
\,d\mm(y) \ra\phi_t \biggr.
\\ +\sum_{k=1}^d \la\mu_{t,x}, 
\int_{\cK}\left(p_k(\cdot +\k_0 W(t)+V(y))
-p_k(\a+\k_0W(t)+V(y))\right)_+d\mm(y)\ra \phi_{x_k} 
\\ \biggl. 
+ \frac12\k_0^2\la\mu_{t,x},
\int_{\cK}S_{\cdot,\a,+} \left(h_{f_1}
\circ g\left(\cdot+\k_0W(t)+V(y)\right)
-h_{f_1}\circ g\left(\a+\k_0 W(t)+V(y)\right)\right)\, d\mm(y)
\ra \phi\biggr\} \, dx\,dt \\
+\int_0^\infty \int_{\R^d} \k_0 \Big\langle \mu_{t,x},\int_{\cK}  
 \left(\s_{f_1}\circ g(\cdot+\k_0W(t)+V(y))
 -\s_{f_1}\circ g(\a+\k_0 W(t)+V(y))\right)_+
 \,d\mm(y)\Big\rangle\\ \X \phi\,dx\,dW(t) \\
+ \int_{\R^d}\int_{\cK} 
\left(u_0(x,y)-g(\a+V(y))\right)_+\,d\mm(y) \phi(0,x)\, dx\ge 0,
\end{multline}
where $S_{a,b,+}:=(a-b)_+
=\sgn_+(g(a+\k_0 W(t)+V(y))-g(b+\k_0W(t)+V(y)))_+$. 

Note that from the formulas for $\s_{\bar f_1}$ and $h_{\bar f_1}$, recalled in the beginning of this section, we may
verify, and this seems a little miraculous(!), 
the equations
\begin{equation*}
\begin{aligned}
&\s_{\bar f_1}\circ \bar g(v)=\int_{\cK} 
\s_{f_1}\circ g(v+V(y))\,d\mm(y),\\
& h_{\bar f_1}\circ \bar g(v)=\int_{\cK} 
h_{f_1}\circ g(v+V(y))\, d\mm(y).
\end{aligned}
\end{equation*}

Also, by the monotonicity of $g$, $p_k$, $k=2,\ldots,d$ 
and $\s_{f_1}$ we can pass the integral 
over $\mathcal{K}$ inside of the positive part 
in each term of \eqref{e1.12'}. Thus, recalling 
the definition of $\bar f_k$, $k=1,\ldots,d$ 
and $\psi_{*\a}$, we obtain
\begin{multline*}
 \int_0^\infty\int_{\R^d}\biggl\{\la\mu_{t,x},  
\left(\bar g\left(\cdot+\k_0 W(t)\right)
-\bar g(\alpha+\k_0W(t))\right)_+ \ra\phi_t \biggr.
\\ +\sum_{k=1}^d \la\mu_{t,x}, 
\left(\bar f_k\circ\bar g(\cdot +\k_0 W(t))
-\bar f_k\circ\bar g(\a + \k_0W(t))\right)_+ \ra \phi_{x_k} 
\\ \biggl. +\la\mu_{t,x},
\frac{1}{2} \k_0^2 S_{\cdot,\a,+} \left(h_{\bar f_1}\circ \bar g\left(\cdot+\k_0W(t)\right)
- h_{\bar f_1}\circ\bar g\left(\a+\k_0 W(t)\right)\right)\, 
\ra \phi\biggr\} \, dx\,dt \\
+\int_0^\infty \int_{\R^d} \k_0 \Big\langle \mu_{t,x}, 
\left(\s_{\bar f_1}\circ 
\bar g(\cdot+\k_0W(t))-\s_{\bar f_1}\circ 
\bar g(\a+\k_0 W(t))\right)_+\,\Big\rangle \phi\,dx\,dW(t) \\
+\int_{\R^d}\left(\bar g(v_0(x))-\bar g(\a)\right)_+
\, \phi(0,x)\, dx\ge 0.
\end{multline*}

Given $\vartheta \in C_c^\infty(\R^{d+1})$ 
and $\tilde \varphi\in C_c^\infty(\R)$, we define the measure 
$m_1=m_1(t,x,\xi)$ applied to $\vartheta \tilde \varphi$ by
\begin{multline}\label{e1.12'''}
\la m_1,\vartheta \tilde\varphi\ra 
:= \int_0^\infty\int_{\R^d}\int_\R \biggl\{\la\mu_{t,x},  
\left(\bar g\left(\cdot+\k_0 W(t)\right)
-\bar g(\xi+\k_0W(t))\right)_+ \ra\vartheta_t  \\
 +\sum_{k=1}^d \la\mu_{t,x}, 
\left(\bar f_k\circ\bar g(\cdot +\k_0 W(t))
-\bar f_k\circ\bar g(\xi + \k_0W(t))\right)_+ \ra \vartheta_{x_k} 
\\ 
\biggl.
 + \la\mu_{t,x},
\frac{1}{2} \k_0^2 S_{\cdot,\a,+} \left(h_{\bar f_1}\circ \bar g\left(\cdot+\k_0W(t)\right)
- h_{\bar f_1}\circ\bar g\left(\xi+\k_0 W(t)\right)\right)\, 
\ra \vartheta\biggr\} \tilde\varphi(\xi)\,d\xi\, dx\,dt \biggr. \\
+\int_0^\infty \int_{\R^d}\int_\R \k_0 \Big\langle \mu_{t,x}, 
 \left(\s_{\bar f_1}\circ \bar g(\cdot+\k_0W(t))-\s_{\bar f_1}\circ \bar g(\xi+\k_0 W(t))\right)_+\,\Big\rangle  \vartheta \tilde\varphi(\xi) 
   \,d\xi\,dx\,dW(t) \\
+ \int_{\R^d}\int_\R \left(\bar g(v_0(x))-\bar g(\xi)\right)_+ \vartheta(0,x)\tilde\varphi(\xi)\,d\xi\, dx.
\end{multline}
We then take $\tilde \varphi=\varphi'$, 
for some $\varphi\in C_c^\infty(\R)$ 
and make an integration by parts in the integral in $\xi$.  
Hence, defining $\rho_1(t,x,\xi):=\mu_{t,x}((\xi,+\infty))$, 
$a_0(\xi)=\bar g'(\xi)$, $a_i(\xi)=(\bar f_i\circ \bar g)'(\xi)$, $i=1,\ldots,d$, $H(\xi):= (h_{\bar f_1}
\circ \bar g)'(\xi)$, $G(\xi)
:=(\s_{\bar f_1}\circ \bar g)'(\xi)$,  
setting $a:=( a_1,\ldots, a_d)$, 
we get from \eqref{e1.12'''}
\begin{multline}\label{e1.12iv}
\la \po_{\xi} m_1,\vartheta \varphi\ra 
= \int_0^\infty\int_{\R^d}\int_\R \biggl\{a_0(\xi+\k_0W(t))  
\rho_1(t,x,\xi) \vartheta_t \\
+\sum_{k=1}^d a_k(\xi +\k_0 W(t))\rho_1(t,x,\xi) \vartheta_{x_k} 
\biggl.
 + \frac12\k_0^2 H(\xi+\k_0W(t))\rho_1(t,x,\xi) \vartheta\biggr\}  \varphi(\xi)\,d\xi\, dx\,dt \\
+\int_0^\infty \int_{\R^d}\int_\R \k_0 
 G(\xi+\k_0W(t))\rho_1(t,x,\xi)  
 \vartheta \varphi(\xi)  \,d\xi\,dx\,dW(t) \\
+ \int_{\R^d}\int_\R  
\bar{g}'(\xi)\En_{\xi<v_0(x)} \vartheta(0,x)\varphi(\xi)\,d\xi\, dx.
\end{multline}
Therefore, we see that  $\rho_1$ is a 
weak solution of the stochastic kinetic equation
\begin{multline}\label{e1.12v}
\po_t (a_0(\xi+\k_0W(t))\rho_1)+a(\xi+\k_0W(t))\cdot\nabla_x\rho_1 -\frac12\k_0^2 H(\xi+\k_0W(t))\rho_1
\\ =\po_\xi m_1+\k_0 G(\xi+\k_0W(t)) \rho_1\frac{dW(t)}{dt},
\end{multline}
in the sense of \eqref{e1.12iv} extended 
from test functions of the form $\vartheta\varphi$ 
to all test functions in $C_c^\infty(\R^{d+2})$. 
Also, from \eqref{e1.12iv}, it follows that 
\begin{equation}\label{e1.12vi}
\underset{t\to0+}{\esslim}
\int_{\R^d}\int_{\R}\rho_1(t,x,\xi)\phi(x,\xi)\,dx\,d\xi
= \int_{\R^d}\int_{\R} 
 \En_{\xi<v_0(x)}\phi(x,\xi)\,dx\,d\xi,
\end{equation}
for all $\phi\in C_c^\infty(\R^{d+1})$, see also 
Remark \ref{rem:init-measure}. 
We observe that $\rho_1$ is a kinetic 
function associated to the Young measure $\mu_{t,x}$. 
Also, it is not difficult to check, by Proposition~\ref{P:1.1}(3),  that $\mu$ 
satisfies \eqref{eq:kinetic-Lp-bound-Young} 
and we also may check that $m_1$, 
defined by \eqref{e1.12'''}, satisfies the conditions 
of a kinetic measure in Definition~\ref{def:kinetic-measure}.

\medskip
  
On the other hand, using the here called 
stochastic Kru{\v{z}}kov inequality (see Proposition~\ref{p:Kato-Kruzkov})
for \eqref{e1.4} for the entropy solution of 
\eqref{e1.4}-\eqref{e1.4'} and for the special 
solution $\psi_{*\g}(t)$,  we get
\begin{multline}\label{e1.16}
 \int_0^\infty\int_{\R^d} 
\biggl\{ \abs{\psi_{*\g}(t)-\bar u}\phi_t
+\sum_{k=1}^d \abs{\bar f_k(\psi_{*\g}(t))-\bar f_k(\bar u)}
\phi_{x_k}\biggr.\\ \biggl. 
+ \frac12\k_0^2 S_{\psi_{*\gamma}(t),\bar u }
\left(h_{\bar f_1}(\psi_{*\g})-h_{\bar f_1}(\bar u)\right)
\phi\biggr\}\,dx\,dt 
+\int_{\R^d}\abs{\psi_{*\g}(0)-\bar u(0,x)}\phi(0,x)\,dx\\
+\int_0^\infty \int_{\R^d} 
\k_0\left|\s_{\bar f_1}(\bar u)
-\s_{\bar f_1}(\psi_{*\gamma}(t))\right|\phi\,dx\,dW(t) \ge0,
\end{multline}
for all $\phi\in C_c^\infty(\R^{d+1})$. 

Let $X(t,x)=\bar f_1(\bar u(t,x))-\k_0W(t)$ 
and observe $\bar u(t,x)=\bar g(X(t,x)+\k_0W(t))$. 
We then get from \eqref{e1.16} as before,  
\begin{multline*}
\int_0^\infty\int_{\R^d}\biggl\{\left(\bar g(\g+\k_0W(t)) 
-\bar g(X(t,x)+\k_0W(t)\right)_+\phi_t\\
+\sum_{k=1}^d \left( \bar f_k\circ 
\bar g( \g+\k_0 W(t))- \bar f_k\circ \bar g(X(t,x)+\k_0W(t))\right)_+
\phi_{x_k}\biggr\} \, dt\,dx \\ 
+ \int_0^\infty\int_{\R^d} 
\frac12\k_0^2 S_{\g,X,+}\left(\bar h_{ f_1}\circ \bar g(\g+\k_0W(t))-\bar 
h_{f_1}\circ \bar g(X(t,x)+\k_0 W(t))\right)\phi \,dx\,dt \\
+\int_{\R^d}\left(\bar g(\g+V(y))
-\bar g(v_0(x)+V(y))\right)_+\phi(0,x))\,dx\\
+\int_0^\infty  \k_0 \int_{\R^d} 
\left(\s_{\bar f_1}\circ 
\bar g(\g+\k_0W(t))-\s_{\bar f_1}
\circ \bar g(X(t,x)+\k_0 W(t))\right)_+\phi\,dx\,dW(t) \ge0.
 \end{multline*}
Hence, given $\vartheta \in C_c^\infty(\R^{d+1})$, 
$\tilde \varphi\in C_c^\infty(\R)$ we can 
similarly define the measure $m_2=m_2(t,x,\xi)$ applied
to $\vartheta \tilde \varphi$ by
\begin{multline}\label{e1.16''}
\la m_2, \psi\tilde \varphi\ra 
:= \int_0^\infty\int_{\R^d}\int_{\R} 
\biggl\{ \left(\bar g(\g+\k_0W(t))
-\bar g(X(t,x)+\k_0W(t)\right)_+\vartheta_t\tilde\varphi (\xi)\\
+\sum_{k=1}^d \left( \bar f_k\circ \bar g( \g+\k_0 W(t))
- \bar f_k\circ \bar g(X(t,x)+\k_0W(t))\right)_+\theta_{x_k}\tilde\varphi(\xi) \biggr\} \,d\xi\, dx \, dt \\ 
 +\frac12\k_0^2 \int_0^\infty\int_{\R^d}\int_\R \left(\bar h_{ f_1}\circ 
 \bar g(\g+\k_0W(t))-\bar h_{f_1}\circ \bar g(X(t,x)+\k_0 W(t))\right) \vartheta \tilde \varphi(\xi)   \,d\xi\,dx\,dt 
\\+\int_{\R^d}\int_{\R}\left(\bar g(\g+V(y))
- \bar g(v_0(x)+V(y))\right)_+\vartheta(0,x)
\tilde \varphi(\xi)\,d\xi\,dx\\
+\int_0^\infty  \k_0 \int_{\R^d}\int_{\cK} 
\left(\s_{\bar f_1}\circ \bar g(\g+\k_0W(t))-\s_{\bar f_1}\circ \bar g(X(t,x)+\k_0 W(t))\right)_+ \vartheta 
\tilde \varphi (\xi) \,d\xi \,dx \,dW(t),
 \end{multline}
Therefore, we again take $\tilde \varphi=\varphi'$ 
for some $\varphi\in C_c^\infty(\R)$ and 
make an integration by parts in the integral in $\xi$. 
Hence, defining $\rho_2(t,x,\xi):=\En_{(-\infty,X(t,x))}(\xi)$, 
we see that $\rho_2$ is a weak solution 
of the stochastic kinetic equation
\begin{multline}\label{e1.16iv}
\po_t(a_0(\xi+\k_0W(t)) \rho_2)+a(\xi+\k_0W(t))
\cdot\nabla_x\rho_2 -\frac12\k_0^2 H(\xi+\k_0W(t))\rho_2
\\=\po_\xi m_2+\k_0 G(\xi+\k_0W(t)) \rho_2\frac{dW(t)}{dt},
\end{multline}
where $a_i$, $i=0,\ldots,d$, $H$ and $G$ are as before. 
Also, from \eqref{e1.16''}, it follows that 
\begin{equation}\label{e1.16v}
\underset{t\to0+}{\esslim}
\int_{\R^d}\int_{\R}
\rho_2(t,x,\xi)\phi(x,\xi)\,dx\,d\xi= 
\int_{\R^d}\int_{\R} \En_{\xi<v_0(x)}
\phi(x,\xi)\,dx\,d\xi,
\end{equation}
for all $\phi\in C_c^\infty(\R^{d+1})$. Since $\rho_2$ is a standard kinetic function, a well known argument shows that the convergence in \eqref{e1.16v} may be strengthen to a convergence in $L^1(\R^N;w_N)$ (see also 
Remark~\ref{rem:init-measure}). 

Therefore, $\rho_1$, $\rho_2$ 
are weak solutions of identical kinetic 
equations, \eqref{e1.12v} and \eqref{e1.16iv}, 
with possibly distinct kinetic measures $m_1$ 
and $m_2$, and satisfy identical 
initial conditions \eqref{e1.12vi} and \eqref{e1.16v}. 

Our next goal is to prove that $\mu_{t,x}=\d_{X(t,x)}$ 
a.s.~and to do that we are going to prove the uniqueness 
of the weak solution of  \eqref{e1.12v}-\eqref{e1.12vi} 
or \eqref{e1.16iv}-\eqref{e1.16v}, independently 
of the corresponding kinetic measure.  
This, in turn, will be a consequence of the next lemma.

\begin{lemma}[rigidity/comparison result]\label{l3.1}
Let $\rho_1(t,x,\xi)$, $\rho_2(t,x,\xi)$ 
be generalized kinetic functions, that is, 
functions taking values in $[0,1]$ such 
that $-\partial_\xi \rho_1$ and 
$-\partial_\xi \rho_2$ are Young measures, 
which solve equations \eqref{e1.12v} 
and \eqref{e1.16iv} with initial conditions $\rho_{0,1}$ 
and $\rho_{0,2}$, respectively, 
where $m_1$ and $m_2$ are kinetic measures 
in the sense of Definition \ref{def:kinetic-measure}. 
Then there is a constant $C>0$ such that
\begin{equation}\label{e3.22}
\begin{split}
&\mathbb{E}\int_{\R^d}\int_\R \a_0(\xi+\k_0W(t)) 
\rho_1(1-\rho_2)(t)w_N\,d\xi\,dx 
\\ & \qquad \leq C \mathbb{E} \int_{\R^d}\int_\R 
a_0(\xi)\rho_{0,1}(1-\rho_{0,2})w_N\,d\xi\,dx,
\end{split}
\end{equation}
for a.e.~$t\in [0,T]$, where $w_N$ is the weight 
given by \eqref{eq:scl-weight-def}.
\end{lemma}

\begin{remark}\label{rem:new-kinetic-eqn}
Observe that the stochastic kinetic 
equations \eqref{e1.12v}, \eqref{e1.16iv} 
are different from the equations 
\eqref{e2.19}, \eqref{e2.22'} (which are 
of the type analyzed in Section \ref{S:4}). 
In particular, the former two equations do not 
have \textit{gradient} noise and 
a second order differential operator. 
They do, however, contain coefficients that 
are predictable random fields. We recall that 
a continuous mapping $H=H(\omega,t,x,u):\Omega\times [0,T]
\times \R^d\times \R\to \R$
is a called a random field when it is viewed as a 
random variable $\omega\mapsto H(\omega,t,x,u)$, with 
$(t,x,u)$ fixed. If, for each fixed $(x,u)$, the 
stochastic process $(\omega,t)\mapsto H(\omega,t,x,u)$ 
is $\seq{\cF_t}_{t\in [0,T]}$-predictable, then $H$ is called 
a predictable random field. Nevertheless, 
given the crucial observation below, 
cf.~\eqref{e1.pf2}, the analysis in Section \ref{S:4} 
carries over to the stochastic kinetic equations 
\eqref{e1.12v} and \eqref{e1.16iv}, 
see in particular Proposition \ref{prop:scl-rigidity}, 
the equation \eqref{eq:scl-S(varrho)-approx}, 
and \eqref{eq:scl-varrho_1(1-varrho_2)} 
and the discussion found at the beginning of Section \ref{S:5}. 

To keep this paper at a reasonable length, we 
will only supply a sketch of the proof of Lemma \ref {l3.1}, 
focusing on the formal argument leading up the crucial equation \eqref{e1.pf2}, from which we can proceed as in Section \ref{S:4}. 
The rigorous proof relies on the usual regularization 
procedure, the It\^{o} product formula, and commutator 
estimates to control regularization errors. 
In fact, the step involving regularization by convolution (in $x,\xi$) 
is simpler (than in Section \ref{S:4}) since 
there are no error terms that require 
second-order commutator estimates, like \eqref{eq:Qreg}, 
that is, all the error terms can be handled 
using the standard DiPerna-Lions folklore 
lemma \cite{DiPerna:1989aa}. 
We refer to Section \ref{S:4} for details.
\end{remark}

\begin{proof}[Sketch of proof of Lemma \ref{l3.1}]

We will first formally derive stochastic kinetic 
equations for $\rho_1$ and $\a_0(\xi+\k_0W(t))(1-\rho_2)$. 
Then, combining the resulting equations, It\^ o's 
product formula will provide (at least formally) 
an equation for $a_0(\xi+\k_0W(t)\rho_1(1-\rho_2)$, 
which we can use to prove \eqref{e3.22}, along 
the lines of Section \ref{S:4}. 

We observe that we can write \eqref{e1.12v} 
as a stochastic differential equation of the following 
form, where we drop the subscript $1$ 
in $\rho_1$ and $m_1$,
\begin{equation*}
	d(a_0(\xi+\k_0W(t))\rho)
	=A\,dt+B\,dW+\po_\xi m,
\end{equation*}
where
$$
A=-a(\xi+\k_0W(t))\cdot\nabla_x\rho
+\frac12\k_0^2 H(\xi+\k_0W(t))\rho,
\quad
B=\k_0 G(\xi+\k_0W(t)) \rho.
$$
Thanks to the It\^o formula, $a_0(\xi+\k_0W(t))$ satisfies 
the stochastic differential equation
\begin{equation}\label{e3.23}
da_0(\xi+\k_0W(t))=\k_0 a_0'(\xi+\k_0W(t))\,dW
+\frac12\k_0^2 a_0''(\xi+\k_0W(t))\,dt,
\end{equation}
and, by the formulas for $a_0$, $\s_{\bar f_1}$, 
and $h_{\bar f_1}$, we have 
$$
a_0(\xi)= \s_{\bar f_1}\circ \bar g(\xi),
\qquad a_0'=(\s_{\bar f_1}\circ \bar g)'(\xi),
\qquad a_0''=(h_{\bar f_1}\circ \bar g)'(\xi),
$$
so
\begin{equation}\label{e3.24}
a_0'(\xi)= G(\xi),\quad a_0''(\xi)=H(\xi).
\end{equation}
Denoting $\tilde a_0(\xi)=\frac1{a_0(\xi)}$ 
we get, also from It\^o's formula, the following 
stochastic differential equation for $\tilde a_0(\xi)$:
$$
d\tilde a_0(\xi+\k_0W(t))=
\k_0 \tilde a_0'(\xi+\k_0W(t))\,dW
+\frac12\k_0^2 \tilde a_0''(\xi+\k_0W(t))\,dt,
$$
where, by virtue of \eqref{e3.24},
\begin{equation}\label{e3.24'}
\tilde a_0'(\xi) =-\frac{G(\xi)}{a_0(\xi)^2}, \quad \tilde a_0''(\xi)=\frac{2G(\xi)^2 - a_0(\xi) H(\xi)}{a_0(\xi)^3}.
\end{equation}
By the It\^o product rule, 
\begin{multline*}
d\rho=d\bigl(\tilde a_0(\xi+\k_0W(t))a_0(\xi+\k_0W(t))\rho\bigr)\\
=a_0(\xi+\k_0W(t))\rho\, d(\tilde a_0(\xi+\k_0W(t)))
+\tilde a_0(\xi+\k_0W(t))\, d(a_0(\xi+\k_0W(t))\rho)\\
+\left[\tilde a_0(\xi+\k_0W(t)), a_0(\xi+\k_0W(t))\rho\right]\\
=\k_0 a_0(\xi+\k_0W(t))\tilde a_0'(\xi+\k_0W(t))\rho\, dW+\frac12 \k_0^2a_0(\xi+\k_0W(t))\tilde a_0''(\xi+\k_0W(t))\,\rho\,dt\\
+\tilde a_0(\xi+\k_0W(t))A\,dt+\tilde a_0(\xi+\k_0W(t)) B\,dW+\tilde a_0(\xi+\k_0W(t))\po_\xi m\\
+\k_0\tilde a_0'(\xi+\k_0W(t))B\,dt.
\end{multline*}
In sum, we deduce that $\rho_1$ satisfies the 
stochastic kinetic equation
\begin{multline*}
\po_t\rho_1+\tilde a_0(\xi+\k_0W(t))a(\xi+\k_0W(t))
\cdot\nabla_x\rho_1 \\
= \frac12\k_0^2\tilde A \rho_1\, 
+ \k_0 \tilde B\rho_1\, \frac{dW(t)}{dt}
+ \tilde a_0(\xi+\k_0W(t))\po_\xi m_1,
\end{multline*}
where
\begin{multline*}
\tilde A = \tilde a_0(\xi+\k_0W(t)) H(\xi+\k_0W(t))+
2\tilde a_0'(\xi+\k_0W(t)) G(\xi+\k_0W(t)) \\
+ a_0(\xi+\k_0W(t))\tilde a_0''(\xi+\k_0W(t)),
\end{multline*}
and
$$
\tilde B = a_0(\xi+\k_0W(t))
\tilde a_0'(\xi+\k_0W(t))
+\tilde a_0(\xi+\k_0W(t)) G(\xi+\k_0W(t)).
$$

At this point, we observe that \eqref{e3.24} 
and \eqref{e3.24'} imply that $\tilde A=0$ and $\tilde B=0$. 
Thus, we conclude that $\rho_1$ 
satisfies the (much simpler) equation
\begin{equation}\label{e1.pf2}
\po_t\rho_1+\tilde a_0(\xi+\k_0W(t))
a(\xi+\k_0W(t))\cdot\nabla_x\rho_1 
=  \tilde a_0(\xi+\k_0W(t))\po_\xi m_1.
\end{equation}

In view of \eqref{e1.16iv} and \eqref{e3.23},
we obtain the following equation 
for $a_0(\xi+\k_0W(t))(1-\rho_2)$:
\begin{multline}\label{e1.pf3}
\partial_t\left(a_0(\xi+\k_0W(t))(1-\rho_2)\right) 
+ a(\xi+\k_0W(t))\cdot\nabla_x(1-\rho_2) 
-\frac{1}{2}\k_0^2H(\xi+\k_0W(t))(1-\rho_2)\\
=-\partial_\xi m_2 
+ \k_0G(\xi+\k_0W(t))(1-\rho_2)\frac{dW(t)}{dt}.
\end{multline}

Given the stochastic kinetic 
equations \eqref{e1.pf2} and \eqref{e1.pf3}, 
we may apply (again formally) It\^ o's product rule to obtain
\begin{align*}
&d\left( a_0(\xi+\k_0W(t))\rho_1(1-\rho_2)\right)  \\
&\qquad=a_0(\xi+\k_0W(t))(1-\rho_2)\, d\rho_1 
+ \rho_1\, d\left( a_0(\xi+\k_0W(t))(1-\rho_2) \right)\\
&\qquad\qquad\qquad 
+ \left[ \rho_1, a_0(\xi+\k_0W(t))(1-\rho_2) \right]\\
&\qquad= -a(\xi+\k_0W(t))\cdot(1-\rho_2)\nabla_x\rho_1\, dt 
+ (1-\rho_2)\partial_\xi m_1 \\
&\qquad\qquad -a(\xi+\k_0W(t))\cdot\rho_1\nabla_x(1-\rho_2)\, dt
+ \frac{1}{2}\k_0^2H(\xi+\k_0W(t))\rho_1(1-\rho_2)\, dt\\
&\qquad\qquad\qquad-\rho_1\partial_\xi m_2 
+ \k_0G(\xi+\k_0W(t))\rho_1(1-\rho_2)dW(t).
\end{align*}

In other words, we have the following 
equation for $a_0(\xi+\k_0W(t))\rho_1(1-\rho_2)$:
\begin{multline}\label{e1.pf4}
\partial_t\left( a_0(\xi+\k_0W(t))\rho_1(1-\rho_2)\right) 
+ a(\xi+\k_0W(t))\cdot\nabla_x\left(\rho_1(1-\rho_2)\right)\\
=\frac{1}{2}\k_0^2H(\xi+\k_0W(t))\rho_1(1-\rho_2) 
+ (1-\rho_2)\partial_\xi m_1 - \rho_1\partial_\xi m_2\\
+  \k_0G(\xi+\k_0W(t))\rho_1(1-\rho_2)dW(t).
\end{multline}

Note that the coefficient $a_0(\xi+\k_0W(t))$ 
in the equation for $(1-\rho_2)$ provides a cancellation with the coefficient $\tilde a_0(\xi+\k_0W(t))$ which 
multiplies the measure $m_1$ on the right-hand-side of \eqref{e1.pf2}. 
This cancellation, which results in the 
term $(1-\rho_2)\partial_\xi m_2$ on the right-hand side 
of equation \eqref{e1.pf4}, is essential for 
the proof of Lemma~\ref{l3.1}, as it will allow us to 
discard this term later on in the analysis based on its sign, 
after integration by parts. Similarly, the 
term $-\rho_1 \partial_\xi m_2$, which is of the 
same nature, will also be discarded by its sign. 
To carry on the proof, we take appropriate test 
functions in the equation \eqref{e1.pf4} and 
manipulate the remaining terms to 
conclude by applying Gronwall's inequality. 

At last, we reiterate that the above argument 
can be turned into a rigorous proof using 
regularization by convolution (in $x,\xi$), 
following Section \ref{S:4}.

\end{proof}

Finally, in view of Lemma \ref{l3.1}, 
we deduce that
\begin{multline*}
\EE\int_{\R^d}\int_{\R}a_0(\xi+\k_0W(t)) 
\rho_1(t,x,\xi)\left(1-\rho_2(t,x,\xi)\right)w_N\,dx\,d\xi 
\\ \le C(T) \int_{\R^d}\int_{\R} \a_0(\xi)
\rho_{1,0}(x,\xi)\left(1-\rho_{2,0}(x,\xi)\right)
w_N\,dx\,d\xi=0,
\end{multline*}
for $0<t\le T$,  and so, since $\rho_1$ 
and $\rho_2$ coincide at $t=0$, both 
being equal to $\En_{\xi<v_0(x)}$, we obtain 
\begin{equation*}
\mu_{t,x}=\d_{\bar f_1(\bar u(t,x))-\k_0W(t)},
\end{equation*}
and consequently 
\begin{equation*}
\nu_{t,x,y}= \d_{g(\bar f_1(\bar u(t,x))+V(y))},
\quad
\text{$\bbP$-a.e.\ in $\Om$}.
\end{equation*}
In particular, it follows that, 
$u^\ve\wto \int_\cK g(\bar f_1(\bar u(\cdot,\cdot)+V(y)))d\mm(y) =\bar u$ in the weak--$\star$ 
topology of $L^2(\Om;L_\loc^2((0,T)\X\R^d))$, 
for all $T>0$ (cf. Remark~\ref{R:2.1} above). Note that we used the uniqueness of the limit to conclude that the whole sequence $u^\ve$ converges, similarly as in the proof of Theorem~\ref{T:1.2}. 


Again, the last assertion in Theorem~\ref{T:1.1} follows directly from Proposition~\ref{P:1.2}~(a). This concludes the proof of Theorem \ref{T:1.1}.

\section{A well-posedness result}\label{S:4}

In this section, we provide a well-posedness 
result for a class of stochastic conservation laws 
that is (more than) general enough to encompass some of the equations encountered earlier in this paper; namely, hyperbolic 
conservation laws with variable coefficients and 
deterministic/stochastic source terms, posed on an 
unbounded spatial domain ($\R^d$), see 
Remark \ref{rem:our-eqs} for further details on 
the class of equations. Since these equations are not all 
covered by the available well-posedness literature 
\cite{BVW1,BVW2,CDK,DV,DHV,Dotti:2018aa,FN,Ha,KSt,Kim,Kobayasi:2016aa,Lv:2016aa}, we will outline some of the arguments leading 
to this result, particularly the uniqueness part of it. 
On a technical level, the approach presented 
here is somewhat different from the one \cite{DV} 
utilized in many of the references listed above.

The initial--value problem for these SPDEs take the form
\begin{equation}\label{eq:scl}
	\begin{split}
		& \pt u + \Div_x A(t,x,u)
		=B(t,u)\dotW(t) + R(t,x,u),
		\quad (t,x)\in (0,T)\times \R^d,\\
		& u(0,x)=u_0(\omega,x), \quad x\in \R^d,
	\end{split}	
\end{equation}
where $W$ is a cylindrical Wiener 
process \cite{DPZ} with noise amplitude $B$, $A=(A_1,\ldots,A_d)$ 
is the flux vector, $R$ is the "deterministic" source term, $u_0$ is the initial function, and $T>0$ is a fixed final time. 
We fix a stochastic basic $\cS$ 
consisting of a complete probability space $(\Omega,\cF, P)$, 
a complete right-continuous 
filtration $\seq{\cF_t}_{t\in [0,T]}$, and a sequence 
$\set{W_k}_{k=1}^\infty$ of independent 
one-dimensional Wiener processes 
adapted to the filtration $\set{\cF_t}_{t\in [0,T]}$. 

We assume that the flux $A$ belongs to 
$C([0,T];C^2(\R^d\times\R;\R^d))$ and
\begin{equation}\label{eq:scl-flux-Lipschitz1}
	\begin{split}
		& \abs{A(t,x,u)}
		\le m_a(t)\left(1+\abs{u}\right)\left(1+\abs{x}\right),
		\\ & \abs{A(t,x,u)-A(t,x,v)} 
		\le m_a(t)\abs{u-v}\left(1+\abs{x}\right),
	\end{split}
\end{equation}
for $t\in [0,T]$, $x\in\R^d$, 
and $u,v\in\R$, where $m_a(t)$ is an integrable function. 
Moreover, 
\begin{equation}\label{eq:scl-flux-divx}
	\abs{(\Div_x A)(t,x,u)}
	\le m_d(t)\left(1+\abs{u}\right), 
	\qquad (\Div_x A)(t,x,0)=0,
\end{equation}
for $t\in [0,T]$, $x\in\R^d$, and $u\in\R$, where 
$m_d(t)$ is another integrable function.  Note that, without loss of generality, we 
may always assume $(\Div_x A)(t,x,0)=0$.

We assume that the source function $R$ belongs to 
$C([0,T];C^1(\R^d\times\R))$, and
\begin{equation}\label{eq:scl-reaction-Lipschitz1}
	\abs{R(t,x,u)}
	\le m_R(t)\left(1+\abs{u}\right),  	
	\quad 
	\abs{R(t,x,u)-R(t,x,v)}
	\le m_R(t)\abs{u-v},
\end{equation}
for $t\in [0,T]$, $x\in\R^d$, and $u,v\in\R$, where 
$m_R(t)$ is an integrable function. 

The driving noise $W$ is a cylindrical Wiener process \cite{DPZ},
\begin{equation}\label{eq:scl-cyl-Wiener1}
	W(t)=\sum_{k\geq 1}W_k(t)\psi_k,
\end{equation}
evolving over a separable Hilbert space $\fU$, 
equipped with an orthonormal basis $\seq{\psi_k}_{k\geq 1}$. 
The series \eqref{eq:scl-cyl-Wiener1} 
converges in an auxiliary (larger) Hilbert space $\fU_0$ with 
Hilbert-Schmidt embedding $\fU\subset \fU_0$. 
The (nonlinear) noise amplitude $B=B(\omega,t,u)$ is an 
operator-valued mapping. For each $u\in L^2(\R^d)$, 
we define $B(t,u)$ by its action on each $\psi_k$: 
\begin{equation*}
	B(t,u)\psi_k:=b_k(\omega,t,\cdot,u(\cdot)), 
	\qquad b_k\in C([0,T]\times\R^d\times\R), 
	\qquad k\in \N.
\end{equation*}
We then obtain
\begin{equation}\label{eq:scl-infinite-Wiener}
	B(t,u)\dW(t) = \sum_{k\ge1} b_k(t,x,u)\dW_k(t).
\end{equation}
We assume that the sequence $\seq{b_k}_{k\ge 1}$ 
satisfy the following conditions:
\begin{align}
	\label{eq:scl-def-noise-B}
	& B^2(t,x,u):=\sum_{k\geq 1}\left(b_k(t,x,u)\right)^2
	\lesssim 1+\abs{u}^2, 
	\\ &
	\label{eq:scl-noise-bk-reg}
	\sum_{k\geq 1}\abs{b_k(t,x,u)-b_k(t,y,v)}^2
	\lesssim \abs{x-y}^2+\abs{u-v}\mu(\abs{u-v}),
\end{align}
for $\omega\in \Omega$, $t\in [0,T]$, $x,y\in\R^d$, 
and $u,v\in\R$, for some continuous nondecreasing 
function $\mu$ on $\R_+$ with $\mu(0+)=0$. 
The "Lipschitz case" corresponds to $\mu(\xi)=\xi$.

\begin{remark}
We have assumed that the coefficients $A$, $B$, and 
$R$ in \eqref{eq:scl} are deterministic. 
However, this is not necessary. Indeed, the 
results presented in this section 
carry over to the case where $A,B,R$ are 
predictable random fields satisfying conditions 
similar to those listed above (cf.~Remark 
\ref{rem:new-kinetic-eqn} for the notion 
of predictable random field).
\end{remark}

The initial function $u_0$ is an $\cF_0$-measurable 
random variable satisfying 
\begin{equation}\label{eq:scl-init-ass}
	u_0\in L^\infty\left(\Omega;L^\infty(\R^d)\right).
\end{equation}

Given  a convex $S\in C^2(\R)$, define 
$Q_S:[0,T]\times \R\times\R \rightarrow \R^d$ by 
$(\partial_u Q_S)(t,x,u)=S'(u)(\partial_u A)(t,x,u)$. 
We call $(S,Q_S)$ an \emph{entropy/entropy-flux pair} 
and write $(S,Q_S)\in \mathscr{E}$. 
For \eqref{eq:scl} the entropy inequalities read 
{\small
\begin{equation}\label{eq:scl-entropy-ineq}
	\begin{split}
		&\pt S(u) + \Div_x Q_S(t,x,u) 
		+ S'(u)\left((\Div_x A)(t,x,u)-R(t,x,u)\right) 
		- (\Div_x Q_S)(t,x,u)
		\\ & \quad 
		\le \sum_{k\ge1} S'(u) b_k(t,x,u)\, \dotW_k(t)
		+ \frac12 S''(u) B^2(t,x,u)
		\quad \text{in $\Dp([0,T)\times \R^d)$, a.s., 
		$\forall (S,Q)\in \mathscr{E}$}.
	\end{split}
\end{equation}
}

\begin{remark}[weighted $L^p$ estimates]\label{rem:drop-the-weight}
For discontinuous solutions, the entropy inequalities 
act as a replacement for the It\^{o} (temporal) and classical (spatial) 
chain rules. It follows from \eqref{eq:scl-entropy-ineq} with 
$S(u)=u^p$ ($p\ge 2$) and a standard 
martingale argument that 
\begin{equation*}
	u\in L^p\left(\Omega;
	L^\infty\left(0,T;L^p(w_Ndx)\right)\right),
\end{equation*}
where $L^p(w_Ndx)$ denotes the weighted 
$L^p$ space of functions $v:\R^d\to \R$ for which 
\begin{equation*}
	\int_{\R^d} \abs{v}^p\, w_Ndx<\infty.
\end{equation*}
Throughout this section, 
we make use of the weight function 
\begin{equation}\label{eq:scl-weight-def}
	w_N(x) = (1 + \abs{x}^2)^{-N},
	\quad N>d/2.
\end{equation}
This function is integrable on $\R^d$ and satisfies 
\begin{equation*}
	\nabla w_N(x) = \frac{-2Nx}{1+\abs{x}^2}w_N(x)
	\quad \Longrightarrow \quad
	\abs{\nabla w_N(x)} \lesssim 
	\frac{w_N(x)}{1+\abs{x}}.
\end{equation*}
Note that $L^p(w_Ndx)$--bounds with $p\in [1,2)$ 
follow trivially from the $L^2(w_Ndx)$--bound. 
\end{remark}

\begin{remark}[weight-free framework]
The It\^{o} noise term continuously injects "entropy" into the system, 
cf.~the $S''B^2$--term in \eqref{eq:scl-entropy-ineq}. 
Suppose $B(t,x,0)=0$. Then the ordinary $L^p$ spaces constitute 
a natural choice for \eqref{eq:scl}, in 
which case we may drop the weight $w_N$ and obtain 
$u\in L^p\left(\Omega;L^\infty\left(0,T;L^p(\R^d)\right)\right)$ 
for all $p\in [2,\infty)$, provided 
\begin{equation}\label{eq:scl-init-ass-new}
	u_0\in L^\infty\left(\Omega;
	\left(L^2\cap L^\infty\right)(\R^d)\right).
\end{equation}
Without this assumption ($B(\omega,t,x,0)\neq 0$), 
weighted $L^p$ spaces appear to be better suited. 

\medskip\noindent
We can also drop the weight $w_N$ at the expense of imposing 
a stronger condition on $B^2$ as $\abs{x}\to \infty$, 
cf.~\eqref{eq:scl-def-noise-B}, namely that
\begin{equation}\label{eq:scl-def-noise-B-new}
	B^2(\omega,t,x,u)
	\leq \left(b(x)\right)^2\left(1+\abs{u}^2\right), 
	\qquad 
	b\in \left(L^2\cap L^\infty\right)(\R^d),
\end{equation}
for $\omega\in \Omega$, $t\in [0,T]$, $x\in\R^d$, 
and $u\in\R$. Under this assumption or $B(\omega,t,x,0)\equiv 0$, 
it is possible to use \eqref{eq:scl-entropy-ineq}, 
with $S(\cdot)\approx \abs{\cdot}$ and $S''(\cdot)\approx \delta(\cdot)$, 
to arrive at an $L^1$ bound, and consequently 
$u\in L^p\left(\Omega;L^\infty\left(0,T;L^p(\R^d)\right)\right)$ 
for all $p\in [1,\infty)$, in the event that $u_0\in 
L^\infty\left(\Omega;\left(L^1\cap L^\infty\right)(\R^d)\right)$. 
At the same time, it is possible to replace the assumptions 
on the flux function, cf.~\eqref{eq:scl-flux-Lipschitz1} 
and \eqref{eq:scl-flux-divx}, by the following more general ones:
\begin{equation}\label{eq:scl-flux-Lipschitz1-new}
	\begin{split}
		& A(t,x,u)=\tilde A(t,x,u)+\tilde{\tilde A}(t,u),
		\\ & \abs{\tilde A(t,x,u)}\le 
		m_a(t)\left(1+\abs{u}\right)\left(1+\abs{x}\right),
		\quad \abs{\tilde{\tilde A}(t,x,u)}
		\le m_a(t)\left(1+\abs{u}^{r_a}\right),
		\\ & 
		\abs{\tilde A(t,x,u)-\tilde A(t,x,v)} 
		\le m_a(t)\abs{u-v}\left(1+\abs{x}\right),
		\\ & 
		\abs{\tilde{\tilde A}(t,u)-\tilde{\tilde A}(t,v)} 
		\le m_a(t)\left(1+\abs{u}^{r_a-1}+\abs{v}^{r_a-1}\right)\abs{u-v},
		\\ & 
		\abs{(\Div_x \tilde A)(t,x,u)}
		\le m_d(t)\left(1+\abs{u}\right), 
		\qquad (\Div_x \tilde A)(t,x,0)=0,
	\end{split}
\end{equation}
for $t\in [0,T]$, $x\in\R^d$, and $u,v\in\R$, where 
$r_a \ge 1$ is a number and $m_a,m_d$ are integrable functions on $[0,T]$. 
"Globally Lipschitz" fluxes correspond 
to setting $\tilde{\tilde A}\equiv 0$ 
in \eqref{eq:scl-flux-Lipschitz1-new}, while "polynomially 
growing" ($x$-independent) fluxes correspond to 
setting $\tilde A\equiv 0$. In the "weight-free" $L^p$ 
framework it is natural to assume \eqref{eq:scl-init-ass-new}.

\medskip\noindent
Most of the works on kinetic solutions for stochastic 
conservation laws have dealt with the torus case ($\Bbb{T}^d$), and 
$x$-independent flux / no reaction term. The works on 
entropy solutions, on the other hand, have considered the 
unbounded domain case ($\R^d$), often with globally Lipschitz 
($x$-independent) flux and no reaction term. In \cite{FN} 
the authors allow for a polynomially growing flux $A=A(u)$ (and 
$R\equiv 0$), corresponding to the $\tilde{\tilde A}=\tilde{\tilde A}(u)$ part of our flux. 
Existence of an entropy solution is proved in \cite{FN} under the 
assumptions \eqref{eq:scl-init-ass-new} and 
\eqref{eq:scl-def-noise-B-new}, whereas uniqueness is 
established under the weaker condition \eqref{eq:scl-def-noise-B}. 
These results, based on entropy solutions, are consistent 
with ours based on kinetic solutions. 

\medskip\noindent
For some specific choices of the noise amplitude $B$ it is 
possible to construct $L^\infty$ solutions of \eqref{eq:scl}, 
that is, $u\in L^\infty_{\omega,t,x}$, assuming 
\eqref{eq:scl-init-ass}. Of course, for $L^\infty$ 
solutions, it is sufficient that $A,R,B$ are merely "locally Lipschitz in $u$".

\medskip\noindent
In what follows, we mostly lay out the results and proofs 
in the context of weighted $L^p$ spaces. However, whenever 
relevant conditions are imposed on the "data" of 
the problem, cf.~\eqref{eq:scl-init-ass-new}, 
\eqref{eq:scl-def-noise-B-new}, and \eqref{eq:scl-flux-Lipschitz1-new}, 
the reader may set "$w_N\equiv 1$" in the stated results.
\end{remark}

We are going to rely on the (more precise) "kinetic" 
interpretation \cite{Perthame:2002qy} of the entropy 
inequalities \eqref{eq:scl-entropy-ineq}. 
The mapping $\chi:\R^2 \to \R$ defined by
\begin{equation*}
	\chi(\xi,u) 
	=
	\begin{cases}
		\En_{0<\xi<u}, & \text{if $u>0$}\\
		0, & \text{if $u=0$}\\
		-\En_{u<\xi<0} & \text{if $u<0$}\\
	\end{cases}	
\end{equation*}
is called a $\chi$ function. Notice that $\chi(\xi,u)
=\En_{\xi<u} - \En_{\xi<0}$ for a.e.~$\xi$, for each 
fixed $u\in \R$. Moreover, $\chi$ is compactly supported 
in the $\xi$-variable, and thus $\chi(\cdot,u)\in L^1(\R)$. 
For any locally Lipschitz continuous $h:\R\to\R$, we have the following 
representation formula:
\begin{equation*}
	h(u)=h(0) +\int_{\R} h'(\xi) \chi(\xi,u)\dxi,
	\qquad u\in \R.
\end{equation*}

We also need the "one-sided" $\chi$-functions 
$\chip(\xi,u)=\En_{\xi<u}$ and $\chim(\xi,u):=\chip(\xi,u)-1$
($=-\En_{\xi\ge u}$). Observe that 
$\chip(\xi,u)=\chi(\xi,u)+\En_{\xi<0}$ and 
$\chim(\xi,u)= \chi(\xi,u)-\En_{\xi\ge 0}$, 
for a.e.~$\xi$, for each fixed $u\in \R$. 
In contrast to $\chi$, the one-sided functions $\chipm(\cdot,u)$ 
are not compactly supported and thus not integrable on $\R$.  
In most applications, however, it is sufficient 
that $\chipm(\cdot,u)$ is in $L^1_\loc(\R)$, for 
each fixed $u\in \R$. 

\begin{remark}[properties of $\chip$]\label{eq:scl-chip-prop}
The following properties are easy to verify:
\begin{enumerate}
	\item $(u-v)_+=\int_{\R} \chip(\xi,u)(1-\chip(\xi,v))\dxi$;
	
	\item $\int_{\R} S'(\xi)\chip(\xi,u)(1-\chip(\xi,v))\dxi
	=\En_{u>v}\left(S(u)-S(v)\right)$, $\forall S\in \liploc(\R)$;
		
	\item $\abs{u-v}=\int_{\R} \abs{\chip(\xi,u)-\chip(\xi,v)}\dxi$;
	
	\item Set $g(\xi,u,v)=\frac12\left(\chip(\xi,u)+\chip(\xi,v)\right)$. 
	Then $\frac14\abs{u-v} =\int_{\R} g-g^2 \dxi$.
\end{enumerate}
\end{remark}

Let us introduce the following notations for further use:
\begin{align*}
	& a_i=a_i(t,x,\xi):=(\partial_u A_i)(t,x,\xi), 
	\quad i=1,\ldots,d,
	\\ & 
	a=\left(a_1,\ldots,a_d\right),
	\quad
	d=d(t,x,\xi):=-(\Div_x A)(t,x,\xi),
\end{align*}
$\overline{a}=\overline{a}(t,x,\xi)= \seq{a,d}$ \cite{Da2}, 
and note that $\Div_{(x,\xi)} \overline{a}:=\Div_x a
+\partial_\xi d=0$. In view of our 
assumptions \eqref{eq:scl-flux-Lipschitz1}, 
\eqref{eq:scl-flux-divx}, and \eqref{eq:scl-reaction-Lipschitz1}, 
we clearly have 
\begin{equation}\label{eq:scl-a-growth}
	\norm{\frac{a(t,x,\xi)}{1+\abs{x}}}_{L^\infty_x}
	\le m_a(t),	\qquad (\omega,t,\xi)\in \Omega\times[0,T]\times\R,
\end{equation}
\begin{equation}\label{eq:scl-d-growth}
	\norm{d(t,x,\xi)}_{L^\infty_x}
	\le m_d(t) \left(1+\abs{\xi}\right),	
	\qquad (\omega,t,\xi)\in \Omega\times[0,T]\times\R,
\end{equation}
and
\begin{equation}\label{eq:scl-R-growth}
	\norm{R(t,x,\xi)}_{L^\infty_x}
	\le m_R(t) \left(1+\abs{\xi}\right),
	\quad 
	\norm{\pxi R(\omega,t,x,\xi)}_{L^\infty_x}
	\le m_R(t),
\end{equation}
for $(\omega,t,\xi)\in \Omega\times[0,T]\times\R$. 
These estimates imply, a.s., $\overline{a},\, 
R \in L^1\left(0,T;L^1_{\loc}(\R^d\times\R)\right)$. 
Besides, we will always assume
\begin{equation}\label{eq:scl-a-d-R-Sobolev}
	\Grad_{(x,\xi)} \overline{a},\, \Grad_{x} R
	\in L^1\left(0,T;L^1_{\loc}(\R^d\times\R)\right)
	\quad \text{a.s.},
\end{equation}
and so, a.s., $\overline{a},\, R \in L^1\left(0,T;
W^{1,1}_{\loc}(\R^d\times\R)\right)$ (for 
the DiPerna-Lions regularization lemma). 

Setting 
$$
\rho=\rho(\omega,t,x):=\chi_+(\xi,u(\omega,t,x))
=\En_{\xi<u(\omega,t,x)},
$$
the kinetic equation reads
\begin{equation}\label{eq:scl-kinetic-eqn}
	\begin{split}
		\pt \rho &+\Div_{(x,\xi)}\Bigl(\overline{a}\rho\Bigr)
		+R \partial_\xi \rho
		\\ & 
		+\sum_{k\ge 1}b_k \partial_\xi \rho\, \dotW_k(t)
		=  \pxi \left(\frac{B^2}{2}\pxi \rho \right)+\partial_\xi m
		\quad \text{in $\Dp([0,T)\times\R^d\times\R)$, a.s.},
	\end{split}
\end{equation}
where $\overline{a}:=\set{a,d}$ 
satisfies $\Div_{(x,\xi)}\overline{a}=0$, 
$B^2$ is defined in \eqref{eq:scl-def-noise-B}, 
and $\partial_\xi \rho=-\delta(\xi-u)$. All the coefficients 
$\overline{a},R,b_k, B^2$ depend on $(t,x,\xi)$. 
On the right-hand side of \eqref{eq:scl-kinetic-eqn}, 
$m$ is the so-called \textit{kinetic measure}.

\begin{remark}\label{rem:our-eqs}
Observe that the stochastic kinetic equations 
\eqref{e2.19} and \eqref{e2.22'}, which arise in 
our first homogenization problem, are both 
of the type \eqref{eq:scl-kinetic-eqn}.
On the other hand, the kinetic equations \eqref{e1.12v} 
and \eqref{e1.16iv} (arising in the second homogenization 
problem) are not, see also Remark \ref{rem:our-eqs}. 
However, combining the arguments 
developed in this section with those used 
in the proof of Lemma \ref{l3.1}, we can also
handle this (new) type of stochastic kinetic equations.
\end{remark}

\begin{definition}[kinetic measure]\label{def:kinetic-measure}
A nonnegative mapping $m:\Omega \to \radon([0,T]\times \R^d\times\R)$ 
is called a (weighted) kinetic measure provided the following three 
conditions hold:

\begin{enumerate}
\item $m(\phi):\Omega\to\R$ is measurable for each 
$\phi\in C_c([0,T]\times \R^d\times\R)$, where $m(\phi)$ 
denotes the action of $m$ on $\phi$, i.e., $m(\phi)=
\int_{[0,T]\times \R^d\times \R} \phi(t,x,\xi)\, m(dt,dx,d\xi)$;

\item the process $(\omega,t)\mapsto m(\phi)([0,t]\times\R^d\times\R)
=\int_{[0,t]\times M\times\R}\phi(x,\xi)\, m(ds,dx,d\xi)$ 
is predictable and belongs to $L^2(\Omega\times [0,T])$, 
for any $\phi\in C_c(\R^d\times\R)$;

\item $m$ exhibits weighted $p$--moments: $m_N:=w_Nm$, 
cf.~\eqref{eq:scl-weight-def}, satisfies
\begin{equation}\label{eq:pmoment-weight}
	\EE \int_{[0,T]\times \R^d\times \R} 
	\abs{\xi}^p \, m_N(dt,dx,d\xi)
	\lesssim_{T,N,p} 1, 
	\qquad \forall p\in [0,\infty).
\end{equation}
\end{enumerate}
\end{definition}

\begin{definition}[kinetic solution]\label{def:scl-kineticsol}
Given an initial function $u_0\in L^\infty\left(\Omega,\mathcal{F}_0;L^\infty(\R^d)\right)$, 
set $\rho_0:=\En_{\xi<u_0}$. A measurable function 
$u:\Omega\times [0,T]\times \R^d\to \R$ is said to be 
a \textit{kinetic solution} of \eqref{eq:scl} if 
$u$ is a predictable $L^2(w_Ndx)$--valued stochastic process such that  
\begin{equation}\label{eq:kinetic-Lp-bound}
	\EE\left(\esup_{t\in [0,T]} 
	\norm{u(t)}^p_{L^p(w_Ndx)}\right)\lesssim_{T,N,p} 1, 
	\qquad \forall p\in [2,\infty),
\end{equation}
and there is a kinetic measure $m$ such that 
$\rho:=\En_{\xi<u}$ satisfies \eqref{eq:scl-kinetic-eqn}.
\end{definition}

\begin{remark}
The property $\partial_\xi \rho = -\delta(\xi-u)$ 
is satisfied by any kinetic solution $\rho$ (and thus $\rho\in BV_\xi$). 
Given a function $H=H(t,x,\xi)$ that is continuous in $\xi$, we assign 
the following meaning to the distribution $H\pxi \rho$:
\begin{equation*}
	\action{H\pxi \rho}{\phi}_{\Dp_{\xi},\D_{\xi}}
	=-H(t,x,u(\omega,t,x))\phi(t,x,u(\omega,t,x)),
	\qquad \phi\in \D_{t,x,\xi},
\end{equation*}
for a.e.~$(\omega,t,x)\in \Omega\times [0,T]\times\R^d$, thereby explaining 
the meaning of \eqref{eq:scl-kinetic-eqn}.
\end{remark}

\begin{remark}[entropy \& kinetic solutions]
It is equivalent to be a kinetic solution according to 
Definition \ref{def:scl-kineticsol} and an entropy solution, 
i.e., a weak solution of \eqref{eq:scl} satisfying \eqref{eq:scl-entropy-ineq}.	
\end{remark}

\begin{remark}[weighted $p$--moments of kinetic measure]\label{rem:weighted-pmoments}
Fix a kinetic solution $\rho$ with kinetic measure $m$. 
For later use, let us compute the $p$-moments of the weighted measure 
$m_N:=w_N m$, where $w_N$ is the weight 
function \eqref{eq:scl-weight-def}. 
It follows from \eqref{eq:scl-kinetic-eqn} that
\begin{equation}\label{eq:scl-m-eqn}
	\begin{split}
		m(\partial_\xi \varphi)([0,T])
		&=\int_{[0,T]\times \R^d\times \R} 
		\pxi \varphi(x,\xi)\, m(dt,dx,d\xi)
		\\ & =\action{\chi_0}{\varphi}
		-\action{\chi(T)}{\varphi}
		+\int_0^T\action{\rho(t)}{\overline{a}(t)
		\cdot \Grad_{(x,\xi)} \varphi}\,dt
		\\ & \qquad 
		-\int_0^T\action{\bigl(R\pxi \rho\bigr)(t)}{\varphi}\,dt
		-\sum_{k\ge 1}\int_0^T
		\action{\bigl(b_k\pxi \rho\bigr)(t)}{\varphi}\,dW_k(t)
		\\ & \qquad\qquad 
		-\int_0^T\action{\left(\frac{B^2}{2}
		\pxi \rho\right)(t)}{\pxi \varphi}\,dt,
		\qquad 
		\forall \varphi \in C^\infty_c(\R^d\times\R),
	\end{split}
\end{equation}
where $\chi:=\rho-\En_{\xi<0}$ and $\chi_0:=\rho_0-\En_{\xi<0}$. 
Fix any convex function $S\in C^2(\R)$ with 
$\abs{S(\xi)} \lesssim \abs{\xi}^{p+2}$, 
$\abs{S'(\xi)} \lesssim \abs{\xi}^{p+1}$, 
$\abs{S''(\xi)} \lesssim \abs{\xi}^p$ ($p\ge 0$), 
i.e., $S\in C^2_{\text{pol}}(\R)$. We will utilize the test function 
$\varphi=\varphi_{\kappa,\ell}(x,\xi):=S'(\xi)\, w_N(x)
\, \phi_\kappa(x)\psi_\ell(\xi) 
\overset{\kappa,\ell\uparrow \infty}{\longrightarrow}
S'(\xi)\, w_N(x)$, where $\phi_\kappa(x)=\phi_1\left(\frac{x}{\kappa}\right)$, 
$\phi_1\in C^\infty_c(\R^d)$, $0\le \phi_1\le 1$, 
$\phi_1=1$ on $\set{\abs{x}\le 1}$, and 
$\phi_1=0$ on $\set{\abs{x}\ge 2}$. 
Moreover, $\psi_\ell(x)=\psi_1\left(\frac{\xi}{\ell}\right)$, 
$\psi_1\in C^\infty_c(\R^d)$, $0\le \psi_1\le 1$, 
$\psi_1=1$ on $\set{\abs{\xi}\le 1}$, 
and $\psi_1=0$ on $\set{\abs{\xi}\ge 2}$. 
We refer to $\seq{\phi_\kappa(x)}_{\kappa\ge1}$, 
and $\seq{\psi_\ell(x)}_{\ell\ge1}$ as truncation 
sequences (on, respectively, $\R^d$ and $\R$). 
Clearly, $\abs{\Grad\phi_\kappa(x)}\lesssim 
\frac{1}{\kappa}\En_{\kappa\le \abs{x}\le 2\kappa}$, 
$\abs{\psi_\ell'(\xi)}\lesssim \frac{1}{\ell}
\En_{\ell\le \abs{\xi}\le 2\ell}$, and
\begin{align*}
	& \pxi \varphi_{\kappa,\ell}
	= S''(\xi)w_N(x)\phi_\kappa(x)\psi_\ell(\xi) + 
	S'(\xi) w_N(x) \phi_\kappa(x)\psi_\ell'(\xi)
	\overset{\kappa,\ell\uparrow \infty}{\longrightarrow} 
	S''(\xi) w_N(x),
	\\ & 
	\Grad_x \varphi_{\kappa,\ell}
	= S'(\xi) \Grad w_N(x)\phi_\kappa(x)\psi_\ell(\xi) 
	+ S'(\xi) w_N(x) \Grad \phi_\kappa(x)\psi_\ell(\xi)
	\overset{\kappa,\ell\uparrow \infty}{\longrightarrow} 
	S'(\xi) \Grad w_N(x).
\end{align*}
Making use of $\varphi_{\kappa,\ell}$ in 
\eqref{eq:scl-m-eqn} and sending $\kappa,\ell\to \infty$, 
we eventually arrive at the following equation
satisfied a.s.~by the weighted kinetic measure $m_N$ ($=w_Nm$):
\begin{equation}\label{eq:p-moment-eqn}
	\begin{split}
		& \int_{[0,T]\times \R^d\times \R} S''(\xi) \, m_N(dt,dx,d\xi) 
		= \int_{\R^d} S(u_0) \, w_N dx
		- \int_{\R^d} S(u(T)) \, w_N dx
		\\ & \qquad 
		+\int_0^T \int_{\R^d} \Biggl(
		-2N\frac{Q_S(t,x,u)\cdot x}{1+\abs{x}^2}+
		(\Div_x Q_S)(t,x,u)
		\\ & \qquad\qquad\qquad\qquad\quad
		+S'(u)\left(R(t,x,u)
		-(\Div_x A)(t,x,u)\right)\Biggr)\, w_N dx\,dt
		\\ & \qquad
		+\sum_{k\ge1}\int_0^T\int_{\R^d}
		S'(u) b_k(t,x,u)\, w_N dx\,dW_k(t)
		\\ & \qquad
		+\frac12 \int_0^T \int_{\R^d} S''(u) B^2(t,x,u)
		\, w_N dx\,dt,
	\end{split}
\end{equation}
for any $S\in C^2_{\text{pol}}(\R)$, $S(0)=0$, $S''\ge 0$. Keeping 
in mind our assumptions \eqref{eq:scl-flux-Lipschitz1}, 
\eqref{eq:scl-flux-divx}, \eqref{eq:scl-reaction-Lipschitz1}, 
\eqref{eq:scl-def-noise-B}, and \eqref{eq:kinetic-Lp-bound}, 
choosing $S(\xi)=\frac{1}{(p+1)(p+2)}\abs{\xi}^{p+2}$ in 
\eqref{eq:p-moment-eqn} gives
\begin{equation}\label{eq:scl-mN-pbound}
	\EE\int_{[0,T]\times \R^d\times \R} 
	\abs{\xi}^p \, m_N(dt,dx,d\xi)
	\le C, \qquad p\in [0,\infty),
\end{equation}
where $C$ depends on $T,N$ and $\norm{u}_{L^{p+2}\left(\Omega;
L^\infty\left(0,T;L^{p+2}(w_Ndx)\right)\right)}$ (see also next remark).

\medskip\noindent
Regarding the "weight-free" $L^p$--framework 
discussed in Remark \ref{rem:drop-the-weight}, 
cf.~\eqref{eq:scl-init-ass-new}, 
\eqref{eq:scl-def-noise-B-new}, and \eqref{eq:scl-flux-Lipschitz1-new}, 
the equation \eqref{eq:p-moment-eqn} continues to hold with 
$w_N\equiv 1$ (and thus $m_N=m$), in which case 
the "$-2N$" term is zero. As a result, 
$\EE \int_{[0,T]\times \R^d\times \R} 
\abs{\xi}^p \, m(dt,dx,d\xi)\le C$, where 
$C$ depends on $T$ and $\norm{u}_{L^{p+2}\left(\Omega;
L^\infty\left(0,T;L^{p+2}(\R^d)\right)\right)}$.

\medskip\noindent
For $L^\infty$--solutions, the bound \eqref{eq:scl-mN-pbound} 
on $m_N$ continues to hold with $C$ depending on $T,N$, and 
$K_\text{max}:=\norm{u}_{L^\infty_{\omega,t,x}}$.
If $R-(\Div_x A)$, $b_k$, $B^2$ are zero on $\R_\xi\setminus 
\left[-K_\text{max},K_\text{max}\right]$, it 
follows from \eqref{eq:p-moment-eqn} that the weighted 
kinetic measure $m_N$ is compactly supported in $\xi$.
\end{remark}

\begin{remark}[improvement of integrability via a martingale argument]\label{remark:improv-integrability}
By the previous remark, the random variable 
$\omega\mapsto \int_{[0,T]\times \R^d\times \R} 
\abs{\xi}^p \, m_N(dt,dx,d\xi)$ belongs to $L^1(\Omega)$. 
One can improve this to $L^q(\Omega)$ for any finite $q\ge 1$. 
To this end, we will argue that
\begin{equation*}
	\begin{split}
		&\EE\left(\esup_{t\in [0,T]}\norm{u(t)}_{L^{p+2}(w_Ndx)}^r\right)+
		\quad 
		\EE\left(\int_{[0,T]\times \R^d\times \R}
		\abs{\xi}^p \, m_N(dt,dx,d\xi)\right)^{\frac{r}{p+2}}
		\lesssim_{r,T,N} 1, 
	\end{split}
\end{equation*}
provided the initial data $u_0$ satisfy
$\EE \left(\norm{u_0}_{L^{p+2}(w_Ndx)}^r\right)<\infty$, 
for $r>p+2$, a condition that clearly is satisfied due to 
\eqref{eq:scl-init-ass}. The case $r=p+2$ is covered by the definition of 
kinetic solution, cf.~\eqref{eq:pmoment-weight} 
and \eqref{eq:kinetic-Lp-bound}. In view of \eqref{eq:p-moment-eqn} 
with $S(\xi)=\frac{1}{(p+1)(p+2)}\abs{\xi}^{p+2}$  
and the growth assumptions \eqref{eq:scl-a-growth},
\eqref{eq:scl-d-growth}, and \eqref{eq:scl-R-growth}, 
it follows easily that
\begin{equation}\label{eq:p-moment-eqn2}
	\begin{split}
		&\esup_{t\in [0,T]} \int_{\R^d} \abs{u(t)}^{p+2} \, w_N dx 
		+\int_{[0,T]\times \R^d\times \R}\abs{\xi}^p \, m_N(dt,dx,d\xi)
		\\ & \qquad \lesssim 
		\int_{\R^d} \abs{u_0}^{p+2} \, w_N dx
		+\int_0^T\int_{\R^d} \abs{u(t)}^{p+2} \,w_N dx \,dt
		+\sup_{t\in [0,T]}\abs{M(t)},
	\end{split}
\end{equation}
for a.e.~$(\omega,t)\in\Omega\times[0,T]$, where
\begin{align*}
	M(t)=\sum_{k\ge1}\int_0^t\int_{\R^d}
	S'(u) b_k(\omega,s,x,u)\, w_N dx\,dW_k(s),
	\quad 
	S'(u)=\frac{1}{p+1}\abs{u}^p u.
\end{align*}
We raise both sides of \eqref{eq:p-moment-eqn2} to the 
power $r/(p+2)>1$, apply Jensen's inequality to the 
second term on the right-hand side, and 
take the expectation, eventually arriving at
\begin{equation}\label{eq:p-moment-eqn3}
	\begin{split}
		&\EE \left(\esup_{t\in [0,T]}
		\norm{u(t)}_{L^{p+2}(w_Ndx)}^r\right)
		+\EE\left(\int_{[0,T]\times \R^d\times \R}
		\abs{\xi}^p \, m_N(dt,dx,d\xi)\right)^{\frac{r}{p+2}}
		\\ & \quad 
		\lesssim_T \EE \left(\norm{u_0}_{L^{p+2}(w_Ndx)}^r\right)
		+\int_0^T \EE \left(\norm{u(t)}_{L^{p+2}(w_Ndx)}^r\right) \,dt
		\\ & \quad\quad\qquad
		+\EE \sup_{t\in [0,T]}\abs{M(t)}^{\frac{r}{p+2}}.
	\end{split}
\end{equation}
A standard martingale argument (Burkholder-Davis-Gundy 
inequality \cite{BDG})  supplies
\begin{align*}
	& \EE \sup_{t\in [0,T]}\abs{M(t)}^{\frac{r}{p+2}}
	\\ & \quad \lesssim_{T,N}\frac12 \EE \left( \esup_{t\in [0,T]} 
	\norm{u(t)}_{L^{p+2}(w_Ndx)}^r\right)
	+\int_0^T \EE \left(\norm{u(t)}_{L^{p+2}(w_Ndx)}^r\right) \,dt+1.
\end{align*}
Making use of this estimate in \eqref{eq:p-moment-eqn3}, followed by 
an application of Gronwall's inequality, leads to the 
sought after estimates. 

\medskip

It is easy to make the previous argument operational 
in the "weight-free" $L^p$--framework discussed in 
Remark \ref{rem:drop-the-weight}, assuming 
\eqref{eq:scl-init-ass-new}, \eqref{eq:scl-def-noise-B-new}, 
\eqref{eq:scl-flux-Lipschitz1-new}. 
The same applies to $L^\infty$--solutions.
\end{remark}

Roughly speaking, the difference between a kinetic solution 
$\rho$ and a so-called \textit{generalized} kinetic 
solution $\varrho$ is that the structural property 
$\partial_\xi \rho = -\delta(\xi-u)$ is 
replaced by the requirement $\partial_\xi \varrho=-\nu$ for some 
Young measure $\nu$ on $\R_\xi$. 
We refer to \cite{DV} for relevant background 
material on Young measures.

In what follows, any function of the form 
$\rho=\rho(z,\xi)=\En_{\xi<u(z)}$ will be called a 
\textit{kinetic function}.  We reserve the term 
\textit{generalized kinetic function} to
functions $\varrho=\varrho(z,\xi)$ taking values in $[0,1]$ such 
that $-\pxi \varrho$ is a Young measure. For us $z=(\omega,x)$ 
or $z=(\omega,t,x)$.

\begin{definition}[generalized kinetic solution]
Fix a generalized kinetic function $\varrho_0(\omega,x,\xi)$. 
We call $\varrho:\Omega\times [0,T]\times \R^d\times \R\to [0,1]$ 
a \textit{generalized kinetic solution} of \eqref{eq:scl} 
with initial data $\rho_0$ if $\widetilde \varrho:=\varrho-\En_{\xi<0}$ 
is $\mathcal{P}/\mathcal{B}(L^2(w_N dx \, d\xi))$ 
measurable and
\begin{equation}\label{eq:kinetic-Lp-bound-Young}
	\EE\left(\esup_{t\in[0,T]}
	\iint_{\R^d\times \R}\abs{\xi}^p
	w_N(x)\, \nu_{\omega,t,x}(d\xi) \dx\right)\lesssim_{T,N,p} 1, 
	\qquad 
	\forall p\in [2,\infty),
\end{equation}
where $\nu:=-\partial_\xi\varrho$ is a Young measure, 
the spatial weight $w_N$ is defined in \eqref{eq:scl-weight-def}, 
and there is a kinetic measure $m$ such 
that $\varrho$ satisfies a.s.
\begin{equation}\label{eq:scl-gen-kinetic-eqn}
	\begin{split}
		\pt \varrho &+\Div_{(x,\xi)}\Bigl(\overline{a}\varrho\Bigr)
		+R \partial_\xi \rho
		\\ & 
		+\sum_{k\ge 1}b_k \partial_\xi \varrho\, \dotW_k(t)
		=  \pxi \left(\frac{B^2}{2}\pxi \varrho \right)+\partial_\xi m
		\quad \text{in $\Dp([0,T)\times\R^d\times\R)$}.
	\end{split}
\end{equation}
\end{definition}

\begin{remark}
Given a function $H(t,x,\xi)$ that is 
continuous in $\xi$ and a generalized kinetic solution $\varrho$, 
we assign the following meaning to the distribution $H\pxi \varrho$:
\begin{equation*}
	\action{H\pxi \varrho}{\phi}_{\Dp_{\xi},\D_{\xi}}
	=-\int_{\R}
	H(\omega,t,x,\xi)\phi(t,x,\xi) \, \nu_{\omega,t,x}(d\xi),
	\qquad \phi\in \D_{t,x,\xi},
\end{equation*}
for a.e.~$(\omega,t,x)\in \Omega\times [0,T]\times\R^d$, 
thereby making precise the meaning of \eqref{eq:scl-gen-kinetic-eqn}.
\end{remark}

\begin{remark}
Although a generalized kinetic solution $\varrho$ 
is merely locally integrable in $\xi$, the associated 
function $\widetilde\varrho$ ($=\varrho-\En_{\xi<0}$) 
is globally integrable; by \eqref{eq:kinetic-Lp-bound-Young},
\begin{equation*}
	\iint_{\R^d\times \R} 
	\abs{\widetilde \varrho(t)}
	\abs{\xi}^p\,w_N(x)
	\, \dxi\dx\lesssim_{T,N,p} 1,	
	\quad t\in [0,T],
	\quad \forall p\in [1,\infty).
\end{equation*}
\end{remark}

\begin{remark}[c{\`a}dl{\`a}g / c{\`a}gl{\`a}d versions]
There are general theorems \cite{RY} ensuring that many 
real-valued stochastic processes $X(t)$ (discontinuous 
semimartingales) have a right-continuous version and, what's 
more, these versions necessarily have left-limits everywhere. 
Right-continuous processes with left-limits everywhere are referred to 
as \textit{c{\`a}dl{\`a}g}. Left-continuous processes with 
right-limits everywhere are referred to as \textit{c{\`a}gl{\`a}d}.

A generalized kinetic solution $\varrho$ is clearly not affected 
by modification of its values on any set of measure zero. 
In fact, $\varrho$ is an equivalence class of functions.
When proving stability and uniqueness results we must 
work with left/right continuous representatives of each 
equivalence class. Indeed, a result from \cite[Proposition 10]{DV} 
(see also \cite[Lemma 1.3.3]{Dotti:2018aa}), easily generalized 
to our setting, says that a generalized kinetic solution $\varrho$
possesses weak left and right limits 
$\varrho^{t,\pm}$ at every instant of time $t$. 
We then introduce left and right continuous representatives of $\varrho$ 
by setting $\varrho^\pm(t):=\varrho^{t,\pm}$ for all $t\in[0,T]$. 
Clearly, $\varrho^{\pm}$ are both predictable 
since $\varrho$ is. Using the left and right continuous representatives 
$\varrho^{\pm}$ one can convert the time-space weak 
formulation \eqref{eq:scl-gen-kinetic-eqn} 
into a formulation that is weak in space only 
(and pointwise in time): for \textit{any} $t\in [0,T]$, a.s.,
\begin{equation}\label{eq:scl-weak-in-space}
	\begin{split}
		\action{\varrho^{\pm}(t)}{\varphi}
		& =\action{\varrho_0}{\varphi}
		+\int_0^t\action{\varrho(s)}{\overline{a}(s)
		\cdot \Grad_{(x,\xi)} \varphi}\,ds
		-\int_0^t\action{\bigl(R\pxi \varrho\bigr)(s)}{\varphi}\,ds
		\\ & \qquad 
		-\sum_{k\ge 1}\int_0^t
		\action{\bigl(b_k\pxi \varrho\bigr)(s)}{\varphi}\,dW_k(s)
		-\int_0^t
		\action{\left(\frac{B^2}{2}\pxi \varrho\right)(s)}{\pxi \varphi}\,ds
		\\ & \qquad \qquad 
		-
		\begin{cases}
			m(\partial_\xi \varphi)([0,t]), 
			& \text{for $\varrho^+$}\\
			m(\partial_\xi \varphi)([0,t)), 
			& \text{for $\varrho^-$}
 		\end{cases}.
 	\end{split}
\end{equation}
Be mindful of the fact that $\action{\varrho^+(t)-\varrho^-(t)}{\varphi}
=-m(\partial_\xi \varphi)(\set{t})$. Since the atomic points of 
$m(\partial_\xi \varphi)(\cdot)$ is at most countable, 
we have $\action{\varrho^+(t)}{\varphi}=\action{\varrho^-(t)}{\varphi}$ 
for a.e.~$t$ and in turn $\varrho^+=\varrho^-$ almost everywhere. 
The real-valued stochastic processes 
$X^\pm (t):= \action{\varrho^{\pm}(t)}{\varphi}$, defined 
by \eqref{eq:scl-weak-in-space},  are of the form 
$X^\pm(t)=A^\pm(t)+M(t)$, where $A^\pm(t)$ 
are finite variation processes and 
$M(t)$ is a continuous martingale. 
Moreover, $A^+(0)=\action{\varrho_0}{\varphi}
-m(\partial_\xi \varphi)(\set{0})$, 
$A^-(0)=\action{\varrho_0}{\varphi}$, and $M(0)=0$. 
Below we note that $m(\partial_\xi \varphi)(\set{0})=0$ 
for kinetic initial data $\varrho_0=\En_{\xi<u_0}$.
Whenever convenient, we may assume that $X^+$ ($X^-)$ 
are c{\`a}dl{\`a}g (c{\`a}gl{\`a}d).

\medskip \noindent
In what follows, we will outline a proof of uniqueness. 
Although we should work with the left/right continuous 
representatives $\varrho^{\pm}$ as 
in \cite[Proposition 10]{DV} (see also \cite{Dotti:2018aa}) and 
make use of the space-weak formulation \eqref{eq:scl-weak-in-space}, 
we will not do so in an attempt to save space and keep the 
presentation as simple as possible. Instead we refer 
to \cite{DV,DHV,Dotti:2018aa,Galimberti:2018aa,Gess:2018ab,Ha} 
for such details, see also \cite{Gess:2018aa,Gess:2019aa}.
\end{remark}

\begin{remark}\label{rem:init-measure}
Let us make a comment on generalized kinetic solutions 
and the satisfaction of the initial condition. 
Suppose $\varrho_0=\En_{\xi<u_0}$ for some function 
$u_0$ satisfying \eqref{eq:scl-init-ass}. It follows 
from \eqref{eq:scl-weak-in-space} that (the right-continuous 
representative of) $\varrho$ satisfies a.s.
\begin{equation}\label{eq:scl-initial-measure}
	\action{\varrho(0)}{\varphi}
	=\action{\varrho_0}{\varphi}
	-m(\partial_\xi \varphi)(\set{0}),
	\qquad \forall \varphi\in C^\infty_c(\R^d\times \R).
\end{equation}
To conclude $\varrho(0)=\varrho_0$ we argue 
that $m\left(\set{0}\times \R^d\times \R\right)=0$. 
The argument is standard \cite{Perthame:2002qy}, so 
we merely sketch it. Following 
Remark \ref{rem:weighted-pmoments}, 
\eqref{eq:scl-initial-measure} implies a.s.~that
\begin{align*}
	\iint\limits_{\R^d\times\R} S'(\xi)
	\left(\widetilde \varrho(0) - \chi(\xi,u_0)\right) 
	\, w_N \dxi\dx
	+\int\limits_{\set{t=0}\times \R^d\times \R} 
	S''(\xi) \, w_N \, m(dt,dx,d\xi)=0,
\end{align*}
for any $S\in C^2(\R)$ for which $S''\ge 0$ 
and $S,S',S''$ grow at most polynomially. 
By Brenier's lemma \cite{Perthame:2002qy}, the first 
integral is nonnegative. As a result, both integrals  
must be zero. In other words, a.s., $\varrho(0)=\varrho_0$ 
and $m\left(\set{0}\times \R^d\times \R\right)=0$.
\end{remark}

Following an approach developed by Perthame \cite{Perthame:2002qy}, 
later extended to the stochastic case in \cite{DV} (see also 
\cite{DV,DHV,Dotti:2018aa,Galimberti:2018aa,Gess:2018ab,Ha,Kobayasi:2016aa,Lv:2016aa}), we establish a rigidity result implying that generalized 
kinetic solutions are in fact kinetic solutions, at least 
when the initial function is a kinetic function, 
$\varrho_0=\En_{\xi<u_0}$. The proof herein involves 
a regularization (via convolution) 
procedure, the It\^{o} formula, and commutator arguments 
(going beyond the deterministic one by DiPerna-Lions) \cite{GK}. 
Essentially the same proof also shows that kinetic 
solutions are uniquely determined by their initial data, 
satisfying an $L^1$ contraction principle. 

\begin{proposition}[rigidity result]\label{prop:scl-rigidity}
Suppose that $b_k,B^2,\overline{a}=\set{a,d},R$ satisfy conditions 
\eqref{eq:scl-def-noise-B}, \eqref{eq:scl-noise-bk-reg}, 
\eqref{eq:scl-a-growth}, \eqref{eq:scl-d-growth}, 
\eqref{eq:scl-R-growth}, \eqref{eq:scl-a-d-R-Sobolev}, 
and $\Div_{(x,\xi)} \overline{a}=0$. Let $\varrho$ be 
a generalized kinetic solution of \eqref{eq:scl} with initial data $\varrho_0$. 
Suppose $m\left(\set{0}\times \R^d\times \R\right)=0$. 
Then, for $t\in [0,T]$,
\begin{equation}\label{eq:scl-contraction-rhomrho2}
	0\le \EE \iint_{\R^d\times\R} 
	\left( \varrho-\varrho^2\right) (t) \, w_N\,d\xi \,dx
	\lesssim_{T,N} \EE\iint_{\R^d\times\R}
	\left(\rho_0-\rho_0^2\right)\, w_N \,d\xi \,dx.
\end{equation}
If $\varrho_0=\En_{\xi<u_0}$ for some $u_0$ satisfying 
\eqref{eq:scl-init-ass}, then $m\left(\set{0}\times \R^d\times \R\right)=0$ 
and thus $\varrho-\varrho^2=0$ a.e.; whence $\varrho=\En_{\xi<u}$ 
for some function $u$ that necessarily 
is a kinetic solution of \eqref{eq:scl}. 
\end{proposition}

\begin{remark}
Informally speaking, cf.~\eqref{eq:scl-weak-in-space}, we have 
$\varrho(t)=V(t)+M(t)$, where $V(t)$ is a finite variation 
process, $M(t)$ is a continuous martingale, and $\varrho(0)=V(0)$. 
In the proof below we need to determine the equation 
satisfied by $S(\varrho(t))$, where $S(\varrho)=\varrho-\varrho^2$. 
Noting that $(\varrho(t))^2=(V(t))^2+2V(t)M(t)+(M(t))^2$, we 
can calculate the first and second terms using 
standard calculus, while the third term can be computed 
using the It\^{o} formula for continuous martingales \cite{RY}. 
Alternatively, we use the It\^{o} formula 
for discontinuous semimartingales \cite{Jacod:2003aa} to write
$S(\varrho(t))=S(\varrho(0))
+\int_0^t S'(\varrho(s-))\, d\varrho(s) +Q_S(t)+J_S(t)$, 
where $Q_S(t)=\int_0^t \frac12 S''(\varrho(s-))
\, d[\varrho](s)$, $[\varrho](t)=[M](t)+\sum_{s\le t} 
\left(\Delta\varrho(s)\right)^2$ is the quadratic variation process, and 
{\small $J_S(t)=\sum_{s\le t} \left( S(\varrho(s))
-S(\varrho(s-))-S'(\varrho(s-)) \Delta \varrho(s) -
\frac12 S''(\varrho(s-)) \left(\Delta \varrho(s)\right)^2 \right)$} 
is the "jump part" coming from the (temporal) discontinuities in $\varrho$. 
With $S(\varrho)=\varrho-\varrho^2$ (and $S''=-2$),
we have $J_S\equiv 0$ and $Q_S(t)=-[M](t)-\sum_{s\le t} 
\left(\Delta\varrho(s)\right)^2\le -[M](t)$.
\end{remark}

\begin{proof}
We will first give an informal proof 
of \eqref{eq:scl-contraction-rhomrho2}. 
Recall that $\varrho$ satisfies a.s.~\eqref{eq:scl-gen-kinetic-eqn}. 
By the It\^{o} and classical chain rules we arrive 
at the following equation for $S(\varrho):=\varrho-\varrho^2$:
\begin{equation}\label{eq:scl-S(varrho)}
	\begin{split}
		\pt S(\varrho) &+\Div_{(x,\xi)}\Bigl(\overline{a}S(\varrho)\Bigr)
		+R \partial_\xi S(\rho)
		\\ & 
		+\sum_{k\ge 1}b_k \partial_\xi S(\varrho)\, \dotW_k(t)
		=  \pxi \left(\frac{B^2}{2}\pxi S(\varrho) \right)
		+S'(\varrho)\partial_\xi m+\mathcal{Q},
	\end{split}
\end{equation}
where $\mathcal{Q}$ contains the difference between certain 
quadratic terms linked to the variation of 
the martingale part and the second-order 
differential operator of the 
equation \eqref{eq:scl-gen-kinetic-eqn}:
\begin{align*}
	\mathcal{Q} & =
	\frac{S''(\varrho)}{2} \sum_{k\ge 1}
	\left(b_k \partial_\xi \varrho \right)^2 
	- \frac{S''(\varrho)}{2}B^2
	\left(\partial_\xi \varrho \right)^2
	\equiv 0.
\end{align*}
The perfect cancellation (i.e., $Q=0$) is the 
basic reason why the Proposition \ref{prop:scl-rigidity} holds.	 
It follows from \eqref{eq:scl-S(varrho)} that
$I(\phi)= I_0(\varphi)+\sum_{i=1}^4 I_i(\varphi)$, 
$t\in [0,T]$, where
\begin{align*} 
	& I(\varphi)=\EE \iint_{\R^d\times\R} S(\varrho(t)) 
	\varphi \,d\xi \,dx,
	\quad
	I_0(\varphi)=\EE\iint_{\R^d\times\R} S(\varrho_0) 
	\varphi \,d\xi \,dx,
	\\ & I_1(\varphi) = \int_0^t \left(\EE \iint_{\R^d\times\R} 
	S(\varrho(s)) \overline{a}(s)\cdot \Grad_{(x,\xi)}
	\varphi \,d\xi \,dx\right)\, ds,
	\\ & I_2(\varphi)=-\frac12 \int_0^t \left(\EE \iint_{\R^d\times\R} 
	B^2(s)\pxi S(\varrho(s)) \pxi \varphi \,d\xi \,dx\right)\, ds,
	\\ & 
	I_3(\varphi) = - \int_0^t \left(\EE \iint_{\R^d\times\R} 
	R(s)\pxi S(\varrho(s)) \varphi \,d\xi \,dx
	\right)\, ds,
	\\ & 
	I_4(\varphi) =-\EE \iiint_{[0,t]\times \R^d\times\R} 
	\pxi \left(S'(\varrho(s))\varphi\right) \,m(ds,dx,d\xi),
\end{align*}
for any $\phi \in C^1_c(\R^d\times \R)$. 
Let us particularize the test function as
\begin{equation}\label{eq:scl-testfunc-kl}
	\varphi(x,\xi)=\varphi_{\kappa,\ell}(x,\xi)
	=w_N(x) \phi_\kappa(x)\psi_\ell(\xi), 
\end{equation}
where the weight function $w_N$ is 
defined in \eqref{eq:scl-weight-def} 
and $\seq{\phi_\kappa}_{\kappa\ge 1}$, 
$\seq{\psi_\ell}_{\ell\ge 1}$ are truncation sequences 
respectively on $\R^d$, $\R$. 	

We rely on \eqref{eq:scl-a-growth} 
and \eqref{eq:scl-d-growth} to supply
\begin{equation*}
	\begin{split}
		&\abs{S(\varrho(s)) \overline{a}(s)
		\cdot \Grad_{(x,\xi)}\varphi_{\kappa,\ell}}
		\lesssim \left(\varrho-\varrho^2\right)(s)
		\abs{a(s)}\psi_\ell\left(\abs{\Grad w_N}
		+\frac{1}{\kappa} 
		\En_{\kappa\le \abs{\xi}\le 2\kappa} w_N\right)
		\\ & \quad\qquad\qquad\qquad\qquad\qquad\qquad
		+\left(\varrho-\varrho^2\right)(s)\abs{d(s)}\frac{1}{\ell} 
		\En_{\ell\le \abs{\xi}\le 2\ell} w_N
		\\ & \quad
		\lesssim \norm{\frac{a(s)}{1+\abs{x}}}_{L^\infty_x}
		\left(\varrho-\varrho^2\right)(s)\psi_\ell w_N
		+m_d(t)\left(\varrho-\varrho^2\right)(s)
		\left(1+\abs{\xi}\right)\frac{1}{\ell}
		\En_{\ell\le \abs{\xi}\le 2\ell} w_N
		\\ & \quad 
		\lesssim \left(m_a(s)+m_d(s)\right)
		\left(\varrho-\varrho^2\right)(s) w_N 
		\in L^1_{\omega,t,x,\xi},
	\end{split}
\end{equation*}
and thus 
\begin{align*}
	& \abs{I_1(\varphi_{\kappa,\ell})} 
	\lesssim \int_0^t \left(m_a+m_d\right)(s)
	\left(\EE \iint_{\R^d\times\R}
	\left(\varrho-\varrho^2\right)(s)
	\, w_N \,d\xi \dx \right)\, ds.
\end{align*}

Next, since $\varrho\in L^\infty_{\omega,t,x,\xi}$ and 
$\pxi \varrho=-\nu(d\xi)$,
\begin{equation*}
	\begin{split}
		& \abs{B^2(s)\pxi S(\varrho(s)) 
		\pxi \varphi_{\kappa,\ell}}\\ & \quad 
		\overset{\eqref{eq:scl-def-noise-B}}{\lesssim} 
		\frac{1}{\ell}
		\En_{\ell\le \abs{\xi}\le 2\ell}
		\left(1+\abs{\xi}^2\right) 
		\abs{1-2\varrho(s)}  \phi_{\kappa}\, w_N \,\nu(d\xi)
		\lesssim \frac{1}{\ell} 
		\left(1+\abs{\xi}^2\right) w_N\, \nu(d\xi), 
	\end{split}
\end{equation*}
and so, recalling \eqref{eq:kinetic-Lp-bound-Young}, 
$\abs{I_2(\varphi_{\kappa,\ell}}\lesssim_{T,N}\frac{1}{\ell}
\overset{\ell\uparrow \infty}{\longrightarrow} 0$.

Evoking \eqref{eq:scl-R-growth}, 
\begin{align*}
	\abs{\pxi \left( R\varphi_{\kappa,\ell}\right)}
	& \le \abs{\pxi R(s) \psi_{\ell}+R(s)\psi_{\ell}'}
	\, \phi_\kappa\, w_N 
	\\ & 
	\lesssim \left(m_R(s)+m_R(s)\left(1+\abs{\xi}\right)
	\En_{\ell\le \abs{\xi}\le 2\ell}\frac{1}{\ell}
	\right) w_N \lesssim m_R(s) \,w_N,
\end{align*}
and thus, after an integration by parts,
\begin{align*}
	\abs{I_3(\varphi_{\kappa,\ell})}
	\lesssim \int_0^t m_R(s)
	\left(\EE \iint_{\R^d\times\R}  
	\left(\varrho-\varrho^2\right)(s)
	\, w_N \,d\xi \,dx\right)\, ds.
\end{align*}

Finally, using again that $\pxi \varrho=-\nu$,
\begin{align*}
	-\pxi \left(S'(\varrho(s))\varphi_{\kappa,\ell}\right)
 	& = -2 \phi_\kappa\, \psi_\ell\, w_N\, \nu(d\xi)
 	-\left(1-2\varrho(s)(s)\right)\, 
 	\phi_\kappa\,  \psi_\ell'\, w_N
 	\\ & \le \left(2\varrho(s)-1\right)
 	\, \phi_\kappa \, \psi_\ell' \, w_N,
\end{align*}
and so, putting $\varrho\in L^\infty_{\omega,t,x,\xi}$ and 
\eqref{eq:pmoment-weight} to good use,
\begin{equation*}
	\begin{split}
		\abs{I_4(\varphi_{\kappa,\ell})}
		\lesssim 
		\frac{1}{\ell}\EE m_N\left([0,T]\times\R^d
		\times \set{\ell\le \abs{\xi}\le 2\ell}\right)
		=O(1/\ell) \overset{\ell\uparrow \infty}{\longrightarrow} 0.
	\end{split}
\end{equation*}

Summarizing our computations (after sending $\kappa\to \infty$),
\begin{equation}\label{eq:scl-rho-rho2-ell}
	\begin{split}
		& \EE \int_{\R^d\times\R} \left( \varrho-\varrho^2\right) (t)
		\, \psi_\ell w_N\,d\xi \,dx
		\lesssim \EE\int_{\R^d\times\R}
		\left( \varrho_0-\varrho_0^2\right)
		\, \psi_\ell w_N \,d\xi \,dx
		\\ & \quad\quad
		+\int_0^t M(s)\left(\EE \int_{\R^d\times\R} 
		\left( \varrho-\varrho^2\right) (s)
		\, \psi_\ell w_N\,d\xi \,dx \right)\,ds+O(1/\ell),
	\end{split}
\end{equation}
where $M$ is an integrable function on $[0,T]$. We arrive at 
the sought after \eqref{eq:scl-contraction-rhomrho2} 
by sending $\ell\uparrow\infty$ and then
applying Gronwall's inequality.

\medskip \noindent 
Unfortunately the equation \eqref{eq:scl-S(varrho)} for 
$S(\rho)$ is only suggestive as the calculations 
involving the chain rule are merely formal. To make the calculations rigorous we regularize the 
"linear" equation \eqref{eq:scl-gen-kinetic-eqn}, bringing in several regularization 
errors that must be controlled. Let $J_\eps^x:\R^d\to \R$, $J^\xi_\delta:\R\to \R$ be 
standard Friedrich mollifiers, and define 
\begin{align*}
	\varrho_{\eps,\delta} (\omega,t,x,\xi) &
	= \varrho \star 
	\left(J_\eps^x J^\xi_\delta\right)
	= \iint_{\R^d\times \R} 
	\varrho(\omega,t,y,\zeta)J^x_{\eps}(x-y) 
	J^\xi_\delta(\xi-\zeta)\, dy \, d\zeta,
	\\ 
	m_{\eps,\delta} (\omega,t,x,\xi) & 
	= m \star \left(J^x_\eps J^\xi_\delta\right)
	=\iint_{\R^d\times \R} J^x_{\eps}(x-y) 
	J^\xi_\delta(\xi-\zeta)\, m(t,dy,d\zeta).
\end{align*}
The mollified quantities $\varrho_{\eps,\delta},m_{\eps,\delta}$ 
are smooth in $x,\xi$ but discontinuous in $t$. However, working with 
suitable representatives (versions), we can ensure 
that $\varrho_{\eps,\delta}, m_{\eps,\delta}$ are 
c{\`a}dl{\`a}g / c{\`a}gl{\`a}d in time $t$, thereby making the 
It\^{o} formula available to us, and thus the arguments below can be made rigorous (see 
e.g.~\cite{DV,Dotti:2018aa,Galimberti:2018aa,Gess:2018ab,Gess:2018aa,Gess:2019aa}). 
In passing, note that $m_{\eps,\delta}$ is a measure on $[0,T]$ 
(depending on the "parameters" $\omega,x,\xi$). 

The following equation holds a.s.:
\begin{equation}\label{eq:scl-kinetic-eqn-reg}
	\begin{split}
		\pt \varrho_{\eps,\delta} 
		& +\Div_{(x,\xi)}\Bigl(\overline{a}\varrho_{\eps,\delta}\Bigr)	
		+R \partial_\xi \varrho_{\eps,\delta}
		+\sum_{k\ge 1}\left( \left(b_k \partial_\xi \varrho\right)
		\star \left(J^x_\eps J^\xi_\delta\right)\right)
		\, \dotW_k(t)
		\\ & \quad
		= \pxi\left( \left(\frac{B^2}{2}\pxi \varrho \right)
		\star \left(J^x_\eps J^\xi_\delta\right)\right)
		+\partial_\xi m_{\eps,\delta}+r_{\eps,\delta}
		\quad \text{in $\Dp([0,T)\times\R^d\times\R)$},
	\end{split}
\end{equation}
where the reminder term $r_{\eps,\delta}=r_{\eps,\delta}(\omega,t,x,\xi)$ 
takes the form
\begin{align*}
	r_{\eps,\delta} & :=\Div_{(x,\xi)}
	\Bigl(\overline{a}\varrho_{\eps,\delta}\Bigr)
	-\Div_{(x,\xi)}\left( \left(\overline{a}\varrho\right)
	\star \left(J^x_\eps J^\xi_\delta\right)\right) 
	+ R\partial_\xi \varrho_{\eps,\delta}
	- \left(R\partial_\xi \varrho\right)
	\star\left(J^x_\eps J^\xi_\delta\right).
\end{align*}
Our assumptions imply that
$\overline{a},R\in L^1\left(0,T;W^{1,1}_{\loc}(\R^d\times \R)\right)$, 
whereas the generalized kinetic solution $\varrho$ belongs a.s.~to 
$L^\infty\left(0,T;L^\infty(\R^d\times\R)\right)$. Moreover, 
$\Div_{(x,\xi)} \overline{a} =0$. Hence, by 
\cite[Lemma II.1]{DiPerna:1989aa}, $r_{\eps,\delta}$ 
converges a.s.~to zero in $L^1_{\loc}$ as $\eps,\delta\to 0$. Given 
\eqref{eq:scl-kinetic-eqn-reg}, we apply the It\^{o} formula 
as well as the classical (spatial) chain rule.  
The result is the following equation for $S(\varrho_{\eps,\delta})$ 
that holds a.s.~in $\Dp([0,T)\times\R^d\times\R)$:
\begin{equation}\label{eq:scl-S(varrho)-approx}
	\begin{split}
		& \pt S(\varrho_{\eps,\delta})
		+\Div_{(x,\xi)}\Bigl(\overline{a}S(\varrho_{\eps,\delta})
		\Bigr)	+R \partial_\xi S(\varrho_{\eps,\delta})
		\\ & \quad
		+\sum_{k\ge 1} S'(\varrho_{\eps,\delta}) 
		\left( \left(b_k \partial_\xi \varrho\right)
		\star \left(J^x_\eps J^\xi_\delta\right)\right)\, \dotW_k(t)
		=  \pxi\left(\frac{B^2}{2}\pxi S(\varrho_{\eps,\delta})\right)
		\\ & \quad \quad
		+S'(\varrho_{\eps,\delta})\partial_\xi m_{\eps,\delta}
		+S'(\varrho_{\eps,\delta})r_{\eps,\delta}
		+ \pxi \left(S'(\varrho_{\eps,\delta})
		\tilde r_{\eps,\delta}\right)
		+\mathcal{Q}_{\eps,\delta},
	\end{split}
\end{equation}
where $\tilde r_{\eps,\delta}
= \frac{B^2}{2}\pxi \varrho_{\eps,\delta}
-\left(\frac{B^2}{2}\pxi \varrho \right)
\star \left(J^x_\eps J^\xi_\delta\right)$ and
{\small
\begin{equation}\label{eq:Qreg}
	\begin{split}
		\mathcal{Q}_{\eps,\delta} & =
		\frac12 S''(\varrho_{\eps,\delta})
		\sum_{k\ge 1} \left( \left(b_k \partial_\xi \varrho\right)
		\star \left(J^x_\eps J^\xi_\delta\right)\right)^2
		-\frac12 S''(\varrho_{\eps,\delta}) 
		\left(\left(B^2\pxi \varrho \right)
		\star \left(J^x_\eps J^\xi_\delta\right)\right)
		\pxi \varrho_{\eps,\delta},
	\end{split}
\end{equation}
}
As a result of assumptions \eqref{eq:scl-def-noise-B} 
and \eqref{eq:scl-noise-bk-reg}, 
$B^2\in L^1\left(0,T;W^{1,1}_{\loc}(\R^d\times \R)\right)$ (besides, 
we know $\varrho\in BV_\xi$). Thus, it is not difficult o show that, 
$\tilde r_{\eps,\delta}$ converges a.s.~to zero 
in $L^1_{\loc}$ as $\eps,\delta\to 0$ \cite{Galimberti:2018aa}. 
Choosing \eqref{eq:scl-testfunc-kl} as test function  
in \eqref{eq:scl-S(varrho)-approx}, recalling 
that $S(\varrho)=\varrho-\varrho^2$, and carrying on as 
before \eqref{eq:scl-rho-rho2-ell}, we deliver
{\small
\begin{equation}\label{eq:scl-rho-rho2-ell-eps-delta}
	\begin{split}
		& \EE \int_{\R^d\times\R} 
		\left( \varrho_{\eps,\delta}-\varrho_{\eps,\delta}^2\right) (t)
		\, \psi_\ell w_N\,d\xi \,dx
		\lesssim \EE\int_{\R^d\times\R}
		\left( \varrho_{0,\eps,\delta}-\varrho_{0,\eps,\delta}^2\right)(0)
		\, \psi_\ell w_N \,d\xi \,dx
		\\ & \quad
		+\int_0^t M(s)\left(\EE \int_{\R^d\times\R} 
		\left( \varrho_{\eps,\delta}-\varrho_{\eps,\delta}^2\right)(s)
		\psi_\ell w_N\,d\xi \,dx \right) \,ds
		\\ & \quad\quad
		+\EE \int_0^T\int_{\R^d}\int_{\R} 
		\left(\abs{r_{\eps,\delta}}
		+\frac{1}{\ell}\En_{\ell\le \abs{\xi}\le 2\ell}
		\abs{\tilde r_{\eps,\delta}} \right) w_N \,d\xi \,dx\,dt
		\\ & \quad\quad\quad
		+\EE \int_0^T\int_{\R^d}\int_{\R}
		\mathcal{Q}_{\eps,\delta} w_N\,d\xi \,dx\,dt+O(1/\ell),
	\end{split}
\end{equation}
}for some integrable function $M$ on $[0,T]$, where 
$\varrho_{0,\eps,\delta} := \varrho_0 
\star \left(J_\eps^x J^\xi_\delta\right)$. 
Provided we show that the "$\eps,\delta\to 0$ limit" of the 
$\mathcal{Q}_{\eps,\delta}$--term is zero, 
we obtain the rigidity inequality \eqref{eq:scl-contraction-rhomrho2} 
by sending $\eps,\delta\downarrow 0$ and $\ell\uparrow\infty$ 
in \eqref{eq:scl-rho-rho2-ell-eps-delta}, 
followed by an application of Gronwall's inequality. 

It remains to compute the limit of the 
$\mathcal{Q}_{\eps,\delta}$--term. 
Recalling that $B^2=\sum_{k\ge 1} b_k^2$, we write 
$\mathcal{Q}_{\eps,\delta}(\omega,t,x,\xi)
=\sum_{k\ge 1} \mathcal{Q}_{\eps,\delta,k}(\omega,t,x,\xi)$, 
where, for $k=1,2,\ldots$,
\begin{align*}		
	& \mathcal{Q}_{\eps,\delta,k}(\omega,t,x,\xi)
	:=\left(\left(b_k^2\pxi \varrho \right)
	\star \left(J^x_\eps J^\xi_\delta\right)\right)
	\pxi \varrho_{\eps,\delta}
	-\left( \left(b_k \partial_\xi \varrho\right)
	\star \left(J^x_\eps J^\xi_\delta\right)\right)^2
	\\ & \quad = 
	\int\!\!\!\!\!\!\iint\!\!\!\!\!\!\!\!\iint 
	\left(\left(b_k(\omega,t,y,\zeta)\right)^2
	-b_k(\omega,t,y,\zeta)b_k(\omega,t,\bar y,\bar\zeta)\right)
	\\ & \quad \qquad \qquad \qquad 
	\times (\pxi \varrho)(\omega,t,y,\zeta)
	(\pxi \varrho)(\omega,t,\bar y,\bar \zeta)
	\\ & \quad \qquad \qquad\qquad\qquad  
	\times J^x_\eps(x-y)J^x_\eps(x-\bar y)
	J^\xi_\delta(\xi-\zeta)J^\xi_\delta(\xi-\bar\zeta)
	\, d\zeta \,dy \, d\bar \zeta\, d\bar y.
\end{align*}
We can switch the roles of $y$ and $\bar y$ as well as 
$\zeta$ and $\bar \zeta$. Add the resulting 
expression for $\mathcal{Q}_{\eps,\delta,k}$ 
to the one above and divide by 2, obtaining
{\small
\begin{equation}\label{eq:scl-Qk-bk-square}
	\begin{split}
		\mathcal{Q}_{\eps,\delta,k}(\omega,t,x,\xi)
		& =  \frac12 \int\!\!\!\!\!\!\iint\!\!\!\!\!\!\!\!\iint
 		\abs{b_k(\omega,t,y,\zeta)-b_k(\omega,t,\bar y,\bar\zeta)}^2
 		(\pxi \varrho)(\omega,t,y,\zeta)
		(\pxi \varrho)(\omega,t,\bar y,\bar \zeta)
		\\ & \qquad \quad\quad \quad
		\times J^x_\eps(x-y)J^x_\eps(x-\bar y)
		J^\xi_\delta(\xi-\zeta)J^\xi_\delta(\xi-\bar\zeta)
		\, d\zeta\, dy \, d\bar \zeta\, d\bar y.
	\end{split}
\end{equation}
}
Summing over $k$, recalling \eqref{eq:scl-def-noise-B}, 
and using $\pxi \varrho=-\nu_{\omega,t,x}(d\xi)$ 
with $\nu(\R)=1$, the following estimate eventually materializes:
{\small
\begin{align*}
 	& \iint \mathcal{Q}_{\eps,\delta}(\omega,t,x,\xi)
 	\, w_N(x)\,\dxi\dx
 	\\ & \quad \lesssim 
 	\frac12 \iint\!\!\!\!\!\!\iint
 	\!\!\!\!\!\!\iint
 	\Bigl(\abs{y-\bar y}^2+
 	\abs{\zeta-\bar\zeta}\mu\left(\abs{\zeta-\bar\zeta}\right)\Bigr)
 	(\pxi \varrho)(\omega,t,y,\zeta)
	(\pxi \varrho)(\omega,t,\bar y,\bar \zeta)
	\\ &  \qquad\qquad \qquad\qquad
	\times J^x_\eps(x-y)J^x_\eps(x-\bar y)
	J^\xi_\delta(\xi-\zeta)J^\xi_\delta(\xi-\bar\zeta)
	\, w_N(x) \, d\zeta \, dy \, d\bar \zeta \, d\bar y\,\dxi\, \dx
	\\ & \quad\quad \lesssim_N \left(\eps + \mu(\delta)\right)
	\overset{\eps,\delta\downarrow 0}{\longrightarrow} 0.
\end{align*}
}
This concludes the proof. 
\end{proof}

\begin{remark} 
Regarding the "weight-free" $L^p$--framework discussed 
in Remark, the proof of Proposition \ref{prop:scl-rigidity} remains 
the same except for a few changes involving 
the terms $I_1(\varphi_{\kappa,\ell})$ 
and $I_2(\varphi_{\kappa,\ell})$ to account 
for the weight-free test function 
$\varphi_{\kappa,\ell}(x,\xi)=\phi_\kappa(x)\psi_\ell(\xi)$
and the modified assumptions \eqref{eq:scl-init-ass-new}, 
\eqref{eq:scl-def-noise-B-new}, and \eqref{eq:scl-flux-Lipschitz1-new}.
\end{remark}

The next theorem contains the main result of this section, namely 
the existence, uniqueness, and $L^1$ stability of kinetic solutions.

\begin{theorem}[well-posedness]\label{thm:scl-L1stab}
Suppose that $b_k,B^2,\overline{a}=\set{a,d},R$ satisfy conditions 
\eqref{eq:scl-def-noise-B}, \eqref{eq:scl-noise-bk-reg}, 
\eqref{eq:scl-a-growth}, \eqref{eq:scl-d-growth}, 
\eqref{eq:scl-R-growth}, \eqref{eq:scl-a-d-R-Sobolev} 
and $\Div_{(x,\xi)} \overline{a}=0$. There exists 
a unique kinetic solution of \eqref{eq:scl} with 
initial data $u_0$ satisfying \eqref{eq:scl-init-ass}. 
If $u_1, u_2$ are two kinetic solutions of \eqref{eq:scl} with 
initial data $u_{1,0}, u_{2,0}$, respectively, then 
\begin{equation}\label{eq:scl-L1-contraction-u1u2}
	\EE \int_{\R^d} \abs{u_1(t,x)-u_2(t,x)}\, w_N \dx
	\lesssim_{T,N} 
	\EE \int_{\R^d} \abs{u_{1,0}(x)-u_{2,0}(x)}\, w_N \dx,
\end{equation} 
for all $t\in [0,T]$, where $w_N$ is defined in \eqref{eq:scl-weight-def}. 
Besides, the unique kinetic solution $u$ of \eqref{eq:scl} 
has a representative in the space $L^p(\Omega;L^\infty(0,T;L^p(w_Ndx)))$ 
which a.s.~exhibits continuous samples paths in $L^p(w_Ndx)$, 
for all $p\in [1,\infty)$.
\end{theorem}

\begin{proof}
As in \cite{Galimberti:2018aa,Gess:2019aa}, we point out 
that the $L^1$ contraction principle \eqref{eq:scl-L1-contraction-u1u2} is 
a simple consequence of Proposition \ref{prop:scl-rigidity}. 
Indeed, define $\overline\varrho=\frac12 \left(\En_{\xi<u_1}
+\En_{\xi<u_2}\right)=:\frac12 \left(\rho_1
+\rho_2\right)$ and also $\overline\varrho_0=\frac12 \left(\En_{\xi<u_{1,0}}
+\En_{\xi<u_{2,0}}\right)=:\frac12 \left(\rho_{0,1}
+\rho_{0,2}\right)$. Note that $\overline\varrho$ 
is a generalized kinetic solution with initial 
data $\overline\varrho_0$, kinetic measure $\overline m=\frac12 (m_1+m_2)$, 
and $\partial_\xi \overline\varrho =-\frac12 
\left(\delta_{u_1}+\delta_{u_2}\right)=:-\overline\nu$. 
Clearly, $\overline m(\set{0}\times\R^d\times\R)=0$ (since $m_1,m_2$ both 
vanish at $t=0$ because of the kinetic initial data)  and thus 
$\varrho(0)=\varrho_0$, cf.~Remark \ref{rem:init-measure}. 
By Proposition \ref{prop:scl-rigidity},
\begin{equation*}
	\EE \iint_{\R^d\times\R} 
	\left( \overline\varrho-\overline\varrho^2\right)(t) \, w_N \dxi\dx
	\lesssim_{T,N} \EE\iint_{\R^d\times\R}
	\left(\overline\varrho_0-\overline\varrho_0^2\right)
	w_N \dxi \dx,
\end{equation*}
for a.e.~$t\in [0,T]$. A simple computation, exploiting
the identities $\rho_i^2=\rho_i$ ($i=1,2$), will reveal that 
$\overline\varrho-\overline\varrho^2=
\frac14 \left(\rho_1-\rho_2\right)^2=\frac14 \abs{\rho_1-\rho_2}$ and 
so $\int_{\R} \left(\overline\varrho-\overline\varrho^2\right)\dxi
=\frac14\abs{u_1-u_2}$. In the same way, 
we have $\int_{\R}\left(\overline\rho_0-\overline\rho_0^2\right)\dxi
=\frac14\abs{u_{1,0}-u_{2,0}}$. 
Consequently, \eqref{eq:scl-L1-contraction-u1u2} holds.

The sample paths of a kinetic 
solution $u$ are a.s.~continuous as a result 
of the uniqueness result. The detailed proof
is the same as in \cite[Corollary 16]{DV} (see also \cite{Dotti:2018aa}). 
Thanks to the continuity of the sample 
paths, the contraction inequality \eqref{eq:scl-L1-contraction-u1u2} 
holds for all $t\in[0,T]$.

The existence part of the theorem can be 
be founded on the vanishing viscosity method 
\cite{BVW1,CDK,DV,FN,KSt}, or operator splitting 
\cite{Bauzet:2015aa,Karlsen:2016aa} to separate 
the deterministic and stochastic effects in \eqref{eq:scl}. 
Existence results on $\R^d$ are provided in 
these references under the assumptions that $R\equiv 0$ and 
$A=A(u)$ does not depend on $t,x$. The techniques 
employed in \cite{Bauzet:2015aa,BVW1,CDK,DV,FN,Karlsen:2016aa,KSt} 
can be adapted to the general context provided by \eqref{eq:scl}. Here we only give a sketch of the proof via the vanishing viscosity method, based on \cite{DV}.

Given $\ve>0$ and consider the following parabolic SPDE
\begin{equation}\label{eq:scl^ve}
	\begin{split}
		& \pt u^\ve + \Div_x A(t,x,u^\ve) - \ve \Delta_x u^\ve
		=B(t,u^\ve)\dotW(t) + R(t,x,u^\ve),
		\hspace{2mm} (t,x)\in (0,T)\times \R^d,\\
		& u^\ve(0,x)=u_0(\omega,x), \hspace{5mm} x\in \R^d.
	\end{split}	
\end{equation}

It is not difficult to show that equation \eqref{eq:scl^ve} is well-posed. 
Indeed, the unique weak solution belonging to the weighted space 
$L^2(\Omega;(C([0,T]);L^2(\omega_N dx)))\cap L^2(\Omega\times[0,T];$ $H^1(\omega_N dx))$ 
can be found as a fixed point of the operator
\begin{align*}
Kv(t) & :=S(t)u_0 + \int_0^tS(t-s)\Big( R(s,\cdot,v(s)) - \Div_x A(s,\cdot,v(s))\Big)\, ds 
\\ & \qquad \qquad \qquad
+ \int_0^t S(t-s)B(s,v(s))\, dW(s),
\end{align*}
where $S(t)$ is the semigroup generated 
by the heat equation in $\mathbb{R}^d$.

Let $u^\ve$ be the weak solution of \eqref{eq:scl^ve}. 
Then, for $S \in C^2(\mathbb{R})$, by It\^{o} formula we have that the 
following equation is a.s. satisfied in the sense of distributions:
{\small
\begin{equation}\label{eq:scl-entropy-ineq-ve}
	\begin{split}
		&\pt S(u^\ve) + \Div_x Q_S(t,x,u^\ve) 
		+ S'(u^\ve)\left((\Div_x A)(t,x,u^\ve)-R(t,x,u^\ve)\right) 
		- (\Div_x Q_S)(t,x,u^\ve)
		\\ & \quad 
		= -\ve S''(u^\ve)|\nabla u^\ve|^2 + \ve\Delta_x \S(u^\ve) 
		+ \sum_{k\ge1} S'(u^\ve) b_k(t,x,u^\ve)\, \dotW_k(t)
		+ \frac12 S''(u^\ve) B^2(t,x,u^\ve).
	\end{split}
\end{equation}
}
where $Q_S:[0,T]\times\mathbb{R}\times\mathbb{R}\to \mathbb{R}^d$ 
is given by $(\partial_u Q_S)(t,x,u)=S'(u)(\partial_u A)(t,x,u)$. 

Let $S(\xi)=|\xi|^p$, $p\geq 2$. Then, similarly as in Remark \ref{rem:drop-the-weight}, taking conveniently 
chosen test functions, after some manipulation it follows that
\begin{equation}\label{eq:4.44}
\mathbb{E}\left(\sup_{t\in[0,T]}\| u^\ve(t) \|_{L^p(\omega_Ndx)}^p\right) 
+ \ve\int_0^T\int_{\mathbb{R}^d}|u^\ve(t,x)|^{p-2}|\nabla u^\ve|^2 \omega_N(x) dx\, dt\leq C,
\end{equation}
where $C=C(p,u_0,T)$ is independent of $\ve$.

Moreover, $u^\ve$ is a kinetic solution of equation \eqref{eq:scl^ve}, 
in the sense that the function $\varrho^\ve(t,x,\xi) := \En_{\xi<u^\ve(t,x)}$ satisfies the SPDE
\begin{equation}\label{eq:scl-kinetic-eqn-ve}
	\begin{split}
		\pt \varrho^\ve &+\Div_{(x,\xi)}
		\Bigl(\overline{a}\varrho^\ve\Bigr)
		+R \partial_\xi \varrho^\ve - \ve \Delta_x \varrho^\ve
		\\ & 
		+\sum_{k\ge 1}b_k \partial_\xi \varrho^\ve\, \dotW_k(t)
		=  \pxi \left(\frac{B^2}{2}
		\pxi \varrho^\ve \right)+\partial_\xi m^\ve
		\quad \text{in $\Dp([0,T)\times\R^d\times\R)$, a.s.},
	\end{split}
\end{equation}
where $m^\ve = \ve |\nabla_x u^\ve|^2\delta_{\xi=u^\ve}$, 
with initial data $\varrho^\ve(0,x,\xi)=\rho_0(x,\xi)
:=\En_{\xi<u_0(x)}$.

Let us denote $\nu^\ve_{t,x}=-\partial_\xi\varrho^\ve(t,x,\xi)=\delta_{\xi=u^\ve(t,x)}$. 
Then, $\nu^\ve$ is a Young measure and 
by \eqref{eq:4.44} we have, in particular, that
\begin{equation}\label{eq:4.46}
\mathbb{E}\int_0^T\int_{\mathbb{R}^d}
\int_\mathbb{R} \abs{\xi}^p 
\, d\nu_{t,x}^\ve(\xi)
\, \omega_N dx\, dt\leq C_p,
\end{equation}
for any $p\geq 0$, uniformly in $\ve$. 
Likewise, \eqref{eq:4.44} also implies that
$$
\mathbb{E}
\int_{[0,T]\times\mathbb{R}^d\times\mathbb{R}}
\abs{\xi}^p \, dm_N^\ve(\xi,t,x)\leq C_p,
$$
uniformly in $\ve$, where $m_N^\ve=\omega_N m^\ve$. 
This last estimate can be improved to the following
\begin{equation}\label{eq:4.49}
\mathbb{E}
\abs{\int_{[0,T]\times\mathbb{R}^d\times\mathbb{R}}
\abs{\xi}^{2p}\, dm_N^\ve(\xi,t,x)}^2
\leq C_p, \qquad p\geq 2.
\end{equation}
Proceeding similarly as in 
Remark \ref{remark:improv-integrability}, 
it suffices to take convenient test functions 
(in connection with the weight $\omega_N$) 
in \eqref{eq:scl-entropy-ineq-ve} with 
$S(\xi)=\abs{\xi}^{2p+2}$, 
squaring the resulting equation 
and taking expectation. Indeed, note that 
\begin{align*}
& \mathbb{E}
\abs{\int_{[0,T]\times\mathbb{R}^d\times\mathbb{R}}
\abs{\xi}^{2p} 
\, dm_N^\ve(\xi,t,x)}^2
\\ & \qquad 
= \frac{1}{(p+2)(p+1)}
\mathbb{E}\abs{\int_0^T 
\int_{\mathbb{R}^d} \ve S''(u^\ve) 
\abs{\nabla_x u^\ve}^2 
\,\omega_N dx \, dt}^2.
\end{align*}
With some manipulation involving 
the It\^{o} isometry and using \eqref{eq:4.44} 
all the other terms can be bounded appropriately 
so that \eqref{eq:4.49} 
follows. We omit the details.

Now, by the theory of Young measures and kinetic functions 
(see e.g.~Theorem 5 and Corollary 6 in \cite{DV}) 
\eqref{eq:4.46} guarantees the existence of 
a sequence $\{ \ve_n\}_n$, a young measure $\nu$ 
and a generalized kinetic function $\varrho:\Omega\times[0,T]\times\mathbb{R}^d\times\mathbb{R}\to [0,1]$ 
such that $\ve_n\to 0$, $\nu^{\ve_n}\to \nu$ 
in the sense of Young measures and $\varrho^{\ve_n}\rightharpoonup 
\varrho$ weakly-$*$ in $L^\infty(\Omega\times[0,T]\times\mathbb{R}^d\times\mathbb{R})$ as $n\to \infty$. 
Moreover, denoting by $\mathcal{M}_b$ the 
space of the bounded Borel Measures on 
$[0,T]\times\mathbb{R}^d\times \mathbb{R}$, 
by \eqref{eq:4.49} there is a kinetic measure $m_N$ 
such that, up to a subsequence, $m_N^{\ve_n}\rightharpoonup m_N$ 
weakly-$*$ in $L^2(\Omega;\mathcal{M}_b)$, as $n\to \infty$. 
Defining $m:=\frac{1}{\omega_N}m_N$, then $m$ 
turns out to be a kinetic measure in the 
sense of Definition \ref{def:kinetic-measure} 
and we may pass to the limit as $\ve=\ve^n \to 0$ 
in equation \eqref{eq:scl-kinetic-eqn-ve} in order 
to conclude that $\varrho$ is a generalized 
kinetic solution of equation \eqref{eq:scl}. 
At this point, the rigidity result implies 
that $\rho=\En_{\xi<u}$ where $u$ 
is a kinetic solution.
\end{proof}

\begin{remark}[strong convergence of the 
parabolic approximations]
Let $\varrho$ and $\varrho^\ve$ be as in 
the proof of Theorem \ref{thm:scl-L1stab}. 
Taking advantage of the particular structure 
of $\varrho^{\ve_n}$ and $\varrho$ we have that
\begin{multline}\label{eq:strong-convergence}
\norm{u^{\ve_n}}_{L^2(\Omega\times[0,T];L^2(\omega_N dx))}^2
-\norm{u}_{L^2(\Omega\times[0,T];L^2(\omega_N dx))}^2
\\ = \int_{[0,T]\times\mathbb{R}^d \times\mathbb{R}}
2\xi(\varrho-\varrho^{\ve_n})\, d\xi
\, \omega_N dx\, dt.
\end{multline}
By Chebyshev's inequality and using \eqref{eq:4.44} 
with $p=3$, for any $R>0$ we have
\begin{align*}
\mathbb{E}\int_0^T\int_{\mathbb{R}^d} 
\int_{\abs{\xi}>R} 
\abs{2\xi (\varrho-\varrho^\ve)} 
\, d\xi \, \omega_N dx\, dt\leq \frac{C}{R}.
\end{align*}

Thus, taking expectation in \eqref{eq:strong-convergence}, 
we may pass to the limit as $\ve_n\to 0$ in order to conclude that
$$
\norm{u^{\ve_n}-u}_{L^2(\Omega\times[0,T];L^2(\omega_N dx))}
\to 0, \qquad\text{as $n\to \infty$}.
$$
In fact, by uniqueness, the whole sequence 
$u^\ve$ converges strongly to the kinetic solution.

Finally, in light of estimate \eqref{eq:4.44}, 
by H\"{o}lder inequality we also deduce that
$$
\norm{u^{\ve}-u}_{L^p(\Omega\times[0,T];L^p(\omega_N dx))}
\to 0, \qquad\text{as $\ve\to 0$},
$$
for any $p\geq 1$.
\end{remark}

\begin{remark}[1/2--H\"older continuous noise coefficient]
Referring to \eqref{eq:scl-infinite-Wiener}, consider 
the simple noise term $b(u)\dW(t)$, 
where $W(t)$ is a one-dimensional Wiener process 
and $b(u)$ is a scalar function. Typical noise functions covered by
the regularity condition \eqref{eq:scl-def-noise-B} include 
$b(u)=\abs{u}^\gamma$, $\gamma>\tfrac12$, which is H\"older continuous 
with exponent $\gamma>\tfrac12$. Condition \eqref{eq:scl-def-noise-B} is the same 
as the one imposed in the existing literature (see e.g.~\cite{DV}). 
Unfortunately, it does not allow for the interesting example 
$b(u)=\sqrt{\abs{u}}$, or any function $b$ that satisfies 
$\abs{b(u)-b(v)}\lesssim \mu(\abs{u-v})$, where 
\begin{equation}\label{eq:scl-int-1-over-mu}
	\int_0^1 \frac{1}{\left(\mu(\xi)\right)^2}
	\dxi = \infty.
\end{equation}
Condition \eqref{eq:scl-int-1-over-mu} embraces  
$\tfrac12$--H\"older continuous noise functions $b$, like 
$b(u)=\sqrt{\abs{u}}$.

Returning to the general case \eqref{eq:scl-infinite-Wiener}, assuming 
$b_k=b_k(\xi)$ $\forall k$, we claim that Proposition \ref{prop:scl-rigidity} (and 
Theorem \ref{thm:scl-L1stab}) actually holds 
with \eqref{eq:scl-noise-bk-reg} replaced by
\begin{equation}\label{eq:scl-noise-bk-reg-new2}
	\sum_{k\geq 1}\abs{b_k(u)-b_k(v)}^2
	\lesssim \left(\mu(\abs{u-v})\right)^2,
\end{equation}
for some continuous nondecreasing function $\mu$ on $\R_+$ 
satisfying $\mu(0+)=0$ and \eqref{eq:scl-int-1-over-mu}. 
To allow for \eqref{eq:scl-noise-bk-reg-new2}, we will make a more 
careful choice of the approximate delta function $J^\xi_\delta$ 
in order to handle to the key error term
\eqref{eq:scl-Qk-bk-square}. Inspired by the work \cite{Yamada:1971aa} 
of Yamada and Watanabe on stochastic differential equations, we pick 
a strictly decreasing sequence $\set{a_n}_{n=0}^\infty$ of 
positive numbers, $a_n\downarrow 0$, recursively 
defined by $a_0=1$ and for $n=1,2,\ldots$ by 
$\int_{a_n}^{a_{n-1}}\frac{1}{\left(\mu(\xi)\right)^2}\dxi = n$. 
For example, with $\mu(\xi)=\sqrt{\xi}$ for $\xi>0$, $a_n = a_{n-1} e^{-n}$; 
hence $a_n=e^{-\frac{1}{2}n(n+1)}$. 
Next, pick positive $C^\infty_c$ functions $\psi_n$ on $\R_+$ with 
$\supp \psi\subset (a_n,a_{n-1})$ and 
\begin{equation}\label{eq:scl-psin-YW}
	\begin{split}
		&0 \le \psi_n(\xi) 
		\le \frac{2}{n\left(\mu(\xi)\right)^2}
		\le \frac{2}{n \xi}, \quad 
		\text{for any $\xi\in \R$},
		\qquad 
		\int_{a_n}^{a_{n-1}} \psi_n(\xi)\, d\xi =1. 
	\end{split}
\end{equation}
We introduce the function $\Psi_n(\xi):=\int_0^{\abs{\xi}} 
\int_0^{\bar\kappa} \psi_n(\kappa)\,d\kappa \, d \bar\kappa$ for 
$\xi \in \R$, which is a symmetric approximation of $\abs{\xi}$.
Since $\psi_n$ (and thus $\Psi_n$) is zero in a 
neighborhood of the origin, we have $\Psi_n\in C^\infty(\R)$ 
and $\Psi_n''(\xi)=\psi_n(\abs{\xi})\le \frac{2}{n\abs{\xi}}$.
Moreover, $\Psi_n(\cdot) \to \abs{\cdot}$ uniformly on $\R$.  

Let us now return to \eqref{eq:scl-rho-rho2-ell-eps-delta} and 
the error term \eqref{eq:scl-Qk-bk-square}, 
replacing  $J^\xi_\delta(\cdot)$ by $\psi_n(\abs{\cdot})$ 
($=\Psi_n''(\cdot)$) and, at the same time, 
renaming $\delta$ by $n$. 
Note that $\sum_{k\ge 1}\abs{b_k(\zeta)-b_k(\bar\zeta)}^2$
is bounded by a constant times $ \left(\mu(\abs{\xi-\zeta})\right)^2
+\left(\mu(\abs{\xi-\bar\zeta})\right)^2$, and 
thus, cf.~\eqref{eq:scl-psin-YW},
\begin{align*}
	\sum_{k\ge 1}
	\abs{b_k(\zeta)-b_k(\bar\zeta)}^2
	\psi_n\left(\abs{\xi-\zeta}\right)
	\psi_n\left(\abs{\xi-\bar\zeta}\right)
	\lesssim \frac{1}{n} \left(\psi_n\left(\abs{\xi-\bar\zeta}\right)
	+\psi_n\left(\abs{\xi-\zeta}\right)\right).
\end{align*}
As a result, 
\begin{align*}
 	& \iint\mathcal{Q}_{\eps,\delta}(\omega,t,x,\xi)
 	\, w_N(x)\,\dxi\dx
 	\\ & \qquad \lesssim 
 	\frac{1}{n}\iint\!\!\!\!\! \!\iint\!\!\!\!\!\!\!\!
 	\iint\!\!\!\!\!\!\!\!\iint 
	\Bigl( \psi_n\left(\abs{\xi-\bar\zeta}\right)
 	+\psi_n\left(\abs{\xi-\zeta}\right)\Bigr)
 	\abs{(\pxi \varrho)(\omega,t,y,\zeta)}
	\abs{(\pxi \varrho)(\omega,t,\bar y,\bar \zeta)}
	\\ & \qquad \qquad\qquad \qquad\qquad
	\times J^x_\eps(x-y)J^x_\eps(x-\bar y)\, w_N(x)
	\, d\zeta\,dy\, d\bar \zeta\, d\bar y \, \dxi \,\dx
	\\ & \qquad \lesssim 
 	\frac{1}{n}\int
 	\left(\iint\abs{(\pxi \varrho)(\omega,t,y,\zeta)}J^x_\eps(x-y)
 	\, d\zeta\, dy\right)	
 	\\ & \qquad\qquad\quad 
	\times
	\left(\iint\abs{(\pxi \varrho)(\omega,t,\bar y,\bar \zeta)}
	J^x_\eps(x-\bar y)\, d\bar \zeta \, d\bar y\right)
	\, w_N(x) \dx \lesssim_N \frac{1}{n}
	\overset{n\uparrow \infty}{\longrightarrow} 0,
\end{align*}
where we have used $\pxi \varrho =-\nu$ with $\nu(\R)=1$. 
Therefore, sending $n\to \infty$, $\eps\to 0$, and then 
$\ell\to \infty$ in \eqref{eq:scl-rho-rho2-ell-eps-delta}, we 
obtain \eqref{eq:scl-contraction-rhomrho2}.
\end{remark}

\section{Comparison principle \& 
stochastic Kru{\v{z}}kov inequality}\label{S:5}

In a standard way, one can use Theorem \ref{thm:scl-L1stab} 
to deduce a comparison result. Indeed, 
\begin{equation}\label{L1-contraction-pos}
	\EE \int_{\R^d} \left(u_1(t)-u_2(t)\right)_+ w_N\,dx
	\lesssim \EE \int_{\R^d} 
	\left(u_{1,0}-u_{2,0}\right)_+\, w_N(x)\dx,
\end{equation} 
which follows from \eqref{eq:scl-L1-contraction-u1u2}
and the identity $2(a-b)_+=\abs{a-b} + (a-b)$ for all $a,b\in \R$. 
As a result, $u_{0,1}\le u_{0,2}$ implies $u_1\le u_2$.  

One can also establish \eqref{L1-contraction-pos} directly, 
following the proof of Proposition \ref{prop:scl-rigidity} 
step-by-step, modulo one change. 
The proof of Proposition \ref{prop:scl-rigidity} makes use 
of the It\^{o} chain rule to compute the equation for 
$\varrho-\varrho^2=\varrho(1-\varrho)$. 
To establish \eqref{L1-contraction-pos}, we use instead the 
It\^{o} product formula to deduce that (formally) the functions
$\rho_1=\En_{\xi<u_1}$ and $\rho_2=\En_{\xi<u_2}$ 
satisfy the inequality
\begin{equation}\label{eq:scl-varrho_1(1-varrho_2)}
	\begin{split}
		&\pt\Bigl( \varrho_1(1-\varrho_2)\Bigr) 
		+\Div_{(x,\xi)}\Bigl(\overline{a}\, 
		\varrho_1(1-\varrho_2)\Bigr)
		+R \partial_\xi \Bigl( \varrho_1(1-\varrho_2)\Bigr)
		\\ &\qquad +\sum_{k\ge 1}b_k 
		\partial_\xi \Bigl( \varrho_1(1-\varrho_2)\Bigr)\, \dotW_k(t)
		\\ & \qquad\quad 
		\le \pxi \left(\frac{B^2}{2}
		\pxi \Bigl( \varrho_1(1-\varrho_2)\Bigr) \right)
		+\partial_\xi\Bigl( (1-\varrho_2) m_1-\varrho_1 m_2\Bigr),
	\end{split}
\end{equation}
where $u_1,u_2$ are two kinetic solutions with 
corresponding kinetic measures $m_1$ and $m_2$. 
Of course, the rigorous proof goes through 
a regularization step that 
justifies the application of the It\^{o} product formula.

More generally, we can derive a 
stochastic Kru{\v {z}}kov inequality inequality, that 
may be considered as a comparison inequality 
which is satisfied a.s.. Particular cases of 
this inequality have been proven to be extremely 
useful in Sections \ref{S:2} and \ref{S:3}.

\begin{proposition}[stochastic 
Kru{\v {z}}kov inequality]\label{p:Kato-Kruzkov}
Let $u_1$ and $u_2$ be two kinetic solutions 
of \eqref{eq:scl} with initial data $u_{1,0}$ 
and $u_{2,0}$, respectively. Suppose 
$\Div_x A=0$. Then, almost surely,
\begin{multline}\label{eq:scl-kato} 
 \int_0^\infty\int_{\R^d}
\biggl\{ \abs{u_1-u_2}\phi_t
+\sgn(u_1-u_2)\left(A(t,x,u_1)-A(t,x,u_2)\right) 
\cdot \nabla_x \phi \biggr. \\
\biggl.  
+ \sgn(u_1-u_2)\left(R(t,x,u_1)-R(t,x,u_2)\right)
\phi\biggr\}\,dx\,dt \\
+\sum_{k\geq 1}\int_0^\infty \int_{\R^d} 
\sgn\left(u_1-u_2\right)
\left(b_k(t,x,u_1)-b_k(t,x,u_2)\right)\phi\,dx\,dW_k(t) 
\\ +
\int_{\R^d} \abs{u_{1,0}-u_{2,0}}\phi(0,x)\,dx\ge 0,
\end{multline}
for any $\varphi \in C_0^\infty(\R\times\R^d)$ 
with $\varphi \geq 0$.
\end{proposition}

Note that, formally, this inequality results by integrating inequality \eqref{eq:scl-varrho_1(1-varrho_2)}. Below, we present a straightforward proof using the fact that the unique solutions are obtained through the vanishing viscosity method.

\begin{proof}
Following the proof of Theorem \ref{thm:scl-L1stab} 
we have that $u_j$, $j=1,2$, may be found as a 
limit in $L^p(\Omega\times[0,T]\times\mathbb{R}^d)$ 
when $\ve\to 0$ of a sequence $\{u_j^\ve\}_{\ve>0}$ of 
weak solutions to the parabolic SPDEs 
\begin{equation*}
	\begin{split}
		& \pt u_j^\ve + \Div_x A(t,x,u_j^\ve) 
		- \ve \Delta_x u_j^\ve
		=B(t,u_j^\ve)\dotW(t) + R(t,x,u_j^\ve),
		\hspace{2mm} (t,x)\in (0,T)\times \R^d,\\
		& u_j^\ve(0,x)=u_{0,j}(\omega,x), 
		\hspace{5mm} x\in \R^d.
	\end{split}	
\end{equation*}
For fixed $\ve>0$, we have that $(u_1-u_2)$ is a 
weak solution of the following equation
\begin{equation*}
	\begin{split}
		& \pt (u_1-u_2)^\ve 
		+ \Div_x \left(A(t,x,u_1^\ve)-A(t,x,u_2^\ve)\right) 
		- \ve \Delta_x (u_1^\ve-u_2^\ve)\\
		&\qquad 
		= \left(B(t,u_1^\ve)-B(t,u_2^\ve)\right)\dotW(t) 
		+ R(t,x,u_1^\ve)-R(t,x,u_2^\ve),
		\hspace{2mm} (t,x)\in (0,T)\times \R^d,\\
		& (u_1^\ve-u_2^\ve)(0,x)=(u_{0,1}-u_{0,2})(\omega,x), 
		\hspace{5mm} x\in \R^d.
	\end{split}	
\end{equation*}

Let $S_\theta(\xi)$ be a $C^2$ convex 
approximation of $|\xi|$, such that $S_\theta'(\xi)$ is 
monotone nondecreasing, $S_\theta'(\xi)=1$, for $\xi>\d$, 
and $S_\theta'(\xi)=-1$, for $\xi\le -\d$. 
Then, for any nonnegative test function $\varphi(t,x)$, after sending 
$\theta \to 0$, by It\^{o} formula we have a.s. that
\begin{multline}\label{|u_2-u_2|^ve-ito}
\int_0^\infty\int_{\mathbb{R}^d} \abs{u_1^\ve-u_2^\ve}
\varphi_t \, dx\, dt 
+ \int_0^\infty\int_{\mathbb{R}^d}\sgn(u_1^\ve-u_2^\ve)
\left(A(t,x,u_1^\ve)-A(t,x,u_2^\ve)\right)
\cdot\nabla_x\varphi \,dx\, dt\\
\qquad -\ve\int_0^\infty \int_{\mathbb{R}^d}
\sgn(u_1^\ve-u_2^\ve)\nabla(u_1^\ve-u_2^\ve)
\cdot\nabla\varphi dx\, dt  \\
\qquad\qquad + \int_0^\infty\int_{\mathbb{R}^d}
\sgn(u_1^\ve-u_2^\ve)
\left(R(t,x,u_1^\ve)-R(t,x,u_2^\ve)\right)
\varphi \, dx\, dt\\
\qquad\qquad\qquad 
+\sum_{k\geq 1}\int_0^\infty\int_{\mathbb{R}^d}
\sgn(u_1^\ve-u_2^\ve)
\left(b_k(t,x,u_1^\ve)-b_k(t,x,u_2^\ve)\right)
\varphi \, dx \, dW_k(t) \\
\qquad\qquad\qquad\qquad 
+ \int_{\mathbb{R}^d} \abs{u_{0,1}-u_{0,2}}
\varphi(0,x)\, dx  \geq 0, 
\end{multline}
where the convergence in  the stochastic integral is enabled by \eqref{eq:scl-noise-bk-reg}.

Recall that both $u_1, u_2$ satisfy estimate \eqref{eq:4.44}, uniformly in $\ve$. 
Thus, as convergence in mean square implies convergence in probability, which, in turn, implies 
a.s.~convergence along a subsequence, we know that the 
third term on the left-hand side of \eqref{|u_2-u_2|^ve-ito} 
converges to zero a.s. along a subsequence $\ve_n\to 0$. 
By the same token, passing to a further subsequence as the 
case may be, taking the limit as $\ve_n\to 0$ in \eqref{|u_2-u_2|^ve-ito}, 
we obtain \eqref{eq:scl-kato}.
\end{proof}

\end{document}